\documentclass{amsart}
\usepackage{epsfig,amsmath,amsthm,amssymb,amsfonts,color, graphicx, verbatim, bm}

\usepackage[all]{xy}
\xyoption{poly}
\xyoption{arc}
  \usepackage{geometry}
\usepackage[colorinlistoftodos,textsize=footnotesize]{todonotes}
\usepackage{soul}
\soulregister\cite7
\soulregister\ref7
\soulregister\eqref7
\usepackage[colorinlistoftodos,textsize=footnotesize]{todonotes}
\newcommand{\hlfix}[2]{\texthl{#1}\todo{#2}}

\numberwithin{equation}{section}

\newcommand{\Q}{{\mathbb{Q}}}
\newcommand{\Z}{{\mathbb{Z}}}
\newcommand{\R}{{\mathbb{R}}}
\newcommand{\C}{{\mathbb{C}}}
\newcommand{\bP}{\mathbb{P}_k}

\newcommand{\Asm}{\sA_{1L}^\times}
\def \AoneL#1#2{(\Asm)_{#2}^{#1}}
\newcommand{\BsG}{B(\sG_{1L})}
\newcommand{\BsTsG}{B(\sTt,\sG_{1L})}
\def \sToneL#1#2{{\sT_{1L}}^{#1}_{#2}}
\def \sToneLt#1#2{(\sTt)^{#1}_{#2}}
\def \sTt{\sT_{1L}^{\rm{twist}}}
\def \cZ#1#2{\mathcal{Z}^{#1}(\square^{#2})}
\def \cZoneL#1#2{\mathcal{Z}_{1L}^{#1}(\square^{#2})}
\def \cN#1#2{\mathcal{Z}^{#1}(\Spec k, #2)}
\def \cNoneL#1#2{{\mathcal{Z}_{1L}}^{#1}(\Spec k, #2)}
\newcommand{\QZ}{{\Q\sZ_{1L}}}
\def\sGoneL#1#2{{\sG_{1L}}_{#2}^{#1}}
\newcommand{\Gprg}{\Q[\sG]}
\newcommand{\fS}{\mathfrak{S}}
\newcommand{\sZ}{{\mathcal Z}}
\newcommand{\sI}{{\mathcal I}}
\newcommand{\sT}{{\mathcal T}}
\newcommand{\sH}{{\mathcal H}}
\newcommand{\sC}{{\mathcal C}}
\newcommand{\sG}{{\mathcal G}}
\newcommand{\sA}{{\mathcal A}}

\newcommand{\sO}{{\mathcal O}}
\newcommand{\G}{{\mathbb{G}}}

\newcommand{\sM}{{\mathcal M}}
\newcommand{\sP}{{\mathcal P}}
\newcommand{\id}{\textrm{id}}
\def\onem#1#2{1-\frac{#1}{#2}}

\newcommand{\sha}{{\,\amalg\hskip -3.6pt\amalg\,}}
\newcommand{\Li} {{\mathbb L}{\rm i}}
\newcommand{\td}{\textrm{tot deg}}
\newcommand{\Alt}{{\rm Alt \;}}
\newcommand{\Hom}{{\rm Hom}}
\newcommand{\sgn}{\text{sgn}}
\newcommand{\rk}{\textrm{rk }}
\newcommand{\bdeg}{\deg_B}
\newcommand{\Spec}{{\rm Spec \,}}
\newcommand{\D}{\partial}

\newcommand{\simord}{\sim_{\textrm{ord}}}
\newcommand{\simv}{\sim_{\textrm{v}}}
\newcommand{\simori}{\sim_{\textrm{ori}}}
\newcommand{\epclass}{[\varepsilon^n(a_0, \bm{a_n})]}

\def\ba #1\ea{\begin{align} #1 \end{align}}
\def\bas #1\eas{\begin{align*} #1 \end{align*}}
\def\bml #1\eml{\begin{multline} #1 \end{multline}}
\def\bmls #1\emls{\begin{multline*} #1 \end{multline*}}
\theoremstyle{plain}
\newtheorem{thm}{Theorem}[section] % reset theorem numbering for each chapter

\theoremstyle{definition}
\newtheorem{eg}[thm]{Example} % same for example numbers
\newtheorem{rem}[thm]{Remark} %and remarks
\newtheorem{lem}[thm]{Lemma} %ditto for lemmas
\newtheorem{remark}[thm]{Remark} %Notice that remarks use the same numbers AND THERE ARE 2 TYPES OF REMARK
\newtheorem{prop}[thm]{Proposition} %Let propostitions use the same numbers
\newtheorem{dfn}[thm]{Definition} % definition numbers are dependent on theorem numbers
\newtheorem{cor}[thm]{Corollary} %Therefore the correlaries are too

\def \Go#1{{\begin{xy}
(0, 0) *{\bullet}= "a",
"a"; "a" **\crv{+(-5,3)&+(5,7)&+(5,-7)}?/0pt/*{\dir{<}} +(0,3)*{^{#1}}
\end{xy}}}
\begin{document}

\title{Rational mixed Tate motivic graphs}
%\author{Susama Agarwala\footnote{University of Hamburg/Oxford University} \quad Owen Patashnick\footnote{Heilbronn Institute of Mathematical Research, Bristol University}{$\>$}\footnote{research partially supported by the Heilbronn Institute for Mathematical Research, Bristol University}}
\author{Susama Agarwala}
\author{Owen Patashnick}
%\address{\quad o.patashnick@bristol.ac.uk}
\date{\today}

\begin{abstract}
In this paper, we study the combinatorics of a subcomplex of the Bloch-Kriz cycle complex \cite{BlochKriz} used to construct the category of mixed Tate motives. The algebraic cycles we consider properly contain the subalgebra of cycles that correspond to multiple logarithms (as defined in \cite{GGL05}).  We associate an algebra of graphs to our subalgebra of algebraic cycles.
%This graphical point of view yields many computational gains when working with the Bloch Kriz cycle complex.
We give a purely graphical criterion for admissibilty. We show that sums of bivalent graphs correspond to coboundary elements of the algebraic cycle complex. Finally, we compute the Hodge realization for an infinite family of algebraic cycles represented by sums of graphs that are not describable in the combinatorial language of \cite{GGL05}.
%called necklace graphs.
%The techniques developed here are a first step towards understanding the cohomology of bar construction of the Bloch Kriz cycle complex.
\end{abstract}

\maketitle

\tableofcontents
\section{Introduction}

Let $\sM_T$ denote the category of mixed Tate motives and denote its associated Galois group by $G_T$. This Galois group has been defined in the literature in at least two distinct contexts, first by Bloch and Kriz (\cite{Bloch91}, \cite{BlochKriz}) but also by Levine (\cite{Le}) in what turned out to be Voevodsky's formalism (see for example \cite{DelGonch}).  Note that Spitzweck and Levine \cite{spitzweckthesis, LevineKthy} has shown that the two definitions are equivalent.

For the purposes of this paper we will take the Bloch-Kriz construction as our definition of $M_T$ and $G_T$.
%The periods of mixed Tate motives are denoted $\sP_T$, and their formal extension by $\hat{\sP_T}=\sP_T[(2\pi i)^{-1}]$. The following conjecture is well known to experts, \cite{Andres}:

%\begin{conj}\label{Tateconj} $\hat{\sP_T}$ is freely generated as an algebra by the logarithms of rational primes $\{\log(p) : p \text{ a prime}\}$, and by the odd zeta values
%$\{\zeta(2n+1): n\in \Z\}$
%of the Riemann zeta function.
%\end{conj}
%The set of Tate periods (the analogous principal homogeneous space associated to $G_T$) is denoted $\sP_T$, and their formal extension by $\hat{\sP_T}=\sP_T[(2\pi i)^{-1}]$.

Although a significant amount of work has gone into understanding $G_T$
%and $\sP_T$
, there is still much that is unknown about Tate motives, even over the rational numbers $\Q$.  In particular, the connection between $G_T$ and the unipotent completions of $\pi^1(\bP^1-\text{n points})^{unip}$ is still of current interest.

For $N\geq 1$, let $k_N$ be the cyclotomic field over $\Q$ generated by an $N$th root of unity, and $\sO_k$ its ring of integers. Let $M_{T,N}$ denote the full Tannakian subcategory of $M_T$ generated by the motivic fundamental group of $\bP^1-\{0,\infty, \mu_N\}$, with associated motivic Galois group $G_{T,N}$ and algebra of periods $\sP_{T,N}$. Here $\mu_N$ are the $N$th roots of $1$, though geometrically it could just  be a set of $N$ distinct points of $\C^*$. A question, probably going back to Grothendieck, is how much of the motivic fundamental group $G_T$ is measured by $G_{T,N}$, in particular $G_{T,1}$. This subcategory, and its integral analogues, were studied by Deligne and Goncharov in \cite{DelGonch}.  They showed, that over a number field, $\sP_{T,N}(\sO_k)$ is generated as a $\Q$ vector space by values of multiple polylogarithms.  There is a natural categorical inclusion $M_{T,N}\hookrightarrow M_T$ which induces surjections $\phi_N:G_T(\sO_k)\twoheadrightarrow G_{T,N}(\sO_k)$ (equivalently an injection $\sP_{T,N}\hookrightarrow \sP_T$).  Brown,  in the case $N=1$, and Deligne, in the cases $N\in \{2,3,4,6,8\}$ showed that $\phi$ was an isomorphism \cite{Brown11, Deligne2010}. Conversely, and more interestingly, Goncharov \cite{Gonch01dihedral} showed that for most $N$, $\phi$ has a nontrivial kernel.  Little is known about this kernel. Even less is known about this kernel if the ground field is a cyclotomic extension of a general number field (as opposed to a cyclotomic extension of $\Q$).

%It is of significant interest to be able to say something about elements in this kernel.

In particular, all known constructions of elements of $M_T$ lie in the subcategory $M_{T,N}$.

What is sorely needed is an approach to construct more general elements of $M_T$, especially ones that do not come from the motivic fundamental groups of  $\G_m-\mu_N$.
This paper is motivated in part by the desire to find a suitable framework to study this kernel.  We do not claim to have found such a framework in this paper, but are hopeful that we have taken a first step in the right direction.

The Bloch-Kriz definition of $M_T$ relies heavily on the theory of algebraic cycles.  While general enough to capture all mixed Tate motives, traditional methods of representing algebraic cycles (such as in terms of formal linear combinations of systems of polynomial equations) are notoriously difficult to work with, so progress in capitalizing on this description of the category to illuminate outstanding conjectures in the field has been slow.  In \cite{GGL05}, Gangl, Goncharov, and Levin suggest a simpler way to understand a subcategory of $\sM_T$ by relating specific algebraic cycles to rooted, decorated, binary trees.  Note that this approach necessarily restricts focus to motives generated by the motivic fundamental groups of $\G_m-\mu_N$.  Any attempt to study the kernel of $\phi$ defined above requires a more general framework.

Soud\'{e}res \cite{Souderes12, Souderes13} extends the family of algebraic cycles studied by Gangl, Goncharov and Levin to include those over a more general base scheme, in particular giving a rigorous construction of unital values of the multiple polylogarithms ie. multiple zeta values, as periods (and not just non-unital values of the multiple logarithms).  The combinatorial properties of these algebraic cycles, however, has not yet been explored.

%In this paper, we generalize the Gangl-Goncharov-Levin construction by considering a connection between algebraic cycles and graphs.
%via a graphical description and exploration of a suitable algebra of algebraic cycles.
%exploration of the connection between algebraic cycles and graphs.
%Note that the larger class of cycles considered in this paper also seem to lie in the subcategory of motives $M_{T,N}$ determined by multiple motivic polylogarithms.  Thus we cannot yet say anything new about the kernel of $\phi$.  On the other hand,

Let $\sA$ be the differential graded algebra (DGA) of cycles introduced by Bloch and Kriz, \cite{BlochKriz}. In this paper we generalize the Gangl-Goncharov-Levin construction as follows; we define a subalgebra of cycles, $\Asm \subset \sA$, that properly contains the subalgebra associated to multiple logarithms studied in \cite{GGL05},
and reinterprets $\Asm$ in terms of graphs.
By considering graphs as opposed to trees, and by loosening the valence restriction on the vertices, we enrich the tools available to study algebraic cycles, and are able to consider a larger subcomplex of cycles.  We hope this will lead to a better understanding of the complexity and richness underlying the Bloch-Kriz cycle complex, even in the restricted subclass we consider.  In particular, in Section \ref{H0BG}, we describe several examples of classes of algebraic cycles that define motives, most of which cannot be described by trees, and compute the Hodge realization of an infinite family of such classes.  Furthermore, in Section \ref{graphs}, we present a purely graphical interpretation of admissibility for the family of algebraic cycles we consider. We also give valency requirements for which classes of algebraic cycles will always be coboundaries in $H^0(\BsG)$.
There is a lot of interesting combinatorial structure in the types of underlying graphs, and their linear combinations, that give rise to allowable classes of motives. We have barely begun to explore this structure and feel strongly that it deserves further study.

The plan for the paper is as follows. In Section \ref{subcomplex}, we review of mixed Tate motives a la Bloch and Kriz \cite{BlochKriz} and introduce the subalgebra of $\bP^1$-linear parametrizable cycles, $\Asm \subset \sA$ of the algebra of admissible cycles. This subalgebra is the focus of our attention this paper. We then define a subcomblex $B(\Asm)$ of the bar construction on admissible cycles, $B(\sA)$. The category of comodules over $H^0(B(\Asm))$ is the (sub)category of motives we wish to study.

Section \ref{graphs} introduces an algebra of graphs, $\sGoneL{}{}$, that corresponds to the algebra $\Asm$. Theorem \ref{isomorphism} shows that the two algebras are isomorphic as DGAs. Since $\Asm$ is subalgebra of $\sA$, this implies that there is an injection from the algebra of graphs developed in this paper to the full Bloch Kriz cycle complex. In the process, we show in Theorem \ref{admissible} that the conditions for an arbitrary irreducible $\bP^1$-linear cycle to be admissible, that is, a generator of $\Asm$, can be defined and computed completely graphically.
%This simplifies our understanding of $\bP^1$-linear algebraic cycles, reducing them to simple combinatorial objects. This also broadens the tools one may use on $\bP^1$-linear cycles for future analysis.

In Section \ref{H0BG}, we give examples of classes in and results about $H^0(\BsG)$.
%We say that any sum of graphs that define a class in $H^0(\BsG)$ is completely decomposable.
In addition we show in Corollary \ref{samehandle} that in any completely decomposable (sum of) graphs either each summand has a two valent vertex, or none do. We further show, Theorem \ref{twovalentgraphthm}, that if a completely decomposable (sum of) graphs has two valent vertices, it is a coboundary in $\BsG$.
%This cuts down the number of cycles that can possibly lead to mixed Tate motives.

In Section \ref{Hodgerealization},
following the algorithm as outlined in \cite{BlochKriz, GGL05} and especially Kimura \cite{Kimura}, we compute the Hodge realization of a projective system of classes whose defining cycles are not describable by trees.  (All previously known explicit computations of the Bloch-Kriz Hodge realization have  been of cycles that can be described by trees.)
%(All previously known explicit computations of the Bloch-Kriz Hodge realization have  been of cycles that can be described by trees.)

%we explicitly calculate the Hodge realization of a second family of bar elements that define classes in $H^0(\BsG)$. This is the only other family of algebraic cycles (besides the original computation on cycles leading to the polylogarithms, \cite{BlochKriz}, and multiple logarithms \cite{GGL05}) for which such a calculation is known. While the particular family of sums of graphs we consider evaluate to $0$ under the Hodge realization functor, we hope that the graphical insights that lead to this breakthrough will allow the construction of an explicit Hodge realization functor, at least on $\bP^1$-linear cycles in future work.

\section{A subcomplex of algebraic cycles\label{subcomplex}}

In this section, we define a particular subcomplex of the Bloch-Kriz cycle complex that we develop in this paper. We begin with a review of the general mixed Tate motive construction via algebraic cycles. Then we proceed to describe parametrized cycles, and finally define the subcomplex of $\mathbb{P}_{1L}$-cycles that we use in the remainder of this paper.

\subsection{A review of mixted Tate motives}

In this paper we work with the category of mixed Tate motives over a field $k$,  $\sM(T)$, as constructed by Bloch and Kriz \cite{Bloch91, BlochKriz}.  When $k$ is a number field, this construction does not depend on any conjectures. In \cite{BlochKriz}, two conjectures are stated; that $gr_rK_n(F)\otimes \Q\cong CH^r(\Spec(F), n)\otimes \Q$, and that a certain algebra is quasi-isomorphic to its Sullivan 1-model. The first conjecture was subsequently proved more generally for all varieties $X$ independently by Bloch \cite{Bloch94, Bloch86}, Levine \cite{Levine94} and Spivakowsky (unpublished). The second conjecture, which is a strengthening of the Beilinson-Soul\'e conjecture for fields, is known for number fields by the work of Borel and Yang on the rank conjecture \cite{BorelYang}.  (The Beilinson-Soul\'e conjecture was already known to be true for number fields by the work of Borel from the early 1970s \cite{Borel74}).

In the rest of this section we review some details of their construction, following \cite{BlochKriz} closely.

We assume the reader is familiar with he concepts of
%Traditional approaches to this material require knowledge of
algebraic cycles, higher Chow groups,
minimal models, 1-minimal models and the bar construction for a commutative
differential graded algebra (DGA)
$A$.  For the reader who wishes to refresh her memory:
%However, we do not require familiarity with these concepts.
%In this section, we present a few necessary details about the bar construction. We leave the technical details to existing references in the literature.
The concept of a generalized minimal model is due originally
to Quillen (see for example \cite{Q}).  In the form used in this paper
(extensions by free one dimensional models)
it is due originally to Sullivan (\cite{Su}, see discussion starting p.\ 316).
A good reference for the applications of minimal models we have in mind is
the treatment in \cite{KM}, Part IV.  The bar
construction is due originally to Eilenberg and Mac Lane.  Good references for the use of the bar construction in this paper are \cite{Chen76} and (\cite{BlochKriz}, section 2).

In order to define the category of mixed Tate motives, $\sM(T)$, it suffices to define its motivic Galois group $G_T$ \cite{BlochKriz}. Equivalently, one may work with its dual Hopf algebra, $\sH_T$. This is defined from the DGA, $\sA$, of admissible algebraic cycles.

%\begin{dfn}
%The category of mixed Tate motives over a  number field, $\sM(T)$, is the category of comodules over $H_T$. The Hopf algebra $H_T$, in turn, is defined \bas H_T = H^0(B(\sA)) \;.\eas
%\label{MTMdfn}\end{dfn}

Below, following loc. cit., we define how to derive a  Hopf algebra from a commutative graded DGA $A$ which is cohomologically connected. That is, $H^0(A) = \Q$ and $H^{-n}(A) = 0$ for $n > 0$. Our DGA $A$ is not a Hopf algebra in general as the differential does not decompose.  The strategy, therefore, is to ``linearize'' $A$, i.e. form the minimal model $\sH(A)$ of $A$, which by construction is a Hopf algebra which is quasi-isomorphic to $A$.  The minimal model can be constructed quite explicitly via the bar construction. We start with a few definitions.

\begin{dfn}
\begin{enumerate}
\item Consider the commutative DGA $A = \oplus_i A_i$. Here, we refer to the grading on $A$ by degree: $\deg(a) = i \Leftrightarrow a \in A_i$. The tensor algebra, $T(A) = \oplus_{n} A^{\otimes n}$ is a commutative algebra under the shuffle product, $\sha$.
\item Let $D(A)$ be the ideal in $T(A)$ of degenerate tensor products, defined by \bas \{ a_1 \otimes \ldots \otimes a_n | a_i \in A\; ; \;  a_j \in k \textrm{ for some } j \} \;.\eas
\item The bar construction on $A$ is defined \bas B(A) = T(A) / D(A) \;.\eas It is a bi-graded algebra, with grading given by tensor degree and algebraic degree. The total degree of a monomial $a_1 \otimes \ldots \otimes a_n \in B(A)$ is defined by a shift in the degree of the tensor components in $A$. That is, \bas \td (a_1 \otimes \ldots \otimes a_n) = \sum_{i=1}^n (\deg(a_i)-1) \;. \eas Hence, the total degree of an element of the bar construction is the difference between the degree and the tensor degree. Write the bar construction as $ B(A) = \bigoplus_{i,j} B(A)^i_j $, where  \bas B(A)^i_j = \bigoplus_{\sum_1^i j_k-1 = j} A_{j_1} \otimes  \cdots \otimes A_{j_i} \eas has total degree $j$.
\end{enumerate}\label{Bardfn}
\end{dfn}

Since $A$ is a DGA, it is endowed with a differential structure (given by $\D : A \rightarrow A$) and a product structure, (given by $\mu : A \otimes A \rightarrow A$). These both extend to define differential structures on the bar construction $B(A)$, called the algebraic and multiplicative differentials, respectively.  Thus $(B(A), \D+\mu)$ is the following bicomplex. Further details and calculations involving the bar complex can be found in Section \ref{H0BG}.

\ba \xymatrix{ & \vdots & \vdots & \vdots & \\ \cdots \ar[r]^\mu &
  B(A)^3_0 \ar[r]^\mu \ar[u]^d & B(A)^2_1 \ar[r]^\mu
  \ar[u]^\D & B(A)^1_2 \ar[r]^\varepsilon \ar[u]^\D & 0
  \\ \cdots \ar[r]^\mu & B(A)^3_{-1} \ar[r]^\mu \ar[u]^\D &
  B(A)^2_0 \ar[r]^\mu \ar[u]^\D & B(A)^1_1 \ar[r]^\varepsilon
  \ar[u]^\D & 0 \\ \cdots \ar[r]^\mu & B(A)^3_{-2} \ar[r]^\mu
  \ar[u]^\D & B(A)^2_{-1} \ar[r]^\mu \ar[u]^\D & B(A)^1_0
  \ar[r]^\varepsilon \ar[u]^\D & \Q \\ \cdots \ar[r]^\mu &
  B(A)^3_{-3} \ar[r]^\mu \ar[u]^\D & B(A)^2_{-2} \ar[r]^\mu
  \ar[u]^\D & B(A)^1_{-1} \ar[r]^\varepsilon \ar[u]^\D & 0 \\ &
  \vdots \ar[u]^\D & \vdots \ar[u]^\D & \vdots
  \ar[u]^\D & }  \label{barbit}\ea

 When A is connected, cohomologically connected, and generated in degree one (a "$K(\pi,1)$ in the sense of Sullivan"), then it's minimal model is isomorphic to $\sH(A) : = H^0(B(A))$, where the cohomology is taken under the total derivative $\D + \mu$. Note that $B(A)$ is a Hopf algebra, with a product structure given by shuffle product, and a coproduct structure given by deconcatenation, which satisfy all the axioms for a Hopf algebra. This induces a well-defined product, coproduct, and Hopf algebra structure on $\sH(A)$.

Bloch and Kriz study a bar construction of a DGA of admissible cycles, $\sA=\oplus_i \sA_i$, defined below. The Hopf algebra dual to the motivic Galois group $G_T$, $H_T$, from definition \ref{MTMdfn} is exactly the Hopf algebra defined above, for the algebra of admissible cycles.

\begin{dfn}\label{higherChow}
\begin{enumerate}
\item Let $\square = \bP^1 \setminus \{1\} $. Write $\square^n = (\bP^1 \setminus \{1\})^n$. The boundary of this space is defined when one of the coordinates is set to $0$ or $\infty$.
\item For $I, J \subset \{1, \ldots n\}$, two disjoint subsets, write $F_{I, J}$ to indicate the codimension $|I \cup J|$ face of $\square^n$ with the coordinates in $I$ set to $0$, and the coordinates in $J$ set to $\infty$. Write $F_{\emptyset, \emptyset} = \square^n$ to indicate the entire space.
\item As usual, let $\cZ{p}{n}$ be the free abelian group generated by algebraic cycles of codimension $p$ in $\square^{n}$. These are the elements of weight $p$.
\item  Write $\cN{p}{n} \subset \cZ{p}{n}$ to be the free abelian subgroup generated by \emph{admisssible} algebraic cycles. A cycle $\mathcal{Z} \in \cN{p}{n}$ is one that intersects each face $F_{I, J}$ of $\square^n$ in codimension $p$, or not at all.
\item Let $\Alt$ be the alternating projection with respect to the action of the group $\fS_n\rtimes (\Z/2\Z)^n$ on $\cN{p}{n}$.  Here the symmetric group $\fS_n$ acts by permutation of coordinates, and the $i$-th copy of $(\Z/2\Z)^n$ acts by taking a coordinate to its multiplicative inverse.
\item Write \bas \sA_i^n = \Alt\cN{n}{2n-i}\otimes \Q \;, \eas where $i$ is the degree of the algebraic cycle and $n$ the codimension. This is a bigraded algebra, by weight and degree. The weight graded pieces, $\sA^n:=\oplus_i \Alt \cN{n}{2n-i}\otimes \Q$, define a complex, by the differential operator defined in equation \eqref{cyclediff}. Each degree graded piece is $\sA_i:=\oplus_n \Alt\cN{n}{2n-i}\otimes \Q$.
\end{enumerate}
\end{dfn}

\begin{rem} The main result of Section \ref{graphDGA}, is to identify which cycles are elements in $\sA$. In order to determine \emph{which} algebraic cycles are admissible, we must consider the space of \emph{all} algebraic cycles, including those that are not admissible. Therefore, in this paper, when we write $\mathcal{Z}^p(\square^n)$, we mean the entire space of algebraic cycles. We denote admissible cycles by the notation $\cN{p}{n}$. \end{rem}

We now define the DGA structure of $\sA$. Consider two admissible cycles, \bas \sZ_i \in \cN{n}{i} \textrm{ and } \sZ_j \in \cN{m}{j}\;.\eas Write $(\sZ_i, \sZ_j) \in \cN{n+m}{2(n+m)-(i+j)}$ to indicate the admissible cycle defined by $\sZ_i$ on the first $2n-i$ coordinates and $\sZ_j$ on the last $2m-j$ coordinates. The product on the associated elements in $\sA$ is given by \bas \mu (\Alt\mathcal{Z}_i \otimes \Alt \mathcal{Z}_j) = \Alt (\mathcal{Z}_i, \mathcal{Z}_j) = (-1)^{ij}\Alt (\mathcal{Z}_j, \mathcal{Z}_i)\;, \eas Where we drop the $\otimes \Q$ notation for simplicity. The last inequality comes from the properties of $\Alt$, and defines a graded commutative structure on $\sA$.

\begin{dfn}
An element $\sZ \in \sA$ is decomposable if it can be expressed as the product of two non-trivial cycles.
\end{dfn}

Next, we define the differential structure on $\sA$. Consider $\sZ \in \sA$. Let $\D_{j,0} \mathcal{Z}$ indicate the intersection of $\mathcal{Z}$ with the face $F_{j,\emptyset}$. Similarly, let $\D_{j,\infty} \mathcal{Z}$ indicate the intersection of $\mathcal{Z}$ with the face $F_{\emptyset, j}$. These two operators define the differential $\D$ on $\sA$: \ba \D \mathcal{Z} = \sum_{j=1}^{2n-i} (-1)^{j-1}(\D_{j, 0} - \D_{j, \infty}) \mathcal{Z} \;.\label{cyclediff}\ea

\begin{rem} It is worth noting at this point the difficulty of finding elements of $\sA$, in particular, in finding cycles that satisfy the condition of admissibility. One of the achievements of this paper is to give a clear, simple condition for identifying admissible cycles for a large subclass of cycles, called $\bP^1$-linear cycles. In particular, see Theorem \ref{admissible}. \end{rem}

For an element $\varepsilon \in \bigoplus_n B(\sA)^n_{i}$ to define a class in $H^i (B(\sA))$, each graded component must have decomposable algebraic boundary. This comes from the fact that $(\D + \mu )(\varepsilon) = 0. $ In order to define what it means for a cycle to have decomposable boundary, let $\pi_m$ be the projection of $\epsilon$ onto the $m^{th}$ tensor component. That is, $\pi_m(\varepsilon) \in B(\sA)^m_{i}$. Then, for each $m$, $\D(\pi_m \varepsilon)$ is a decomposable element. In particular, \bas \D(\pi_m \varepsilon) = -\mu (\pi_{m+1} \varepsilon) \;.\eas

\begin{dfn} Consider an element $\varepsilon  \in B(\sA)$.
\begin{enumerate}
\item
The projection $\pi_i(\varepsilon) \in B(\sA)_i^n$ is decomposable if it has a decomposable algebraic boundary. That is, if there
exists an $\varepsilon' \in B(\sA)_{i+1}^{n+1}$ such that $\D(\pi_i(\varepsilon)) =
- \mu (\varepsilon')$. That is, the coboundary of the projection $\pi_i(\varepsilon)$ is in
the image of the product map $\mu$.
\item
An element $\varepsilon \in B(\sA)$ is completely decomposable if, for all $i$, $\pi_{i}(\varepsilon)$
is decomposable, with \bas \D\pi_{i}(\varepsilon) = -\mu \pi_{i+1}(\varepsilon) .\eas
\end{enumerate}
\label{decompdfn}\end{dfn}

\begin{dfn} We say that the element $\varepsilon \in \bigoplus_n B(\sA)^n_{i}$ is minimally decomposable if it is completely decomposable, and cannot be written as a sum of two non-trivial completely decomposable elements. That is one cannot write  $\varepsilon = \varepsilon_1 + \varepsilon_2$, where each $\varepsilon_i \neq 0$ and is completely decomposable. \label{mindecompdfn} \end{dfn}

\begin{rem}Notice that if $\varepsilon$ is minimally decomposable, it is determined by $\pi_{n_0} (\epsilon)$, where $n_0$ is the smallest integer for which $\pi_n(\varepsilon) \neq 0$. Therefore, by abuse of notation, we say that $\pi_{n_0} (\varepsilon)$ defines a class in $H^i(B(\sA))$. In all examples in this paper, $n_0=1$.
%Furthermore, if $\varepsilon$ is minimally decomposable, $n_0=1$.
%For the purposes of this paper, we may assume that $n_0 = 1$.
\label{mindecomprmk}\end{rem}

Next we give an example of an admissible cycle that defines a class in $H^0(B(\sA))$.

\begin{eg} \label{Totaroeg}
Consider the cycle $\sZ_T(a) = \Alt (t,1-t,1-\frac{a}{t}) \in \sA^2_1$. This is a parametric representation of the algebraic cycle determined by the set of system of equations $\{x+y=1, xz=x+a\}$. This, the Torato cycle \cite{Totaro92}, has codimension $2$ in $\square^3$ and is a degree one element in $A$, $\sZ_T(a) \in \sA^2_1$,

We check that $\sZ_T(a)$ has a completely decomposable boundary. Therefore, it defines a class in $H^0(B(\sA))$. To see this, compute $\D \sZ_T(a)$. The intersections $\D_{\infty, i} \sZ_T(a)$ give the empty cycles for $i \in \{1, 2, 3\}$. This is because setting one of the coordinates of $\sZ_T(a)$ to $\infty$ sets a different coordinate to $1$. The same holds for $\D_{0, 1} \sZ_T(a)$ and $\D_{0, 2} \sZ_T(a)$. Therefore, \bas \D \sZ_T(a) =  \D_{0, 3} \sZ_T(a) = \Alt (a, 1-a) =  \mu [\Alt (a) | \Alt (1-a)]\;. \eas The last equality comes from the product structure on $A$. Since $(a)$ and $(1-a)$ are constant cycles, $\D [\Alt (a) | \Alt (1-a)] = 0$ by the Lebnitz rule. Therefore, $\sZ_T(a) \oplus -[\Alt (a) | \Alt (1-a)] \in \ker (\D + \mu)$. Since $\sZ_T(a)$ has total degree $0$ in $B(\sA)$, it defines a class in $H^0(B(\sA))$.

The Hodge realization functor associates the period $\Li_2(a)$ to the cycle $\sZ_T(a)$ \cite{BlochKriz}. To do this, consider the $\sA$ module, $\sT$, defined by maps from $n$ simplices, $\Delta_n$, to $\square^n$. There is an element $\zeta(a)$ in the circular bar construction $B(\sT, \sA)$ such that $\zeta(a) + 1 \otimes \sZ_T(a)$ defines a class in $H^0(B(\sT, \sA))$. The summands of $\zeta(a)$ that are supported completely on $\Delta_2$ defines the integrand of the associated period.
\end{eg}

This example hints at another shortcoming of the current state of technology surrounding algebraic cycles. We are interested in defining elements of $B(\sA)$ that define classes of $H^0(B(\sA))$. In particular, we are interested in cycles with boundaries that can be written as products of other cycles, as is the case for the Torato cycle in example \ref{Totaroeg}. In Section \ref{examples}, we provide several examples of such sums of cycles in weight $4$. However, we have not yet addressed this issue of how to find such sums in general. However, we hope that the graphical point of view presented in this paper will shed light on the problem of identifying cycles with completely decomposable boundaries. We leave this for future work.

%This paper, has, however, taken a step forward in identifying $0$ classes of $H^0(B(\sA))$. There are two major advances in this area. The graphical point of view taken in this paper shows that any admissible cycle corresponding to a graph with a two valent vertex leads to a $0$ cohomology class, see Theorem \ref{twovalentthm}. Furthermore, we are interested in associating periods to the classes in $H^0(B(\sA))$, by identifying the appropriate element of $H^0(B(\sT, \sA))$. We do not have a complete algorithm for the latter. However, in Section \ref{Hodgerealization} we do this calculation for an infinite family of algebraic cycles. This calculation shows that the family of cycles we consider are $0$ classes in $H^0(B(\sA))$.

\subsection{A subalgebra of $\sA$}

%\cite{T}, \cite{Bloch91}, eg \cite{GGL07}

%What does $(t,1-t,1-\frac at)$ really mean?  What the notation says without explanation.  What we want it to mean.  More accurately described as $(\frac tu, 1-\frac tu, 1-\frac{au}t)$ without $1$ in any coordinate.  $\A^1\to \A^1-\{1\}$.  making a consistent choice of projectivization.

Unfortunately, the standard parametric notation for Tate cycles is rather misleading.  For example, consider the usual form for the Totaro cycle $\sZ_T(a) = \Alt (t,1-t,1-\frac{a}{t})\in \sA^2_1$, defined in example \ref{Totaroeg}, and in the literature \cite{Totaro92, GGL05}. It is technically defined on $\square_k^3=(\bP^1-\{1\})^3$, but is written as if it is defined on $\sA^3=(\bP^1-\{\infty\})^3$. In actuality, the Totaro cycle (for $a\in k^*$) is an algebraic cycle defined by the system of equations $\{x+y=1, xz=x+a: (x,y,z)\in (\bP^1-\{1\})^3\}$ together with a parametrization map $\bP^1\rightarrow (\bP^1-\{1\})^3$. However, when manipulated in practice, the cycle is understood a) to come equipped with a parametrization map and b) to be defined at the hyperplanes with one coordinate equal to $\infty$, and not defined at the hyperplanes with one coordinate equal to $1$.  This is unnecessarily obtuse. It can be described as the intersection of the image of
\bas
\bP^1 &\rightarrow  (\bP^1)^3 \\
(T:U) &\mapsto   (\frac TU, \frac{U-T}{U}, \frac{T-aU}{T})
\eas
with the complement of the hyperplanes of $(\bP^1)^3$ defined by setting some coordinate equal to 1.

In light of this example, we work with parametrized cycles.

\begin{dfn} \label{parammap}
A parametrized cycle is a pair, $(Z, \phi)$, consisting of an algebraic cycle $Z \in \cZ{p}{n}$ and a parametrization $\phi:\bP^{n-p} \rightarrow (\bP^1)^n $ satisfying the following: $\phi$ induces a map on the group of algebraic cycles, \bas \phi_* : \sZ^0(\bP^{n-p}) \rightarrow \sZ^p((\bP^1)^n) \;. \eas  Then given the inclusion $i : \square^n \hookrightarrow (\bP^1)^n$, we have
 \bas Z = i^*\phi_*(\bP^{n-p});, \eas where $\bP^{n-p}$ is the generator of  $\sZ^0(\bP^{n-p})$.
\end{dfn}

For $\sZ \in \cZ{p}{n}$, write the parameterizing map $\phi = (\phi_1, \ldots , \phi_n)$, where each $\phi_i$ corresponds to the image in a coordinate of $\square^n$. There are, of course, multiple possible parameterizations of any cycle $\sZ \in \cZ{p}{n}$. In this paper, we are interested in the algebraic cycles themselves, not the particular parameterizations. If the same cycle $Z$ can be represented by two different parametrizations, $(Z, \phi)$ and $(Z, \phi')$, we say that $\phi$ and $\phi'$ are equivalent parameterizations. In this paper, we are interested in cycles that \emph{can} be endowed with a $\bP^1$-linear parametrization.

\begin{dfn}A cycle $\sZ \in \cZ{p}{n}$ is $\bP^1$-linear if it can be parameterized by a $\phi$ such that each component can be written as $\phi_i \in  \{(\onem{t_1}{a_it_2})^\zeta, (\onem{t_2}{a_it_1})^\zeta,
\frac{t_1}{a_it_2}^\zeta \, \}$, with $a_i \in k^\times$, $\zeta \in \{\pm 1\}$. Such a $\phi$ is called a $\bP^1$-linear parametrization, and can be written as a map on $\bP^1$ via the following commutative diagram:
\bas
\xymatrix{ \bP^1 \ar[r] \ar[d]^{\phi_j} & \bP^{n-p} \ar[d]^{\phi} \\ \bP^1 \ar@{^(->}[r]_{i_j} & (\bP^1)^{n}}
 \;. \eas
The top arrow is given by a map \bas(t_1:t_2)\mapsto (0:\ldots :0:t_1:0:\ldots :0:t_2:0:\ldots :0),\eas
and the bottom arrow is given by inclusion into the $j^{th}$ coordinate. \end{dfn}

\begin{dfn} Denote the free abelian groups of $\bP^1$-linear cycles by $\cZoneL{n}{m}$. Write $\cNoneL{n}{m}$ to denote the free abelian group of $\bP^1$-linear \emph{admissible} cycles. \end{dfn}

The goal of this section is to define a subDGA of $\sA$, the algebra of admissible cycles, that is generated by $\cNoneL{n}{2n-i}$. Call it \bas \sA_{1L}= \bigoplus_i \sA_{1L, i} = \bigoplus_{n, i}\Alt\cNoneL{n}{2n-i} \otimes \Q\;.\eas %This is generated by elements of the form $(f_1, \ldots , f_n)$. The form of these functions come from the fact that the parametertization map $f$ is the restriction of a map to $(\bP^1)^k$. Therefore, each coordinate function can be written as the product of fractional linear transformations \bas  f_m:=\prod_{n=1}^{r_{m}}\frac{a_nt_{m_1}-b_nt_{m_2}}{c_nt_{m_1}-d_nt_{m_2}}\eas where $a_n,b_n,c_n,d_n\in k$.

The graded commutative structure on $\sA_{1L}$ comes from the product structure on $\sA$, along with the fact that the product of two parameterizable cycles is still parameterizable. It remains to check that the differential structure on $\sA$ is well defined on $\sA_{1L}$. The differential on $\sA$ comes from intersecting each coordinate of an element $\sZ \in \sA^n_i$ with the appropriate $0$ and $\infty$ face of $\square_k^{2n-i}$. Consider $Z \in \cNoneL{n}{2n-i}$. Let $\phi$ be a parametrization on $Z$. Then, the intersection of $Z$ with a particular face corresponds to the pullback of $\phi$ by the appropriate face map. Therefore, the differential of $Z$ is also a $\bP^1$-linear parametrizable cycle.

If $\Alt Z \in \sA_{1L}$ is a decomposable cycle of codimension $i$, write $\Alt Z = \Alt (Z_1, \ldots Z_r)$ as above. The Leibnitz rule and properties of $\Alt$ show that $\D \Alt (Z_1, \ldots Z_r)$ is also parametrizable.

The subalgebra $\sA_{1L}$ naturally defines a subcomplex, $B(\sA_{1L})$, that defines a subcategory of $\sM(T)$ from \cite{BlochKriz}. This is exactly the subcategory of mixed Tate motives consisting of elements with either a prescribed pole or zero at the points $0$, $\infty$, or $1$ in each coordinate.

\begin{comment}
\begin{remark}
In the following proposition we will need all components of our parametrizing maps to split
completely.  For simplicity, we could take
%we take
our ground field to be algebraically closed.  This assumption is not the best possible.  If we pass to the splitting field $K$ of all the components of all parametrizing maps (note that, since the components are all defined on $\P^1$, they factor completely into liners over some finite extension of the ground field $k$), we suspect the following proposition can be extended to the general case if care is taken to properly take into account the natural action of Gal($\bar{K}/K$).
\end{remark}
\end{comment}

The algebra $\sA_{1L}$ contains all the Totaro cycles. Moreover, it contain a large class of cycles which correspond to the multiple logarithms \cite{GGL05}. Therefore, conjecturally, it contains all the cycles necessary to define the full category of mixed Tate motives. There has been some effort to understand subalgebras of $\sA_{1L}$ in terms of polylogarithms and multiple logarithms \cite{GGL05, Souderes12}. In this paper, we study a subalgebra $\Asm \subset \sA_{1L}$ that specifically excludes the Totaro cycles. The combinatorics of the cycles in $\Asm$ are studied in Section \ref{graphs}. The graphs introduced in Section \ref{graphs} correspond to the subalgebra $\Asm$, which excludes cycles with coordinates of the form $a_i \frac{t_i}{t_j}$.

\begin{dfn}
Let $\Asm$ be the algebra of $\bP^1$-linear cycles, where $\phi_i \in  \{(\onem{t_1}{a_it_2})^\zeta, (\onem{t_2}{a_it_1})^\zeta\}$, with $a_i \in k^\times$, $\zeta \in \{\pm 1\}$
\end{dfn}

\section{Motivic Graphs\label{graphs}}

The first graphical description of some of the algebraic cycles that
arise in the category $\sM(T)$ of mixed Tate motives was given by
Herbert Gangl and his collaborators in \cite{GGL07} and \cite{GGL05} in
their description of $R$-deco trees. These provide a description of a
particular proper sub DGA of $\Asm$.

In particular, they represent a subalgebra of cycles by labelled oriented
trees.  For example,

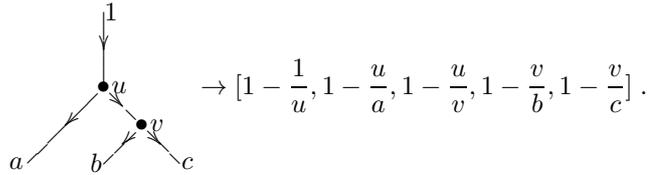
\bas
\begin{xy}
(0,10) *{}="r" +(1,0)*{1},
(0,0) *{\bullet}= "u" +(2,0)*{u},
(5, -5) *{\bullet}= "v"+(2,0)*{v},
(10, -10) *{}= "c" +(1,0)*{c},
(0, -10) *{}= "b" +(-1,0)*{b},
(-10, -10) *{}= "a" +(-1.5,0)*{a},
"r"; "u" **{\dir{-}}?/0pt/*{\dir{>}},
"u"; "v" **{\dir{-}}?/0pt/*{\dir{>}},
"v"; "c" **{\dir{-}}?/0pt/*{\dir{>}},
"v"; "b" **{\dir{-}}?/0pt/*{\dir{>}},
"u"; "a" **{\dir{-}}?/0pt/*{\dir{>}},
\end{xy} \rightarrow [\onem{1}{u},\onem{u}{a}, \onem{u}{v}, \onem{v}{b},
\onem{v}{c}]\;.
\eas

Note that this assignment depends on several arbitrary choices, such
as a choice of orientation as well as a choice of affine patch.

In this section we give a more general graphical depiction
that encapsulates all $\Asm$ cycles, using decorated oriented,
non-simply connected graphs.

For example, the tree and cycle above come from the labeled
oriented graph
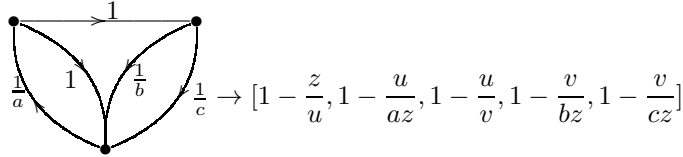
\bas
\begin{xy}
(-12, 10) *{\bullet}= "u",
(12, 10) *{\bullet}= "v",
(0, -7) *{\bullet}= "z",
"u"; "v" **{\dir{-}}?/0pt/*{\dir{>}}+(1,1.5)*{1},
"u"; "z" **\crv{+(1,12)}?/0pt/*{\dir{>}} +(-2,-1)*{1},
"z"; "u" **\crv{+(-1,-12)}?/0pt/*{\dir{>}}+(-2,1)*{\frac{1}{a}} ,
"z"; "v" **\crv{+(1,-12)}?/0pt/*{\dir{<}} +(3,0)*{\frac{1}{c}},
"v"; "z" **\crv{+(-1,12)}?/0pt/*{\dir{>}} +(2,-1)*{\frac{1}{b}},
\end{xy} \rightarrow [\onem{z}{u},\onem{u}{az}, \onem{u}{v}, \onem{v}{bz},
\onem{v}{cz}] \eas by taking the affine patch at $z= 1$, which
graphically amounts to removing the vertex labeled $z$ and changing
the labels from the edges of the graph to the vertices of the tree.

The approach of this paper produces far more algebraic cycles that are not seen via the approach given
in \cite{GGL07} and \cite{GGL05}. In particular, we can study cycles
represented by graphs that cannot be represented by a tree in any
affine patch.  For example, the graph
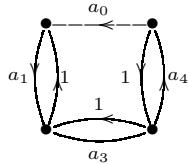
\bas {\begin{xy}
(-7,7) *{\bullet}="x",
(7,7) *{\bullet}= "y",
(-7, -7) *{\bullet}= "z",
(7, -7) *{\bullet}= "w",
"y"; "x" **{\dir{-}}?/0pt/*{\dir{>}}+(0,2)*{_{a_0}},
"x"; "z" **\crv{+(-3,7)}?/0pt/*{\dir{>}}+(-2,0)*{_{a_1}},
"x"; "z" **\crv{+(3,7)}?/0pt/*{\dir{<}}+(1, 0)*{_{1}},
"y"; "w" **\crv{+(-3,7)}?/0pt/*{\dir{>}}+(-2,0)*{_{1}},
"y"; "w" **\crv{+(3,7)}?/0pt/*{\dir{<}}+(2, 0)*{_{a_4}},
"z"; "w" **\crv{+(-7,3)}?/0pt/*{\dir{<}}+(0,2)*{_{1}},
"z"; "w" **\crv{+(-7,-3)}?/0pt/*{\dir{.}}+(0, -2)*{_{a_3}},
\end{xy}} \eas
in this paper corresponds to the algebraic cycle \bas \Alt
[\onem{z}{x}, \onem{x}{a_1 z}, \onem{w}{z}, \onem{z}{a_2w},
  \onem{y}{w}, \onem{w}{a_3y}, \onem{y}{a_0x}] \;.\eas Yet there is no
affine patch one can take (i.e. a vertex one can remove) that will
result in a tree of the form studied in \cite{GGL05}.

The aim of Section \ref{graphs} is to construct an algebra of graphs,
$\sGoneL{}{} = \bigoplus_{\bullet, \star}\sGoneL{\bullet}{\star}$ that is isomorphic to the algebra of
admissible cycles $\Asm$ as DGAs. The definition of
this algebra is given at the end of Section
\ref{admissiblegraphs}. Most of Sections \ref{algofgraphs} and
\ref{admissiblegraphs} are devoted to building up $\sGoneL{}{}$
step by step. We begin with a general set of oriented graphs with labeled
and ordered edges, $\sG(k^\times)$. This corresponds to the set of generators of the free abelian group $\cZoneL{\bullet}{2\bullet-\star}$. We define a monoid structure on the set, such that $\sG(k^\times)$ generates an algebra, $\Gprg$.
Then we consider the alternating representation on the graphs, by imposing an equivalence relation on them by the ordering of this edges. This gives an algebra homomorphism from $\Gprg^\bullet_\star/\simord$ to the
algebra of cycles $\Alt \cZoneL{\star}{2\bullet-\star}$.

However, we wish for a DGA homomorphism to the algebra of admissible, $\bP^1$-linear
cycles, $\Asm \subset \cN{\star}{2\bullet-\star}$. To do this, we define a subset of
$\sG_{ad}(k^\times)\subset \sG(k^\times)$ which we show corresponds
to admissible graphs in Theorem \ref{admissible}. We write $\Q[\sG_{ad}]$ to indicate the algebra generated by $\sG_{ad}(k^\times)$. In order to establish a DGA isomorphism between $\Asm$ and $\Q[\sG_{ad}]$, we must define a differential operator on graphs. To do this, we need two further equivalence relations among graphs, which we call $\simv$ and $\simori$. In Section \ref{graphDGA}, we show that $\Gprg/(\simord,
\simv)$ is a DGA of graphs.

In Section \ref{admissiblegraphs}, we show one of the main findings of
this paper, that \emph{admissibility of $\bP^1$-linear cycles can be
  encoded purely by labeled oriented graphs.} In particular, there is no further
algebraic input necessary. Imposing the third equivalence relation gives the desired isomorphism \bas \sGoneL{}{} =
\Q[\sG_{ad}(k^\times)]/(\simord,\simv,\simori) \simeq
\Asm \;.\eas

%Note that set algebra of graphs we work with, $\sGoneL{\bullet}{\star}$ is rather
%different from the set allowed by Kontsevich? in his description of a
%Hopf algebra of graphs. In particular we rely on loops and multiple
%edges which he specifically forbids. Also, our associated Hopf algebra
%(constructed out of the DGA defined in section \ref{DGA} using the
%construction given in section \ref{MTMdef}), and in particular its
%coproduct structure, does not appear to be related to the one(s) given
%by Kreimer and Connes-Kreimer in [citation], at least in any obvious
%way. {\color{red} Does this need to be said twice? I'm tempted to nix
%  it here and leave it earlier.}

\subsection{An interesting algebra of graphs\label{algofgraphs}}

In this section, we introduce a general set of biconnected graphs with
oriented, labeled, and ordered edges. We impose a product structure on
it. This defines an algebra of graphs that corresponds to the
algebra of general (not necessarily admissible)
algebraic cycles.

We work over a number field $k$.

\begin{dfn}
  Let $\sG(k^\times)$ be the set of graphs with biconnected connected
  components, with oriented and ordered edges, each labeled by an
  element of $k^\times \times \Z/2\Z$.
\end{dfn}

In practise, we say that the edges of $G$ are labeled by a non-zero
number and a sign.

For a graph $G \in \sG(k^\times)$,
let $V(G)$ be the set of vertices of $G$, and
$E(G)$ be the unordered set of edges of the graph. However, we are working with graphs with ordered edges. Therefore we must consider the ordered set of edges.

\begin{dfn}
Let $\omega(G)$ be
  the ordered set of edges of $G$, where $\omega(e)$ expresses the
  ordinality of the edge $e \in E(G)$ in $\omega(G)$. Write
  $\sgn_{\omega(e)}$ to indicate the sign associated the edge $e$.
\end{dfn}

The loop number, or first Betti number, of a graph $G \in
\sG(k^\times)$ is \ba h^1(G) = |E(G)| - | V(G)| + h^0(G) \label{bettidef}\;,\ea
where $h^0(G)$ counts the number of connected components of the
graph. The vector space $H^1(G)$ is spanned by graphical cycles of the
unoriented graph underlying $G$.

\begin{rem}
  There are multiple conventions regarding the definition of cycles in
  graphs in the literature. In this paper, $L \subset E(G)$, together
  with an orientation, (possibly different from the orientation on the
  individual edges in $E(L)$) is a graphical cycle of the graph $G$ if
  it defines a path in $G$ that starts and ends at the same
  vertex. Specifically, the path in $G$ defined by the edges of $L$
  does not need to respect the orientation of the edges in $L$. A
  graphical \emph{loop} is a graphical cycle that does not intersect
  itself until the final vertex. \end{rem}

In this paper, we will concern ourselves only with graphical loops of $G
\in \sG(k^\times)$.

\begin{eg}
\label{samplegraph}
 Consider the graph
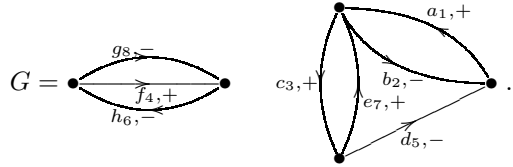
\bas G = {\begin{xy}
(0, 0) *{\bullet}= "a",
(20, 0) *{\bullet}= "b",
(35,10) *{\bullet}="t",
(35,-10) *{\bullet}= "u",
(55, 0) *{\bullet}= "z",
"t"; "z" **\crv{+(-15,0)}?/0pt/*{\dir{>}} +(1,-2)*{_{b_2,-}},
"z"; "t" **\crv{+(15,0)}?/0pt/*{\dir{>}}+(2,2)*{_{a_1,+}},
"t"; "u" **\crv{+(-5,10)}?/0pt/*{\dir{>}} +(-3,0)*{_{c_3, +}},
"u"; "t" **\crv{+(5,-10)}?/0pt/*{\dir{>}} +(3.5,-2.5)*{_{e_7, +}},
"u"; "z" **{\dir{-}}?/0pt/*{\dir{>}}+(1,-2.5)*{_{d_5,-}},
"a"; "b" **\crv{+(-10,7)}?/0pt/*{\dir{>}} +(-2,1)*{_{g_8, -}},
"a"; "b" **\crv{+(-10,-7)}?/0pt/*{\dir{<}} +(-2,-1)*{_{h_6, -}},
"a"; "b" **{\dir{-}}?/0pt/*{\dir{>}}+(1,-1.5)*{_{f_4, +}},
\end{xy}} \;.\eas These are in $\sG(k^\times)$, assuming $a_1 \ldots g_8$ are all in $k^\times$.  The subscripts on the coefficients indicate
the ordering of the edges, the signs
on the edges are as indicated.
\end{eg}

We impose a product structure on the set $\sG(k^\times)$. For $G,
G' \in \sG(k^\times)$, let $G \amalg G'$ be the disjoint union of the
graphs, without an overall ordering imposed on the union of the
edges. The product of two graphs $G \cdot G'$ is the graph $G \amalg
G'$, with the edges of $G$ appearing before the edges
of $G'$. In particular, this is a non-commutative product, \bas G
\cdot G' \neq G' \cdot G\;, \eas as the ordering of the edge set, $E(G
\amalg G')$, in the two cases is not the same.

\begin{eg} \label{skewprodeg}
In this example, we concern ourselves primarily with the ordering of
the edges in the product. Therefore, we write label the edges with elements of $k^\times$ and the ordering, and neglect to indicate the sign. One may
assume, without loss of generality, that the signs are all positive
in the graphs below.

Consider the graphs
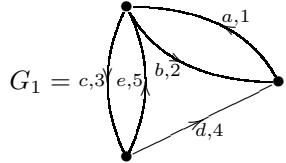
\bas G_1 = {\begin{xy}
(0,10) *{\bullet}="t",
(0,-10) *{\bullet}= "u",
(20, 0) *{\bullet}= "z",
"t"; "z" **\crv{+(-15,0)}?/0pt/*{\dir{>}} +(-2,-1)*{_{b,2}},
"z"; "t" **\crv{+(15,0)}?/0pt/*{\dir{>}}+(2,1)*{_{a,1}},
"t"; "u" **\crv{+(-5,10)}?/0pt/*{\dir{>}} +(-2,0)*{_{c,3}},
"u"; "t" **\crv{+(5,-10)}?/0pt/*{\dir{>}} +(-2,0)*{_{e,5}},
"u"; "z" **{\dir{-}}?/0pt/*{\dir{>}}+(1,-1.5)*{_{d,4}},
\end{xy}}
\eas
and
\bas G_2 = 
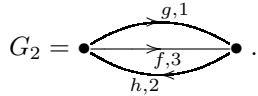
{\begin{xy}
(30, 0) *{\bullet}= "a",
(50, 0) *{\bullet}= "b",
"a"; "b" **\crv{+(-10,7)}?/0pt/*{\dir{>}} +(2,1.5)*{_{g,1}},
"a"; "b" **\crv{+(-10,-7)}?/0pt/*{\dir{<}} +(-2,-1.5)*{_{h,2}},
"a"; "b" **{\dir{-}}?/0pt/*{\dir{>}}+(1,-1.5)*{_{f,3}},
\end{xy}} \;. \eas

First, notice that the graph in example \ref{samplegraph} cannot be
written as the product of $G_1$ and $G_2$, since the edges of one
connected component do not precede the edges of the other, as written.

The product \bas G_1\cdot G_2 = 
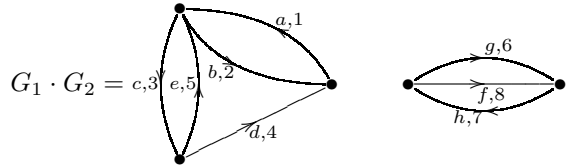
{\begin{xy}
(0,10) *{\bullet}="t",
(0,-10) *{\bullet}= "u",
(20, 0) *{\bullet}= "z",
(30, 0) *{\bullet}= "a",
(50, 0) *{\bullet}= "b",
"t"; "z" **\crv{+(-15,0)}?/0pt/*{\dir{>}} +(-2,-1)*{_{b,2}},
"z"; "t" **\crv{+(15,0)}?/0pt/*{\dir{>}}+(2,1)*{_{a,1}},
"t"; "u" **\crv{+(-5,10)}?/0pt/*{\dir{>}} +(-2,0)*{_{c,3}},
"u"; "t" **\crv{+(5,-10)}?/0pt/*{\dir{>}} +(-2,0)*{_{e,5}},
"u"; "z" **{\dir{-}}?/0pt/*{\dir{>}}+(1,-1.5)*{_{d,4}},
"a"; "b" **\crv{+(-10,7)}?/0pt/*{\dir{>}} +(2,1.5)*{_{g,6}},
"a"; "b" **\crv{+(-10,-7)}?/0pt/*{\dir{<}} +(-2,-1)*{_{h,7}},
"a"; "b" **{\dir{-}}?/0pt/*{\dir{>}}+(1,-1.5)*{_{f,8}},
\end{xy}} \; \eas
while the product in the other order is
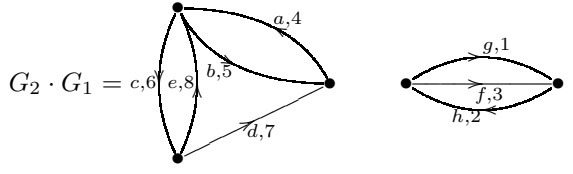
\bas G_2 \cdot G_1 = {\begin{xy}
(30, 0) *{\bullet}= "a",
(50, 0) *{\bullet}= "b",
(0,10) *{\bullet}="t",
(0,-10) *{\bullet}= "u",
(20, 0) *{\bullet}= "z",
"t"; "z" **\crv{+(-15,0)}?/0pt/*{\dir{>}} +(-2,-1)*{_{b,5}},
"z"; "t" **\crv{+(15,0)}?/0pt/*{\dir{>}}+(2,1)*{_{a,4}},
"t"; "u" **\crv{+(-5,10)}?/0pt/*{\dir{>}} +(-2,0)*{_{c,6}},
"u"; "t" **\crv{+(5,-10)}?/0pt/*{\dir{>}} +(-2,0)*{_{e,8}},
"u"; "z" **{\dir{-}}?/0pt/*{\dir{>}}+(1,-1.5)*{_{d,7}},
"a"; "b" **\crv{+(-10,7)}?/0pt/*{\dir{>}} +(2,1.5)*{_{g,1}},
"a"; "b" **\crv{+(-10,-7)}?/0pt/*{\dir{<}} +(-2,-1)*{_{h,2}},
"a"; "b" **{\dir{-}}?/0pt/*{\dir{>}}+(1,-1.5)*{_{f,3}},
\end{xy}} \; .\eas

It is the ordering on the two graphs that distinguishes the two
products. Everything else about the labeled oriented graphs $G \cdot
G'$ and $G' \cdot G$ is the same.
\end{eg}

This non-commutative product gives $(\sG(k^\times), \cdot)$ a free
monoidal structure. The unit in the monoid is given by the empty
graph, which has no loops and no edges, and therefore no labels.

\begin{dfn}Let $\Gprg$ be the free algebra generated by the
  monoid $(\sG(k^\times), \cdot)$. \end{dfn}

Just as with the cycles, we are not interested in the order of the coordinates, but their image under $\rm{Alt}$. Therefore, we are also only interested in an alternating projection on the edges of the graphs. There is a $\fS_{n}\rtimes (\Z/2\Z)^n$ action on the edges of a graph $G \in \sG(k^\times)$. This action permutes the order of the edges in
the graph, and changes the assigned signs. An element $g \in
\fS_{n}\rtimes (\Z/2\Z)^n$ is of the form $g = (\sigma, \vec{\sgn})$,
where $\sigma \in \fS_{|E(G)|}$, and $\vec{sgn} \in (\Z/2\Z)^n$ is
  an ordered set of signs. Write $\vec{\sgn}_j$ for the $j^{th}$
  entry of the ordered set. Furthermore, write \bas \sgn(g) = \sgn(\sigma) \prod_j
  \vec{\sgn}_j \;, \eas where $\sgn(\sigma)$ indicates the sign of the permutation $\sigma  \in \fS_{|E(G)|}$.

The action of $\fS_{n}\rtimes (\Z/2\Z)^n$ on the algebra of graphs is as follows: \bas gG = \begin{cases} 0 & \textrm{if } |E(g)| \neq n \\ \begin{cases}\omega(gG) := \sigma(\omega(G)) & \\ \sgn_i(gG) = \vec{\sgn}_i
  \sgn_i (G) & \end{cases} & \textrm{else} \end{cases} \;. \eas That is, if $|E(G)| = n$ then
The ordering and signs of the edges in $gG$, for $g =
  (\sigma, \vec{\sgn})$, are determined by $\sigma$ and
  $\vec{\sgn}$ respectively.

The action of $\fS_{n}\rtimes (\Z/2\Z)^n$ defines an equivalence
class on $\Gprg$.

\begin{lem}Letting $n$ vary, any two monomials $G$ and $G' \in \Gprg$ are equivalent if
  and only if there is an element $g \in \fS_{n}\rtimes (\Z/2\Z)^n$
  relating the two: \bas G \simord \sgn(g) gG \;
  . \eas \label{orderequiv}\end{lem}

The proof comes from the identity, inverse and composition laws of the
group $\fS_{|E(G)|}\rtimes (\Z/2\Z)^n$, and we omit it.

In Lemma \ref{skewproduct}, we show that $\Gprg/\simord$ is generated as an algebra by connected graphs. In other words, under the equivalence $\simord$, any disconnected element of $\Gprg$ is no longer primitive.

First we give an example.

\begin{eg}

To illustrate the equivalence classes from Lemma \ref{orderequiv},
  consider the graph $G$ in example \ref{samplegraph} as a monomial in $\Gprg$.
\bas G = 
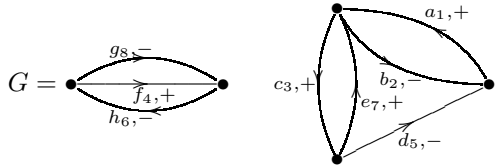
{\begin{xy}
(0, 0) *{\bullet}= "a",
(20, 0) *{\bullet}= "b",
(35,10) *{\bullet}="t",
(35,-10) *{\bullet}= "u",
(55, 0) *{\bullet}= "z",
"t"; "z" **\crv{+(-15,0)}?/0pt/*{\dir{>}} +(1,-2)*{_{b_2,-}},
"z"; "t" **\crv{+(15,0)}?/0pt/*{\dir{>}}+(2,2)*{_{a_1,+}},
"t"; "u" **\crv{+(-5,10)}?/0pt/*{\dir{>}} +(-3,0)*{_{c_3, +}},
"u"; "t" **\crv{+(5,-10)}?/0pt/*{\dir{>}} +(3.5,-2.5)*{_{e_7, +}},
"u"; "z" **{\dir{-}}?/0pt/*{\dir{>}}+(1,-2.5)*{_{d_5,-}},
"a"; "b" **\crv{+(-10,7)}?/0pt/*{\dir{>}} +(-2,1)*{_{g_8, -}},
"a"; "b" **\crv{+(-10,-7)}?/0pt/*{\dir{<}} +(-2,-1)*{_{h_6, -}},
"a"; "b" **{\dir{-}}?/0pt/*{\dir{>}}+(1,-1.5)*{_{f_4, +}},
\end{xy}} \;, \eas with the edges ordered as indicated by the
subscripts, as usual. This graph is a primitive element of $\Gprg$.

However, in the ring quotiented by the equivalence class, $\Gprg/\simord$, we see that $G \simord G_1 \cdot G_2$, where $G_1$ and $G_2$ are the graphs defined in example \ref{skewprodeg},
\bas G \simord  G_1 \cdot G_2 = {\begin{xy}
(0,10) *{\bullet}="t",
(0,-10) *{\bullet}= "u",
(20, 0) *{\bullet}= "z",
(30, 0) *{\bullet}= "a",
(50, 0) *{\bullet}= "b",
"t"; "z" **\crv{+(-15,0)}?/0pt/*{\dir{>}} +(-2,-1)*{b_2},
"z"; "t" **\crv{+(15,0)}?/0pt/*{\dir{>}}+(2,1)*{a_1},
"t"; "u" **\crv{+(-5,10)}?/0pt/*{\dir{>}} +(-2,0)*{c_3},
"u"; "t" **\crv{+(5,-10)}?/0pt/*{\dir{>}} +(-2,0)*{e_5},
"u"; "z" **{\dir{-}}?/0pt/*{\dir{>}}+(1,-1.5)*{d_4},
"a"; "b" **\crv{+(-10,7)}?/0pt/*{\dir{>}} +(2,1.5)*{g_6},
"a"; "b" **\crv{+(-10,-7)}?/0pt/*{\dir{<}} +(-2,-1)*{h_7},
"a"; "b" **{\dir{-}}?/0pt/*{\dir{>}}+(1,-1.5)*{f_8},
\end{xy}} \; ,
\eas
which is not primitive. Notice that both signs and orderings have been
changed in this example.
\label{prodeg}
\end{eg}

As an algebra, $\Gprg$ is bigraded by first Betti number, or weight, and degree
of the graphs. That is, if $G \in \Gprg_\bullet^\star$, then $h_1(G)
= \bullet$, while $\star = h_1(G) - V(G) + h_0(G)$. From the formula for the first Betti number of a graphs in \eqref{bettidef},  if $G \in \Gprg_\bullet^\star$, \ba |E(G)| = 2\bullet - \star \label{euler}\;. \ea

As the equivalence relation $\simord$ does not affect the underlying
topology of the graph, $\Gprg/\simord$ is also bigraded by weight and degree of the graphs.

\begin{rem}
The unit of this algebra is in $\Q[\sG]_0^0$. It is represented by the empty graph.
\end{rem}

\begin{eg}
For instance, consider the graph in examples \ref{samplegraph} and \ref{prodeg}.
\bas G ={\begin{xy}
(0, 0) *{\bullet}= "a",
(20, 0) *{\bullet}= "b",
(35,10) *{\bullet}="t",
(35,-10) *{\bullet}= "u",
(55, 0) *{\bullet}= "z",
"t"; "z" **\crv{+(-15,0)}?/0pt/*{\dir{>}} +(1,-2)*{_{b_2,-}},
"z"; "t" **\crv{+(15,0)}?/0pt/*{\dir{>}}+(2,2)*{_{a_1,+}},
"t"; "u" **\crv{+(-5,10)}?/0pt/*{\dir{>}} +(-3,0)*{_{c_3, +}},
"u"; "t" **\crv{+(5,-10)}?/0pt/*{\dir{>}} +(3.5,-2.5)*{_{e_7, +}},
"u"; "z" **{\dir{-}}?/0pt/*{\dir{>}}+(1,-2.5)*{_{d_5,-}},
"a"; "b" **\crv{+(-10,7)}?/0pt/*{\dir{>}} +(-2,1)*{_{g_8, -}},
"a"; "b" **\crv{+(-10,-7)}?/0pt/*{\dir{<}} +(-2,-1)*{_{h_6, -}},
"a"; "b" **{\dir{-}}?/0pt/*{\dir{>}}+(1,-1.5)*{_{f_4, +}},
\end{xy}}
\eas

This graph has five loops, five vertices and two connected
components. Therefore, it is in $\Gprg_5^2/\simord$.
\label{gradingeg} \end{eg}

\begin{dfn}
Let $\sG_0(k^\times) \subset \sG(k^\times)$ be the subset of
biconnected graphs with ordered, labeled, oriented edges. That is, there are no disconnected graphs in $\sG_0(k^\times)$.
\end{dfn}

\begin{lem}
The algebra $\Gprg/\simord$ is generated by the set $\sG_0(k^\times)/\simord$ as a skew symmetric
bigraded algebra.\label{skewproduct}
\end{lem}

\begin{proof}
For any disconnected graph $G \in
\Gprg_n^i$, there is an element $g
= (\sigma, \id) \in \fS_{2n-i} \rtimes (\Z/2\Z)^n$ that rearranges the
order of the edges of each connected component consecutively. Since $\sgn(g) =
\sgn(\sigma)$, by Lemma \ref{orderequiv}, \bas G \simord \sgn(g) (gG) =
\sgn(g) G_1\cdot G_2 \cdot \ldots G_m \;, \eas
with each $G_i \in \sG_0(k^\times)$.

The product preserves the bigrading as the $0^{th}$ and first Betti numbers are additive under disjoint union, as are the sizes of the edge and vertex sets. For $G\in \Q[\sG]_{n}^{i}/\simord$ and $G'\in \Q[\sG]_{n'}^{i'}/\simord$, we have \bas G \cdot G' \in \Q[\sG]_{n+n'}^{i+i'}/\simord \;.\eas

To see that this is skew symmetric, as above, write
\bas G\cdot G' \simord (-1)^{|E(G)||E(G')|} G'\cdot G = (-1)^{ii'} G' \cdot G\;.\eas The last equality comes from the fact that
$|E(G)| = 2n - i$ and $|E(G')| = 2n' - i'$.
\end{proof}

Since \bas \Gprg/\simord = \Q[\sG_0(k^\times)] /\simord\;,\eas
for the rest of this paper, we consider only elements of $\sG_0(k^\times)$.

\subsection{A brief interlude on algebraic cycles\label{homomorphism}}

In this section we introduce the relationship between the graphs defined above and algebraic cycles generating $\cZoneL{p}{n}$. As of yet, we make no claims on admissibility of cycles.

\begin{dfn} Define $\QZ$ to be the group ring generated by the free abelian group of $\bP^1$-linear cycles \bas \QZ =\bigoplus_{p, i}\Alt \cZoneL{p}{2p-i} \otimes \Q \;.\eas This is a skew symmetric algebra. Write $\QZ_p^i = \Alt \cZoneL{n}{2p-i} \otimes \Q$. \end{dfn}

There is a homomorphism, $Z$, from $\Gprg_\bullet^\star/\simord$ to $\QZ$. Note that $\Asm \subset \QZ$. In Section \ref{graphstocycles}, we show that $Z$ is a DGA homomorphism onto $\Asm$, that becomes an isomorphism of DGAs when $\Gprg_\bullet^\star$ is subjected to more equivalence relations. That is the isomorphism we seek in this paper. In this section, we show that elements of $\Gprg$ correspond to parameterizations of $\bP^1$-linear algebraic cycles on $\square^{|E(G)|}_k$.

\begin{dfn}
Each connected graph, $G \in \sG(k^\times)$, with loop number $p$ and $n$ edges defines a parametrization, $\phi: \bP^{|V(G)|-1} \rightarrow (\bP^1)^n$, of an algebraic cycle $Z(G) \in \cZoneL{p}{n}$.
The $\omega(e)^{th}$ coordinate of the cycle $Z(G)$ is
\bas \phi_{\omega (e)} = \left( 1 - \frac{x_{s(e)}}{a_e
  x_{t(e)}}\right)^{\sgn_{\omega(e)}} \;, \eas where $x_{s(e)}$ and
$x_{t(e)}$ are variables assigned to the vertices at the source and
target of the edge $e \in E(G)$, and $a_e$ is the label of edge
$e$.  \label{maptocycle}
\end{dfn}

Thus we have, for $\bullet = h_1(G)$ and $\star = h_1(G) - |V(G)| + h_0(G)$, a set map \ba Z: \sG(k^\times) & \rightarrow
\bigoplus_{\bullet, \star}\cZoneL{\bullet}{2\star-\bullet} \nonumber \\ G & \rightarrow [\phi_1 | \ldots |
  \phi_{|E(G)|}] \; \label{cyclemap} \ea from graphs to parametrized $\bP$-linear cycles.

To make this map concrete, we explicitly derive the system of polynomials defined by a graph $G$. First we introduce a function that relates edges of a graph to the loops of
$G$.

\begin{dfn}
For $e \in E(G)$, and $L$ a loop of $G$, define \bas \epsilon(e, L)
= \begin{cases} 1 & \textrm{if $e \in E(L)$ oriented as $L$ is,} \\ 0
  & \textrm{if $e \not \in E(L)$,} \\ -1 & \textrm{if $e \in E(L)$,
    oriented opposite to $L$ .} \end{cases}
\eas \label{loopcharacteristic}
\end{dfn}

Given this notation, we are ready to define the system of polynomials defined by a graph $G \in \sG_0(k^\times)$.

\begin{thm}
 For a graph $G \in \sG_0(k^\times)$ indicate the label of the edge $e \in E(G)$ as $a_e \in k^\times$. Suppose $h_1(G) = p$, and $|E(G)|= n$. Let $\beta = \{L_1 \ldots L_p\}$ be a loop basis of $H_1(G)$. The algebraic cycle $Z(G)$ is defined by the
  system of $p$ polynomial equations, each associated to an element of
  the loop basis as follows: \ba 1 = \prod_{e\in E(G)} (a_e (1-
  \phi_{\omega(e)}))^{\epsilon(e, L_i)} \;. \label{systemeqeg} \ea
\label{systemeqthm}\end{thm}

\begin{proof}
Given $\beta$, a loop basis for $H_1(G)$, begin with a loop, $L_1$. Subsequent elements of the
system of equations are similarly defined.

Consider an edge, $e \in E(L_1)$. The $\omega(e)^{th}$ coordinate of
the cycle $Z(G)$ is defined by the function
$\phi_{\omega(e)}(x,y)$, where $x$ and $y$ are the variables associated
to the vertices at the endpoints of $e \in E(G)$. Suppose that in the
orientation inherent in $L_1$ as an element of a loop basis, $L_i$
flows from the vertex associated to $x$ directly to the vertex
associated to $y$. This is not necessarily the orientation of the edge
connecting the vertices associated to $x$ and $y$, but the second
orientation on the edges induced by the orientation of $L_1$. Then one
can associate to the edge $e \in E(G)$, the equation \ba x = y
(a_e(1-\phi_{\omega(e)}))^{\epsilon(e, L_1)} \;. \label{coordeq} \ea

There is an unique edge in $L_1$, $e' \neq e$, with an endpoint at the
vertex associated to the variable $y$. As above, associate to the edge
$e'$ the equation \bas y = z (a_{e'}(1-\phi_{\omega(e')}))^{\epsilon(e',
  L_1)} \;.  \eas Substituting this into \eqref{coordeq} gives \bas x
= z(a_e(1-\phi_{\omega(e)}))^{\epsilon(e,
  L_1)}(a_{e'}(1-\phi_{\omega(e')}))^{\epsilon(e', L_1)}\;. \eas
Continuing along the entire loop in this manner gives \bas x= \prod_{e
  \in E(G)} (a_e(1-\phi_{\omega(e)}))^{\epsilon(e, L_1)} x \;, \eas which
simplifies to an expression of the form \eqref{systemeqeg} \bas 1 =
\prod_{e \in E(G)} (a_e(1-\phi_{\omega(e)}))^{\epsilon(e, L_1)} \;. \eas

Since $\beta$ is a loop basis, the function $\phi_{\omega(e)}$,
associated to each edge of $G$ is used in the system of equations
defined in \eqref{systemeqeg}, and the functions thus derived are
independent of each other.
\end{proof}

Notice that the specific form of this system of equations depends on the loop basis for $H_1(G)$. However, a different loop basis will give an equivalent system of polynomials.

\begin{eg}
Recall the graph in example \ref{samplegraph}:

\bas {\begin{xy}
(0,10) *{\bullet}="t",
(0,-10) *{\bullet}= "u",
(20, 0) *{\bullet}= "z",
"t"; "z" **\crv{+(-15,0)}?/0pt/*{\dir{>}} +(-2,-1)*{r_2},
"z"; "t" **\crv{+(15,0)}?/0pt/*{\dir{>}}+(2,1)*{r_1} ,
"t"; "u" **\crv{+(-5,10)}?/0pt/*{\dir{>}} +(-2,0)*{r_3},
"u"; "t" **\crv{+(5,-10)}?/0pt/*{\dir{>}} +(-2,0)*{r_5},
"u"; "z" **{\dir{-}}?/0pt/*{\dir{>}}+(1,-1.5)*{r_4},
\end{xy}} \;.\eas

Define a basis \bas \beta = \
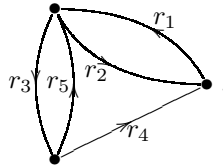
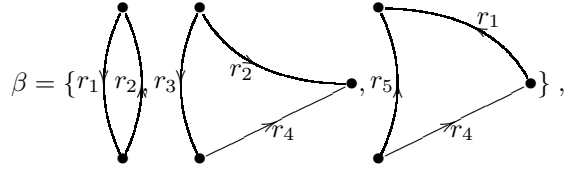
{ {\begin{xy}
(0,10) *{\bullet}="t",
(0,-10) *{\bullet}= "u",
"t"; "u" **\crv{+(-5,10)}?/0pt/*{\dir{>}} +(-2,0)*{r_1},
"u"; "t" **\crv{+(5,-10)}?/0pt/*{\dir{>}} +(-2,0)*{r_2},
\end{xy}}, {\begin{xy}
(0,10) *{\bullet}="t",
(0,-10) *{\bullet}= "u",
(20, 0) *{\bullet}= "z",
"t"; "z" **\crv{+(-15,0)}?/0pt/*{\dir{>}} +(-2,-1)*{r_2},
"t"; "u" **\crv{+(-5,10)}?/0pt/*{\dir{>}} +(-2,0)*{r_3},
"u"; "z" **{\dir{-}}?/0pt/*{\dir{>}}+(1,-1.5)*{r_4},
\end{xy}}, {\begin{xy}
(0,10) *{\bullet}="t",
(0,-10) *{\bullet}= "u",
(20, 0) *{\bullet}= "z",
"z"; "t" **\crv{+(15,0)}?/0pt/*{\dir{>}}+(2,1)*{r_1} ,
"u"; "t" **\crv{+(5,-10)}?/0pt/*{\dir{>}} +(-2,0)*{r_5},
"u"; "z" **{\dir{-}}?/0pt/*{\dir{>}}+(1,-1.5)*{r_4},
\end{xy}}
\}\;,\eas where all the loops in are oriented counterclockwise.

A system of equations for this graph is given by the polynomials \bas 1  & = r_1 r_2 (1-f_1) (1-f_2)  \\ 1  & =  \frac{r_3 r_4}{r_2} \frac{(1-f_3)(1-f_4)}{(1-f_2)} \\  1  & =  \frac{r_1 r_4}{r_5} \frac{(1-f_1)(1-f_4)}{(1-f_5)} \;. \eas

\end{eg}

This brings us to an important invariant of the graphs in $\Gprg_\bullet^\star/\simord$, the loop coefficient.

\begin{dfn}
Given a loop of $G$, the loop coefficient of a loop $L$ of $G$ is defined \ba \chi_G(L) = \prod_{E(G)} r_e^{\epsilon(e, L)} \; .
\ea
\label{loopcoef} \end{dfn}

In this notation, we can restate the image of the map $Z$. For $G \in \Gprg_n^p/\simord$ with basis $\beta = \{L_1, \ldots L_p\}$ a basis of $H^1(G)$, the cycle $Z(G)$ is defined by \ba \{1 =   \chi_G(L_i ) \prod_{e \in E(L_i)}(1-
  \phi_{\omega(e)})^{\epsilon(e, L_i)} \}_{L_i \in \beta} \;.\label{systemeqloopcoef}\ea

We can extend the set map $Z$ thus defined to the algebra $\Gprg_\bullet^\star/\simord$, where $Z(G)$ maps a graphs to an algebraic cycle
under the alternating projection.

\begin{thm}The set map $Z$ in \eqref{cyclemap} induces a grading preserving algebra homomorphism \bas Z: \Gprg/\simord & \rightarrow
\QZ \nonumber \\ G & \rightarrow \Alt[\phi_1 | \ldots |
  \phi_{|E(G)|}] \;.\eas \label{Zhomo}\end{thm}

\begin{proof}
The equivalence relation $\simord$ equates different orderings of edges of graphs as $\Alt$ combines different orderings of coordinates into a single generator of $\QZ$. Therefore, $Z$ maps generators of $\Gprg/\simord$ to generators of $\QZ$. Lemma \ref{skewproduct} shows that the algebra structure of
$\Gprg/\simord$ matches the algebra structure of $\QZ$.

It remains to check that if $G \in \Q[\sG]_p^i/\simord$, then $Z(G) \in \QZ_p^i$. First notice that by the parametrization given in definition \ref{maptocycle}, Writing $G = G_1\cdot \ldots \cdot G_m$ in terms of its connected
components, the cycle $Z(G)$ is
parameterized by the map \bas \phi : \prod_{i=1}^m\bP^{|V(G_i)|-1}
 \rightarrow \square_k^{|E(G)|} \;.\eas Therefore, the cycle $Z(G)$ has codimension \bas E(G) - V(G) + h_0(G) = h_1(G) = p \eas in $\square_k^{|E(G)|}$. By equation \eqref{euler} this implies that $Z(G) \in \cZoneL{p}{2p-i}$.
\end{proof}

Finally, in conjuction with Theorem \ref{systemeqthm}, this allows for a statement about reducible cycles.

\begin{cor}
If $G$ is a disconnected graph in $\Gprg/ \simord$, then $Z(G)$ is a reducible cycle. \label{reduciblecor}
\end{cor}

\subsection{The DGA structure on graphs\label{graphDGA}}

In this section, we define a differential structure on the algebra of graphs. In order to do this, we need to define an additional equivalence relation on $\sG_0(k^\times)$.

In particular, we consider graphs that differ only by a rescaling of the labels of the edges attached to a vertex.

\begin{dfn}
Consider $\alpha \in k^\times$, and $v \in V(G)$. The vertex rescaled graph, $v_\alpha(G)$, is the labeled oriented graph $G$ with labels
changed as follows: for each edge $e$ of $G$, if an edge terminates
(starts) at $v$, multiply (divide) its label by $\alpha$ to get the
label of the edge in $v_\alpha(G)$; otherwise, keep the same label for
$e$.  The signs associated to, and the ordering of the edges of $G$ by
$\omega$ do not change. \label{rescaledfn}
\end{dfn}

Vertex rescaling a graph corresponds to rescaling all instances of a variable in the parametrized $\bP$-linear cycle $Z(G)$ by a constant multiple. This does not effect the cycle at all. In otherwords, $G$ and $v_\alpha(G)$ correspond to two different parameterizations of $Z(G)$. We call this procedure label rescaling with respect to a vertex, or label rescaling at $v$.

\begin{eg}
For the graph $G$ in example \ref{samplegraph}, one can rescale the
rightmost vertex by $\alpha$ to obtain the following graph
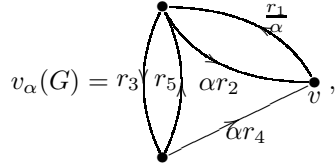
\bas v_\alpha(G) = {\begin{xy}
(0,10) *{\bullet}="t",
(0,-10) *{\bullet}= "u" ,
(20, 0) *{\bullet}= "z" +(0,-2) *{v},
"t"; "z" **\crv{+(-15,0)}?/0pt/*{\dir{>}} +(0,-3)*{\alpha r_2},
"z"; "t" **\crv{+(15,0)}?/0pt/*{\dir{>}}+(2.5,1)*{\frac{r_1}{\alpha}} ,
"t"; "u" **\crv{+(-5,10)}?/0pt/*{\dir{>}} +(-2,0)*{r_3},
"u"; "t" **\crv{+(5,-10)}?/0pt/*{\dir{>}} +(-2,0)*{r_5},
"u"; "z" **{\dir{-}}?/0pt/*{\dir{>}}+(1,-2)*{\alpha r_4},
\end{xy}} \;,
 \eas where the ordering of the edges is given by the subscripts.

\end{eg}

\begin{rem}
Vertex rescaling is an equivalence relation on the set
$\sG_0(k^\times)$. We write it $\simv$. \end{rem}

In the sequel, we consider the algebra of graphs up to this equivalence set. We are interested in graphs only as a tool to understand their corresponding algebraic cycles. We work with graphs up to this rescaling since two graphs that differ by a vertex rescaling correspond, under the homomorphism $Z$ defined in Section \ref{homomorphism}, to different parameterizations of the same cycle.

To see this, notice that vertex rescaling does not change the loop coefficient of the graph.

\begin{lem}
Loop coefficients are invariant under rescaling at vertices.
\label{rescalinginvariance}\end{lem}

\begin{proof}
Let $L$ be a loop in $G$, with $G \in \sG_0(k^\times)$. For $v \in
V(L)$, a vertex in $L$, $v$ is attached to exactly two of edges of
$L$, $e_1$ and $e_2$. We compare $\chi_G(L)$ and $\chi_{v_\alpha(G)}(L)$.

If $v$ is the terminal vertex of $e_1$ and the source vertex of $e_2$,
then the respective coefficients are $r_1$ and $r_2$ in $G$, and
$r_1 \alpha$ and $\frac{r_2}{\alpha}$ in $v_\alpha(G)$. Both
numbers either appear in the numerator or the denominator of the
coefficient of $L$. Thus the contributions of $\alpha$ cancel in
$\chi_{v_\alpha(G)}(L)$.

The other two cases are as follows. The vertex $v$ is either the
source (target) vertex of both $e_1$ and $e_2$. Then the coefficients
are $\frac{r_1}{\alpha}$ and $\frac{r_2}{\alpha}$ ($r_1
\alpha$ and $r_2 \alpha$). One label appears in the numerator of
the loop, the other in the denominator, so the contribution of
$\alpha$ cancels $\chi_{v_\alpha(G)}(L)$.

Therefore \bas \chi_G(L) = \chi_{v_\alpha(G)}(L)\eas as desired.
\end{proof}

Therefore, given form of the system of polynomials defined by each of these graphs, $Z(G) = Z(v_\alpha(G))$.

\begin{thm} The parametrized cycles $Z(G)$ and $Z(v_\alpha(G))$ correspond to the same cycle, under different parametrizations \bas Z(G) = Z(v_\alpha(G)) \in \cZoneL{h_1(G)}{|E(G)|} \;.\eas \label{vertexscalesamecycle}\end{thm}

\begin{proof}
Let $\phi(G)$ be the paremetrization defined by the graph $G$ of the cycle $Z(G)$, and $\phi(v_\alpha (G))$ be the paremetrization defined by the graph $v_\alpha (G)$ of the cycle $Z(v_\alpha (G))$.

From Definition \eqref{maptocycle}, we see that the $\omega(e)^{th}$ coordinate of $\phi(v_\alpha (G))$ is \bas \phi_{\omega(e)}(v_\alpha (G)) = \begin{cases} \onem{s_e}{a_et_e} &  \textrm{if }v \not \in \{s_e, t_e\} \\ \onem{\alpha s_e}{a_e t_e} &  \textrm{if } v =s_e  \\\onem{ s_e}{\alpha a_e t_e} &  \textrm{if } v =t_e  \end{cases} \;.\eas  Since the loop coefficients are the same, and the only change in the coordinates is to replace each occurance of $v$ with $\alpha v$, the procedure in Theorem \ref{systemeqthm} gives the same system of equations for both $Z(G)$ and $Z(v_\alpha(G))$.
\end{proof}

Therefore, the algebra homorphism $Z$ defined in Section \ref{homomorphism} passes to an alegebra homomorphism under the quotient $\simv$: \bas Z: \Gprg_\bullet^\star/ (\simord,\simv) \rightarrow \Alt \cZoneL{\bullet}{2\bullet-\star} \;. \eas

As we mentioned before, the algebra $\Alt \cZoneL{\bullet}{2\bullet-\star}$ does not have a DGA structure. However, the algebra $\Q[\sG]/ (\simord,\simv)$ does. On individual graphs, this is defined by a modified contraction of the edges. We devote the rest of this section to develop this differential.

\begin{dfn}
Consider $G \in \sG_0(k^\times)$. For $e \in E(G)$, define the graph $G / e$ to be that formed by contracting the edge $e$ and identifying the
  vertices $s_e$ and $t_e$ as a new vertex $v$. If the edge $e \in
  E(G)$ has the same source and target vertex, then $G/ e = 0$. If
  contracting the edge $e$ leads to a one connected graph, split the
  graph into its biconnected components at the articulation vertex. \label{contractdef}
\end{dfn}

The above definition is not the standard definition of an edge
contraction in graphs. The standard definition has been modified to
fit the algebraic properties of the graphs we need in this paper, namely the splitting of graphs at the articulation vertex. Furthermore, the ordering of $G/e \in \Gprg$ is induced from the ordering of $G$.

\begin{dfn}
  Let $\omega(G)$ be the ordering of the edges of the graph $G$. Then $\hat\omega_e(G\setminus e)$ is the ordering of the graph $G \setminus e$ which is the same as $\omega(G)$ with the $\omega(e)^{th}$ element removed.
\end{dfn}

We are now ready to define a differential operator on $G \in \Q[\sG]_\bullet^\star/ (\simord,\simv)$.

\begin{thm}
  Consider a monomial $G \in \Q[\sG]_\bullet^\star/ (\simord,\simv)$. Let $e \in E(G)$ be an edge with label $r_e$ and source vertex $s_e$. There is a degree $1$ differential operator \bas \D: \Q[\sG]_\bullet^\star/ (\simord,\simv) &
  \rightarrow \Q[\sG]_\bullet^{\star+1}/ (\simord,\simv) \\ G & \rightarrow \sum_{e\in E(G)}
  (-1)^{\omega(e)-1} (\hat\omega_e, ((s_{e})_{r_e}(G))/e) \;. \eas \label{diffdef} \end{thm}

By direct calculation, one sees that the Leibnitz rule for this operator is
 \ba \D(G \cdot G') = \D(G) \cdot G' + (-1)^{\star} G \cdot \D (G') \; . \label{Leibniz}\ea

We prove this theorem in steps. Before starting the proof, we give an
example of the action of $\D$. Recall that the notation $(s_{e})_{r_e}$ in Theorem \ref{diffdef} is the vertex rescaling from definition \ref{rescaledfn}.

\begin{eg}
For example, for the graph in example \ref{samplegraph}, with
  $\omega$ ordered according to the numbering of the labels,
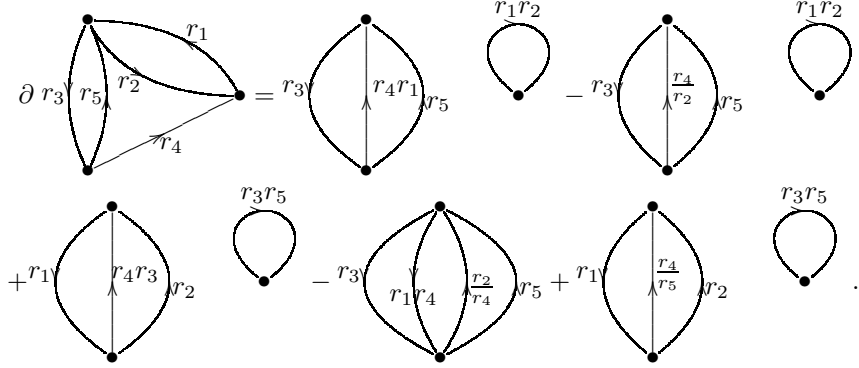
\bas \D \; {\begin{xy}
(0,10) *{\bullet}="t",
(0,-10) *{\bullet}= "u",
(20, 0) *{\bullet}= "z",
"t"; "z" **\crv{+(-15,0)}?/0pt/*{\dir{>}} +(-2,-1)*{r_2},
"z"; "t" **\crv{+(15,0)}?/0pt/*{\dir{>}}+(2,1)*{r_1} ,
"t"; "u" **\crv{+(-5,10)}?/0pt/*{\dir{>}} +(-2,0)*{r_3},
"u"; "t" **\crv{+(5,-10)}?/0pt/*{\dir{>}} +(-2,0)*{r_5},
"u"; "z" **{\dir{-}}?/0pt/*{\dir{>}}+(1,-1.5)*{r_4},
\end{xy}}=
 {\xy
(0,10) *{\bullet}="t",
(0,-10) *{\bullet}= "u",
(20, 0) *{\bullet}= "z",
"t"; "u" **\crv{+(-15,10)}?/0pt/*{\dir{>}} +(-2,1)*{r_3},
"u"; "t" **\crv{+(15,-10)}?/0pt/*{\dir{>}}+(2,-1)*{r_5},
"t"; "u" **{\dir{-}}?/0pt/*{\dir{<}} +(4,1)*{r_4r_1},
"z"; "z" **\crv{+(-7,5)&+(7,9)&+(7,-9)}?/0pt/*{\dir{>}}+(0,2)*{r_1r_2},
\endxy} -
 {\xy
(0,10) *{\bullet}="t",
(0,-10) *{\bullet}= "u",
(20, 0) *{\bullet}= "z",
"t"; "u" **\crv{+(-13,10)}?/0pt/*{\dir{>}} +(-2,1)*{r_3},
"u"; "t" **\crv{+(13,-10)}?/0pt/*{\dir{>}}+(2,-1)*{r_5},
"t"; "u" **{\dir{-}}?/0pt/*{\dir{<}} +(2,1)*{\frac{r_4}{r_2}},
"z"; "z" **\crv{+(-7,5)&+(7,9)&+(7,-9)}?/0pt/*{\dir{>}}+(0,2)*{r_1r_2},
\endxy}
 \\ +
 {\xy
(0,10) *{\bullet}="t",
(0,-10) *{\bullet}= "u",
(20, 0) *{\bullet}= "z",
"t"; "u" **\crv{+(-15,10)}?/0pt/*{\dir{>}} +(-2,1)*{r_1},
"u"; "t" **\crv{+(15,-10)}?/0pt/*{\dir{>}}+(2,-1)*{r_2},
"t"; "u" **{\dir{-}}?/0pt/*{\dir{<}} +(3,1)*{r_4r_3},
"z"; "z" **\crv{+(-7,5)&+(7,9)&+(7,-9)}?/0pt/*{\dir{>}}+(0,2)*{r_3r_5},
\endxy}
 -{\xy
(0,10) *{\bullet}="t",
(0,-10) *{\bullet}= "u",
"t"; "u" **\crv{+(-7,10)}?/0pt/*{\dir{>}} +(0,-2)*{r_1r_4},
"u"; "t" **\crv{+(7,-10)}?/0pt/*{\dir{>}}+(2,-1)*{\frac{r_2}{r_4}},
"t"; "u" **\crv{+(-20,10)}?/0pt/*{\dir{>}} +(-2,1)*{r_3},
"u"; "t" **\crv{+(20,-10)}?/0pt/*{\dir{>}}+(2,-1)*{r_5},
\endxy}
 +
 {\xy
(0,10) *{\bullet}="t",
(0,-10) *{\bullet}= "u",
(20, 0) *{\bullet}= "z",
"t"; "u" **\crv{+(-13,10)}?/0pt/*{\dir{>}} +(-2,1)*{r_1},
"u"; "t" **\crv{+(13,-10)}?/0pt/*{\dir{>}}+(2,-1)*{r_2},
"t"; "u" **{\dir{-}}?/0pt/*{\dir{<}} +(2,1)*{\frac{r_4}{r_5}},
"z"; "z" **\crv{+(-7,5)&+(7,9)&+(7,-9)}?/0pt/*{\dir{>}}+(0,2)*{r_3r_5},
\endxy} \;.
\eas

\end{eg}

First, we define a contraction operator on graphs with labeled edges.

\begin{dfn}
For $e \in E(G)$, we write the contraction of an edge $\D_e (G) = (s_e)_{r_e}(G)/ e$.
\label{Desourcedef}\end{dfn}

In this notation, the operator defined in Theorem
\ref{diffdef} can be rewritten \bas \D (G)= \sum_{e \in E(G)} (-1)^{\omega(e)-1}\D_e (G) \;.\eas
Notice that if $r_e= 1$ then $\D_e (G) = (G/ e)$. This further implies that the loop coefficient is invariant under contraction.

\begin{lem}
Consider $G \in \Gprg/(\simord, \simv)$. Let $L$ be a loop in $G$ with more than one edge, and $e \in E(L)$. Then \bas \chi_G(L) = \chi_{\D_e G}(L/e) \;.\eas
\end{lem}

\begin{proof}
It is sufficient to consider $G$ connected. If $s$ is the source vertex of $e$, and $r$ the label, the equivalent graph $s_r(G)$ is such that the label of $e = 1$.

In Lemma \ref{diffreparam} below, we show that contraction is well defined on  $\Gprg/(\simord, \simv)$. Therefore, $\D_e(G) \simv \D_e(s_r(G))$. Since the label of $e$ is $1$, the contraction $ \D_e(s_r(G)) = s_r(G) / e$, and
\bas \chi_G(L) = \chi_{s_r(G)}(L) = \chi_{\D_e s_r(G)}(L\setminus e) = \chi_{\D_e G}(L\setminus e)\;.\eas  The first equality comes from Lemma \ref{rescalinginvariance}. The second equality comes from the form of $\D_e(s_r(G))$. Finally, the third equality comes from the equivalence of the two contractions (Lemma \ref{diffreparam}).
\end{proof}

Working under the equivalence relations $\simv$ gives an important representation of graphs that simplifies the calculation of the derivatives.

\begin{lem}
For any given $G \in \sG_0(k^\times)$, and any subtree $T \subset G$, there is a tree $G_T$ such that the labels of the edges in $T$ are $1$ and $G \simv G_T$.  In particular, any
monomial $G \in \Q[\sG]/(\simord, \simv)$ can be rescaled such that any
spanning forest of $G$ is labeled by
$1$. \label{rescaletree}
\end{lem}

\begin{proof}
Without loss of generality, assume that the graph $G \in \sG_0(k^\times)$
is a connected graph. Otherwise, the following arguments apply to each
connected component of $G$.

Let $T$ be a spanning tree of $G$. Label the vertices $\{v_1, \ldots
v_{|V(G)|}\} \in V(G)$ such that $v_1$ has valence $1$ in $T$. Let $\{r_2
\ldots r_{|V(T)|}\}$ be the labels of the edges in $E(T)$, where
$r_i$ labels the edge connected to $v_i$.

Rescale the graph $G$ at the vertex $v_2$ by $r_2$ (resp. $1/r_2$) if
$v_2$ is a source (resp. target) vertex of the edge labeled by
$r_2$. In the rescaled graph $(v_2)_{r_2}(G)$
(resp. $(v_2)_{1/r_2}(G)$) the edge connecting $v_1$ and $v_2$ is
labeled by 1. By similar logic, there is a series of rescaling'
coefficients, $\{ \alpha_1, \ldots , \alpha_{|V(G)-1}\}$, where each
$\alpha_i$ is a rational function of the $r_js$ such that edges of the
spanning tree $T$ in $(v_{|V(G)|-1})_{\alpha_{|V(G)-1}}( \ldots
((v_1)_{\alpha_1}(G)) \ldots )$ are all labeled by $1$.
\end{proof}

\begin{eg}
Consider again, the graph in example \ref{samplegraph}. The loop coefficient of the loop defined by the inner triangle of legs, oriented clockwise is $\frac{r_2r_5}{r_4}$. The same graph can be relabeled to have a spanning tree labeled with ones as follows.

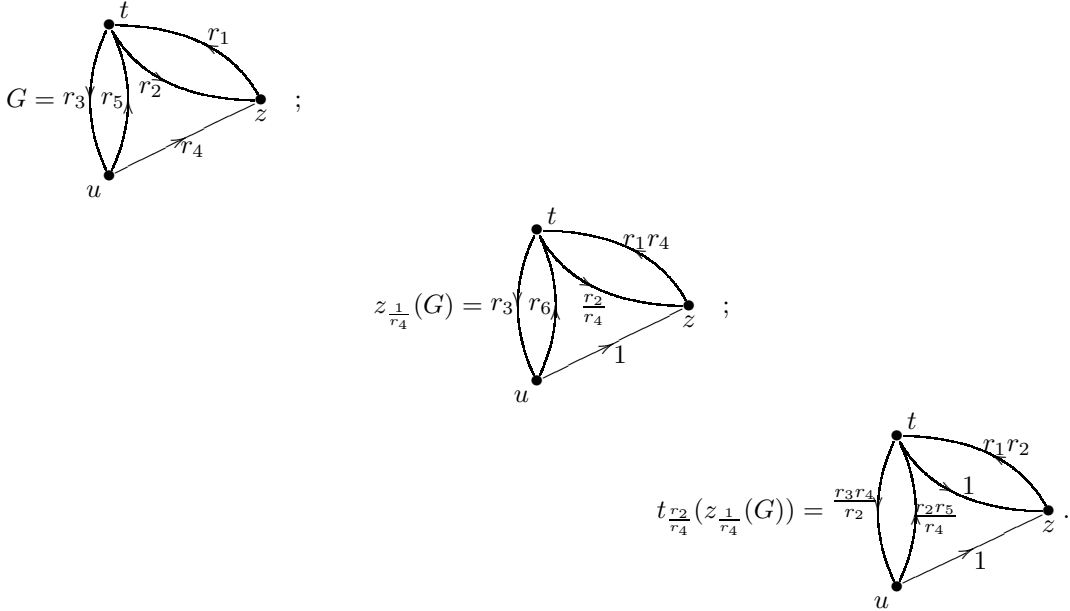
\bmls G = {\begin{xy}
(0,10) *{\bullet}="t" +(2,2)*{t},
(0,-10) *{\bullet}= "u" +(-2,-2)*{u},
(20, 0) *{\bullet}= "z" +(0,-2)*{z},
"t"; "z" **\crv{+(-15,0)}?/0pt/*{\dir{>}} +(-2,-1)*{r_2},
"z"; "t" **\crv{+(15,0)}?/0pt/*{\dir{>}}+(2,1)*{r_1} ,
"t"; "u" **\crv{+(-5,10)}?/0pt/*{\dir{>}} +(-2,0)*{r_3},
"u"; "t" **\crv{+(5,-10)}?/0pt/*{\dir{>}} +(-2,0)*{r_5},
"u"; "z" **{\dir{-}}?/0pt/*{\dir{>}}+(1,-1.5)*{r_4},
\end{xy}}  \quad ;  \\ \quad z_{\frac{1}{r_4}}(G) =
{\begin{xy}
(0,10) *{\bullet}="t" +(2,2) *{t},
(0,-10) *{\bullet}= "u" +(-2,-2) *{u},
(20, 0) *{\bullet}= "z" +(0,-2) *{z},
"t"; "z" **\crv{+(-15,0)}?/0pt/*{\dir{>}} +(0,-3)*{\frac{r_2}{r_4}},
"z"; "t" **\crv{+(15,0)}?/0pt/*{\dir{>}}+(2,1)*{r_1 r_4} ,
"t"; "u" **\crv{+(-5,10)}?/0pt/*{\dir{>}} +(-2,0)*{r_3},
"u"; "t" **\crv{+(5,-10)}?/0pt/*{\dir{>}} +(-2,0)*{r_6},
"u"; "z" **{\dir{-}}?/0pt/*{\dir{>}}+(1,-1.5)*{1},
\end{xy}} \quad ; \\  t_{\frac{r_2}{r_4}}(z_{\frac{1}{r_4}}(G)) =
{\begin{xy}
(0,10) *{\bullet}="t" +(2,2) *{t},
(0,-10) *{\bullet}= "u" +(-2,-2) *{u},
(20, 0) *{\bullet}= "z" +(0,-2) *{z},
"t"; "z" **\crv{+(-15,0)}?/0pt/*{\dir{>}} +(2,1)*{ 1},
"z"; "t" **\crv{+(15,0)}?/0pt/*{\dir{>}}+(2,1)*{r_1 r_2} ,
"t"; "u" **\crv{+(-5,10)}?/0pt/*{\dir{>}} +(-3,1)*{\frac{r_3 r_4}{r_2}},
"u"; "t" **\crv{+(5,-10)}?/0pt/*{\dir{>}} +(2.5,-1)*{\frac{r_2r_5}{r_4}},
"u"; "z" **{\dir{-}}?/0pt/*{\dir{>}}+(1,-1.5)*{1},
\end{xy}} \;.
\emls

\label{spanningtreeexample}
\end{eg}

Contrary to appearance, we have made no choice in our definition of the differential operator
$\D_e$. We could just as easily have written \bas \D_e (G) =
(-1)^{\omega(e) -1}(t_{e})_{1/r_e}(G)/e \;, \eas where $t_e$ is the target vertex of the edge $e$. This is because
the two graphs are equivalent under under vertex
rescaling.

\begin{lem}
For $G \in \Q[\sG]$, let $t$ and $s$ be the target and source vertices, respectively, of the edge $e \in E(G)$. Then \bas t_{1/a}(G)/ e  =
 s_{a}(G)/ e \eas in $\Q[\sG]/ (\simord,\simv)$. \label{contractvertindep}
\end{lem}

\begin{proof}
We show that there is a vertex rescaling such that \bas t_{1/a}
(G)/ e \sim s_{a}(G)/ e\; .\eas By construction, the edge $e
\not \in E(G/ e)$, and the vertices $t_e, s_e \in V(G)$ are are replaced
by a single vertex $v \in V(G /e )$.

In the graph $t_{1/a}(G)$, the label of $e$ is multiplied by
$\frac{1}{a}$, as are all the edges terminating on $t_e$. All edges
starting at $t_e$ are multiplied by $a$. The edges attached to $s_e$
and not $t_e$ are unaffected. Similarly, in the graph
$s_{a}(G)$, the label of $e$ is multiplied by $\frac{1}{a}$, as
are all the edges starting at $s_e$. All edges terminating at $s_e$
are multiplied by $a$. The edges attached to $t_e$ and not $s_e$ are
unaffected.

Therefore, contracting $e$ and identifying $s$ with $t$ at the new vertex in the contracted graph, $v = V(G /e) \setminus V(G)$, \bas
v_{1/a} (s_{a}(G)/ e) = t_{1/a}(G)/ e \;.\eas Similarly,
one may also write \bas s_{a}(G)/ e =v_a(t_{1/a}(G)/ e)
\;.\eas
\end{proof}

Choosing $a= r_e$, the label of the edge $e$, shows
that, in $\Q[\sG]/ (\simord,\simv)$, it does not
matter if $\D_e$ is defined according to the source vertex of $e$ or
the target vertex.

Next we show that the contraction operator $\D_e$ is well defined under
vertex rescaling.

\begin{lem}
The operator $\D$ defined above is well defined on $\Q[\sG]/ (\simord,\simv)$. \label{diffreparam}\end{lem}

\begin{proof}
Since $\D = \sum_{e \in E(G)}(-1)^{\omega(e)}\D_e$, for any $g \in \fS{|E(G)|} \rtimes \Z/2\Z^{|E(G)|}$, \bas \D G = \D gG \eas for all $G \in \sG_0(k^\times)$.

It remains to check that, for $G \in \sG_0(k^\times)$, \ba \D
(G) \sim \D (v_\alpha(G)) \label{diffsim} \ea for any
$v \in V(G)$. Before proceeding, we note that vertex rescaling is multiplicative. That is, for $v \in V(G)$, \ba v_\alpha (v_\beta (G)) = v_{\alpha\beta} (G) \label{vertexmult} \;. \ea

Fix $v \in V(G)$. For any edge $e$ not incident upon $v$, \bas
\D_e(v_\alpha G)= v_\alpha \D_e (G) \sim \D_e (G) \;.\eas Therefore consider only the edges $e \in E(G)$ that are incident upon $v$. They are labeled by $r_e$. By Definition \ref{Desourcedef} and Lemma
\ref{contractvertindep}, \bas \D_e(G)
\simv \begin{cases}v_{r_e}(G)/e  & v \textrm{ source of
  } e \\ v_{1/r_e}(G)/e & v \textrm{
    target of } e \;. \end{cases} \eas Recall by the definition of $\D_e$ that if $v$ is the source of $e$, the above equivalence is an exact equality.

Similarly, \bas \D_e(v_\alpha (G))
\simv \begin{cases}v_{r_e/\alpha}(v_\alpha(G))/e & v
  \textrm{ source of } e \\ v_{1/r_e\alpha}(v_\alpha(G))/e & v \textrm{ target
    of } e \;. \end{cases} \eas

By the multiplicativity of vertex rescaling, \eqref{vertexmult}, we rewrite this \bas \D_e(v_\alpha (G))
\simv \begin{cases}v_{r_e}(G)/e & v
  \textrm{ source of } e \\v_{1/r_e}(G)/e & v \textrm{ target
    of } e  \end{cases} \simv \D_e(G) \;.\eas Therefore, $\D(G) \simv \D(v_\alpha (G))$ for any $G \in G_0(k^\times)$ and $v \in V(G)$.

\end{proof}

Thus far, we have shown that the operator $\D$ is well defined on $\Q[\sG]/ (\simord,\simv)$. Next, we show that the operators $\D_e$ commute.

%To simplify notation in this proof, we denote, for two edges $e, e'
%\in E(G)$ of a graph $G \in \sG_0(k^\times)$, \bas
%\widehat{(\hat{\omega}_e)}_{e'} = \hat{\omega}_{\{e, e'\}} \;. \eas

\begin{lem} Contractions along different edges commute in $\Q[\sG]/(\simord,\simv)$ : $\D_e \circ \D_{e'} = \D_{e'}
  \circ \D_e.$ \label{commutingderivatives}\end{lem}
\begin{proof}

There are two cases to consider: when the edges  $e$ and $e'$ form a cycle in $G$, and when they do not.

If $e \cup e'$ is a union of loops in $G$, then by definition \ref{contractdef}, $\D_e G = \D_{e'} G = 0$. If $e \cup e'$ is a loop in $G$, then $e'$ defines a loop in $\D_e G $, and $e$ a loop in $\D_{e'} G$. Therefore, $\D_e \circ \D_{e'} G =  \D_{e'} \circ \D_e  G = 0$.

If $e \cup e'$ is not a cycle in $G$, there is a spanning tree $T$ such that $e, e' \subset E(T)$. By Lemma \ref{rescaletree}, write $G$ such that the edges of $T$ are labeled by $1$. In this case, $\D_e \circ \D_{e'} G = (G/e')/e = G/\{e' \cup e\} = \D_{e'} \circ \D_{e} G$.
\end{proof}

We are now ready to prove Theorem \ref{diffdef}.

\begin{proof}[Proof of Theorem \ref{diffdef}]
  Lemma \ref{diffreparam} shows that the operator \bas \D : \Q[\sG]/(\simord,\simv) \rightarrow  \Q[\sG]/(\simord,\simv) \eas is well defined.

  To see that $\D \circ \D = 0$, write \bas \D \circ \D = \sum_{e \in
    E(G/e)} (-1)^{\omega(e)}\D_e(\sum_{e' \in E(G)}
  (-1)^{\omega(e')}\D_e'(G) )\;.\eas Assume without loss of
  generality that $\omega(e) < \omega (e')$. Then the term $\D_e
\circ   \D_{e'}$ appears in $\D \circ \D$ with sign
  $(-1)^{\omega(e)}(-1)^{\omega(e')}$ while $\D_{e'} \circ \D_e$ appears
  with sign $(-1)^{\omega(e)}(-1)^{\omega(e')-1}$.  By Lemma
  \ref{commutingderivatives}, $\D_e \circ \D_{e'} = \D_{e'}\circ \D_e$. Thus the
  two contributions cancel.

  To see that $\D$ is a degree one operator, note that if $G / e$ is
  not $0$, then \bas h_1(G) = h_1(G/e)\;. \eas However, \bas |V(G/e)|
  = |V(G)| -1 + (h_0(G/e) - h_0(G)) \;.\eas Recall from equation \ref{euler}
  that if $G \in \sGoneL{\bullet}{i}$, the degree \bas i =
  h_1(G) - |V(G)| + h_0(G) \;.\eas Similarly, the degree of $(\hat
  \omega_e, G/e)$ is given by \bas h_1(G/e) - |V(G/e)| + h_0(G/e)
  =h_1(G) - (|V(G)| -1 + (h_0(G/e) - h_0(G))) + h_0(G/e) \\ = h_1(G) -
  |V(G)| + h_0(G) +1 = i + 1\eas
\end{proof}

So far, we have shown that $\Gprg/(\simord, \simv)$ is a bigraded DGA, and that $Z$ is a homorphism of algebras from $\Gprg_\bullet^\star/(\simord, \simv)$ to $\Alt \cZoneL{\bullet}{2\bullet-\star}$.  However, we are ultimately interested in graphs $\sGoneL{}{}$ that correspond to the algebra $\Asm$ under the algebra homomorphism $Z$ defined in Section \ref{homomorphism}. In Section \ref{admissiblegraphs}, we define the algebra of admissible graphs, and show that $\sGoneL{}{}$ is a DGA under the differential defined in this section. In Section \ref{graphstocycles}, we show that $\sGoneL{}{}$ is isomorphic to $\Asm$ as a DGA.

\subsection{Admissible graphs\label{admissiblegraphs}}

So far, we have said nothing about admissible cycles. By the arguments presented in Sections \ref{homomorphism} and \ref{graphDGA}, there is an algebra homomorphism
\bas Z: \Gprg/(\simord, \simv) \rightarrow \Alt \cZoneL{\bullet}{2\bullet-\star} \;, \eas Theorem \ref{vertexscalesamecycle} shows that generators of
$\Gprg/(\simord, \simv)$ map to generic $\bP^1$-linear
cycles, under $Z$, not necessarily to admissible ones. In this section, we define a
subalgebra of admissible graphs, which, in Section \ref{graphstocycles}, we show corresponds to admissible cycles.

There is a compact way of reading off loop coefficients for graphs if the graph is parametrized as in Lemma \ref{rescaletree}, by setting the each label of a spanning tree to $1$.

\begin{lem}
Consider a connected graph $G \in \Gprg/\simord, \simv$. Each $T$, a spanning tree of $G$ defines a loop basis of $H^1(G)$, the loop coefficients of which are the labels of the edges $E(G) \setminus E(T)$. \label{spanningtreeloopbasis}
\end{lem}
\begin{proof}
Each spanning tree of a connected graph defines a set of loops in $G$ as follows. For a spanning tree $T$, each oriented edge, $e \in E(G) \setminus E(T)$
defines a graphical loop, $L_e$, in conjunction with a subset of $E(T)$. The
orientation of the graphical loop is determined by the orientation of $e$. The rank of the loop space of $G$ is $\rk H^1(G) = |E(G)| - |V(G)| + 1$. Since $|E(T)| = |V(G)| - 1$, we see that $\rk H^1(G) = |E(G) \setminus E(T)|$. Furthermore,
$\cup_{e \in E(G) \setminus E(T)} E(L_e) = E(G)$. Therefore, the set $\{L_e\}_{e \in E(G) \setminus E(T)}$ defines a basis of $H^1(G)$.

By choosing a parametrization where $T$ is labeled by ones, the graphical loop coefficient $L_e$ is exactly the label of $e$.\end{proof}

We are now ready to define a class of graphs called admissible graphs. We show in Section \ref{graphstocycles} that these correspond to admissible cycles under the homomorphism $Z$ defined in Section \ref{homomorphism}.

\begin{dfn}
A graph $G \in \sG_0(k^\times)$ is admissible if: \begin{enumerate}
\item The connected components of $G$ are strongly connected.
\item There is no graphical loop in $G$ that has loop coefficient $1$.
\end{enumerate}
\label{admisgraphdfn}\end{dfn}

We recall the definition of a strongly connected graph in the first condition.

\begin{dfn}
An oriented graph is strongly connected if for any two vertices $v, w
\in V(G)$, there is a path from $v$ to $w$ and one from $w$ to $v$
which respect the orientation of the edges of $G$. \end{dfn}

By Lemma {\ref{spanningtreeloopbasis}}, definition \ref{admisgraphdfn} implies that if a graph $G \in G_0(k^\times)$ can be parametrized such that there exists a loop with all edged labeled by ones, then $G$ is not admissible.

Finally we add one more equivalence relation among graphs that is useful in Section \ref{graphstocycles}.

\begin{dfn}
For $G \in \sG$, let $\bar G\in \sG$ be the graph with the same underlying labled unoriented graph structure, but with the orientations of each edge switched. Define an equivalence relation $\simori$ that relates graphs with all orientations switched: $G \simori \bar{G}$.
\label{changeoridfn}\end{dfn}

\begin{eg}
For instance, if
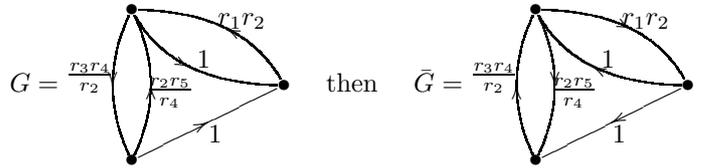
\bas G = {\begin{xy}
(0,10) *{\bullet}="t",
(0,-10) *{\bullet}= "u",
(20, 0) *{\bullet}= "z",
"t"; "z" **\crv{+(-15,0)}?/0pt/*{\dir{>}} +(2,1)*{ 1},
"z"; "t" **\crv{+(15,0)}?/0pt/*{\dir{>}}+(2,1)*{r_1 r_2} ,
"t"; "u" **\crv{+(-5,10)}?/0pt/*{\dir{>}} +(-3,1)*{\frac{r_3 r_4}{r_2}},
"u"; "t" **\crv{+(5,-10)}?/0pt/*{\dir{>}} +(2.5,-1)*{\frac{r_2r_5}{r_4}},
"u"; "z" **{\dir{-}}?/0pt/*{\dir{>}}+(1,-1.5)*{1},
\end{xy}} \quad \textrm{ then } \quad  \bar G = {\begin{xy}
(0,10) *{\bullet}="t",
(0,-10) *{\bullet}= "u",
(20, 0) *{\bullet}= "z",
"t"; "z" **\crv{+(-15,0)}?/0pt/*{\dir{<}} +(2,1)*{ 1},
"z"; "t" **\crv{+(15,0)}?/0pt/*{\dir{<}}+(2,1)*{r_1 r_2} ,
"t"; "u" **\crv{+(-5,10)}?/0pt/*{\dir{<}} +(-3,1)*{\frac{r_3 r_4}{r_2}},
"u"; "t" **\crv{+(5,-10)}?/0pt/*{\dir{<}} +(2.5,-1)*{\frac{r_2r_5}{r_4}},
"u"; "z" **{\dir{-}}?/0pt/*{\dir{<}}+(1,-1.5)*{1},
\end{xy}} \;. \eas

\end{eg}

Switching the orientation of all the edges of a graph corresponds to a reparametrization of $Z(G)$. If the $\omega(e)^{th}$ edge of $G$ corresponds to the parameterization $\phi_{\omega(e)} = \onem{t_i}{a_{\omega(e)}t_j}$, then the $\omega(e)^{th}$ edge of $\bar{G}$ corresponds to the parameterization $\bar{\phi}_{\omega(e)} = \onem{t_j}{a_{\omega(e)}t_i}$, which differs from $\phi_{\omega(e)}$ by the change of variables $t_i \rightarrow 1/ t_i$. We show that these two are both parameterizations of the same cycle in Section \ref{graphstocycles}, Corrollary \ref{equivcycles}.

\begin{dfn}
There is a subalgebra \bas \sGoneL{}{} \subset \Q[\sG] / (\simord,\simv,\simori) \eas generated over $\Q$ by admissible graphs. \bas \sGoneL{}{} = \Q[G | G \in \sG, \textrm{ admissible}]/ (\simord,\simv,\simori) \;.\eas
\end{dfn}

\begin{lem}
The differential operator, $\D$, restricts to a differential operator on $\sGoneL{}{}$.
\label{derivsindomain}\end{lem}

\begin{proof}
By the Leibnitz rule, it is sufficient to consider connected graphs. We show that if $G$ is a connected admissible graph, then so is
  $\D_e(G)$ for any $e \in E(G)$.

First, we check that if $G$ is strongly connected, then
  $G/ e$ is as well. If $v$ and $w \in G$, and in the same connected
  component of $G/e$, then the paths between $v$ and $w$ are either
  shortened by the contraction of the edge $e$, or
  unaffected. Therefore, the connected components of $G/e$ are
  strongly connected, as desired.

As taking the derivative along any edge does not affect the loop
  coefficient of any loop in $G$, for $G \in
  \sGoneL{\bullet}{\star}$, $\D_e(G)\in \sGoneL{\bullet}{\star
    +1}$.
\end{proof}

Therefore, $\sGoneL{}{}$ is a sub DGA of $\Gprg/(\simord, \simv, \simori)$. We show that the homomorphism $Z$ defined in Section \ref{homomorphism} is well defined on $\Gprg/(\simord, \simv, \simori)$.

\begin{thm}
Let $\bar{G}$ be as in Definition \ref{changeoridfn}. The graphs $G, \bar{G} \in \Gprg_\bullet^\star$ map to the same algebraic cycle in $\Alt \cZoneL{\bullet}{2\bullet - \star}$ under $Z$.
\label{orihomo}
\end{thm}

\begin{proof}
Recall from Theorem \ref{systemeqthm} and equation \eqref{systemeqloopcoef} that, given a basis of $H^1(G)$, $\beta = \{L_1, \ldots, L_\bullet\}$, the cycle $Z(G)$ is defined by the set of equations \bas \{1 =   \chi_G(L_i ) \prod_{e \in E(L_i)}(1-
  \phi_{\omega(e)})^{\epsilon(e, L_i)} \}_{L_i \in \beta} \;.\eas Note that the set $\beta$ also defines a basis of $H^1(\bar G)$, and that  for each $L_i \in \beta$, $\chi_{\bar{G}}(L_i ) = (\chi_G(L_i ))^-1$, as the only difference between $G$ and $\bar G$ is the orientation of the edges. Similarly, the function $\epsilon(e, L_i)$ defined on $G$ is the negative of the same defined on $\bar{G}$. Therefore, the cycle $Z(\bar G)$ is defined by the set of equations \bas \{1 =   (\chi_G(L_i ))^{-1} \prod_{e \in E(L_i)}(1-
  \phi_{\omega(e)})^{-\epsilon(e, L_i)} \}_{L_i \in \beta} \;.\eas

  That is, $Z(G)$ and $Z(\bar{G})$ are defined by the same algebraic cycles.
\end{proof}

Therefore, $Z: \Gprg_\bullet^\star/(\simord, \simv, \simori) \rightarrow \Alt \cZoneL{\bullet}{2\bullet - \star}$ is a well defined algebra homomorphism.
In the following section, we show that this sub DGA is isomorphic to $\Asm$.

\subsection{From graphs to admissible cycles\label{graphstocycles}}

We now return to the homomorphism defined in Section \ref{homomorphism}. In this section, we show that the map $Z$ defined in \eqref{cyclemap}, restricts to an isomorphism of DGAs between $\sGoneL{}{}$ and
$\AoneL{}{}$.

To compare the DGA of admissible cycles to the DGA of admissible
graphs, we show that the homomorphism $Z$, when restricted to $\sGoneL{}{}$ is compatible with both the differential on $\AoneL{}{}$, defined in equation \eqref{cyclediff}, and the differential on $\sGoneL{\bullet}{\star}$, defined in Theorem \ref{diffdef}.

Recall from definition \ref{higherChow} the faces $F_{I,J}$ of $\square^n$.

\begin{lem}
For $G \in \sGoneL{}{}$,
 then the derivative \bas Z(\D_e (G)) = \begin{cases} Z(G) \cap
  F_{\omega(e), \emptyset} & \textrm{ if $\sgn_{\omega(e)} = +$ }\\ Z(G)
  \cap F_{\emptyset, \omega(e)} & \textrm{ if $\sgn_{\omega(e)} = -$
    .} \end{cases} \eas
\label{equivderivs}
\end{lem}

\begin{proof}
Consider $G$ to be a connected graph. We consider two cases, when $\D_e(G)$ is connected, and when it is a disconnected graph.

The cycle $Z(G)$ is equipped with a parametrization \bas \phi : \bP^{|V(G)|-1} \rightarrow
(\bP^1)^{|E(G)|} \;,\eas
where the coordinate of $Z(G)$ corresponding to the $\omega(e)^{th}$ edge is $\phi_{\omega(e)} =(
\onem{x}{a_{\omega(e)}y})^{\sgn_{\omega(e)}}$. Recall from definition \ref{parammap} that $Z(G)$ is the cycle defined by intersecting the image of $\phi$ with $\square^{|E(G)|}$. In otherwords, $Z(G) = i^*\phi_*$, where $i: \square^{|E(G)|} \hookrightarrow (\bP^1)^{|E(G)|}$.

Let $\iota_{I,J}: F_{I,J} \rightarrow \square^n$ be the injection into the appropriate face of codimension $|I\cup J|$.  If $\sgn_{\omega(e)} = +$, (resp. $\sgn_{\omega(e)} = -$), the
intersection $Z(G) \cap F_{{\omega(e)},\emptyset}$ (resp.$Z(G)
\cap F_{\emptyset, {\omega(e)}}$) is the further pullback $\iota_{\omega(e), \emptyset}^*(i^*\phi_*)$ (resp. $\iota^*_{\emptyset, \omega(e)}(i^*\phi_*)$).

For the remainder of this proof, we assume that $\sgn_{\omega(e)} = +$. The calculation for $\sgn_{\omega(e)} = -$ is similar, and left to the reader.

The intersection $Z(G) \cap F_{{\omega(e)},\emptyset}$ imposes the restriction $x = a_e y$. Therefore, it can be parameterized by \ba \phi_{\D_e}: \bP^{|V(G)|-2} \rightarrow (\bP^1)^{|E(G)|-1} \;,\label{diffparam}\ea formed by removing the $\omega(e)^{th}$ coordinate of $\phi$ and replacing each instance of $x$ with $a_e y$. If $\D_e(G)$ is connected, this is exactly the parametrization defined by the contracted graph. Therefore, the Lemma holds when $\D_eG$ is connected.

If $\D_e(G) = \prod_{i=1}^k G_i$ is disconnected, then the parametrization defined by this disconnected graph, \bas \phi': \prod_{i=1}^k \bP^{|V(G_i)|-1} \rightarrow (\bP^1)^{|E(G)|-1} \;,\eas is different from the parametrization $\phi_{\D_e}$ defined by the contraction $\D_e$ in equation \eqref{diffparam}. However, consider the affine space $\mathbb{A}_k^{|V(G)|-2}$ defined by setting $x=a_e y = 1$ in $\bP^{|V(G)|}$. Then there is a product of corresponding affine spaces, $\prod_{i=1}^k\mathbb{A}_k^{|V(G_i)|-1}$ associated to the disconnected parametrization, each formed by setting the variable of the new vertex defined by the contraction to $1$. The two parametrizations $\phi_{\D_e}$ and $\phi'$ agree on these affine spaces. On the hyperplanes at infinity, at least one of the parameterizing variables is $0$. Since $G$ is strongly connected, none of the coordinates correspond to purely sink vertices in either $G$ or $\D_e(G)$. Therefore, setting a parameterization variable to $0$ corresponds to setting a coordinate of the image of $\phi'$ or $\phi_{\D_e}$ to $1$. However, $\square ^{|E(G)|-1}$ omits precisely the points of $\bP^{|E(G)|-1}$ where one of the coordinates is set to $1$. Therefore, the parameterized cycles $Z(\D_eG) = (i^*\phi'_*)$ and $\D_eZ(G) =i^*\phi_{\D_e\;*}$ agree on the pullback to $\bP^{|E(G)|-1}$, as desired.
\end{proof}

This is the key step to understanding the relationship between the differential on graphs and the differential on cycles.

\begin{thm}
If $G \in \sGoneL{}{}$, then \bas \D Z(G) = Z (\D(G)) \;. \eas
\label{graphcyclehomo} \end{thm}

\begin{proof}
Recall from equation \eqref{cyclediff} that \bas \D Z(G) = \sum_{e\in E(G)}(-1)^{\omega(e)-1} (\D_{\omega(e), \emptyset} - \D_{\emptyset, \omega(e)})Z(G) \;.\eas From Lemma \ref{equivderivs}, \bmls \D Z(G) = \sum_{e ;\; \sgn(e) = +} (-1)^{\omega(e)-1}( Z(\D_e G) -  \D_{\emptyset, \omega(e)} Z(G)) \\ + \sum_{e; \; \sgn(e) = -} (-1)^{\omega(e)-1}( Z(\D_e G) -  \D_{\omega(e), \emptyset} Z(G)) \;. \emls The theorem follows from the fact that $\D_{\emptyset, \omega(e)} Z(G)$ is empty if $\sgn(e) = +$ and $\D_{\omega(e), \emptyset} Z(G) = \emptyset$ if $\sgn(e) = -$.

As above, we only do the calculation for $\sgn(e) = +$, as the calculation for $\sgn(e) = -$ is similar. By definition, \bas \D_{\emptyset, \omega(e)} Z(G) = Z(G) \cap F_{\emptyset, \omega(e)} \;.\eas That is, the coordinate $\phi_{\omega(e)} = \onem{x}{a_e y} = \infty$. This implies that $\frac{x}{y} = \infty$. Since $G$ is strongly connected, there is another edge $e'$ such that $t_e = s_{e'}$. Then $\phi_{\omega(e')} = \onem{y}{a_{e'} x} = 1$. Therefore, $\D_{\emptyset, \omega(e')} Z(G) = \emptyset$.
\end{proof}

For any two edges $e, e' \in E(G)$, with $G \in
\sGoneL{}{}$, the derivatives $\D_e$ and $\D_{e'}$ commute, by Lemma \ref{commutingderivatives}.  Therefore, we can talk about contracting a
subgraph of another graph, without noting the order in which the edges
are contracted.

\begin{dfn}
  Let $G' \subset G$, with $E(G') = \{e_1, \ldots, e_n\}$. We write
  \bas \D_{G'} (G) =\D_{e_n} ( \ldots ( \D_{e_1}(G))\ldots )
  \;,\eas where $e_i \in E(G')$. \end{dfn}

Notice that if the contracted graph $G'$ is not a subtree of $G$, then $\D_{G'}(G) = 0$.

We use this shorthand to show that the graphs in $\sGoneL{\bullet}{\star}$ correspond
exactly to admissible cycles in $\AoneL{\bullet}{\star}$.  Recall
that an algebraic cycle in $\cN{\bullet}{\star}$ is admissible if it
intersects all faces of $\square^{2\bullet - \star}$ in codimension $\bullet$ or
not at all.

\begin{thm}
For $G \in \Q[\sG]/(\simord,\simv,\simori)$, the cycle $Z(G)$ is admissible if and only if $G \in \sGoneL{}{}$.
\label{admissible}\end{thm}

\begin{proof}
It is sufficient to look at connected graphs.

Consider a $G \in \Q[\sG]/(\simord,\simv,\simori)$ such that there exists a loop, $L$ with
loop coefficient $1$ in $G$. Specifically, chose a graph $G \not \in
\sGoneL{\bullet}{\star}$. By Lemma \ref{rescaletree}, we can label the
edges of any spanning tree of $L$ by ones. Since rescaling does not
change the loop coefficient, by Lemma \ref{rescalinginvariance}, all
edges of $L$ can be labeled by ones. Let $T \subset L$ be a subgraph
of the loop $L$ consisting of all but two of the edge of $L$. Label $E(L \setminus T) =  \{e_1,e_2\}$.
Let $I = \{e \in E(T) | \sgn_e = +\}$ and $J = \{e \in E(T) |
\sgn_e = -\}$. The graph $\D_T (G)$, formed by taking the
derivative of $G$ along the edges in $T$, corresponds to
intersecting $Z(G)$ with the face $F_{I,J}$. The
$\omega(e_1)^{th}$ and $\omega(e_2)^{th}$ coordinate of
$Z(\D_T(G))$ are of the form
$\sgn(e_i)(\onem{x}{y})^{\sgn(e_i)}$, for $i \in \{1,2\}$. This cycle
is not admissible.

To see this, notice that the intersection of $Z(\D_T(G))$ with the
face $F_{\omega(e_1), \emptyset}$ (if $\sgn(e_1) = +$) or
$F_{\emptyset, \omega(e_1)}$ (if $\sgn(e_1) = -$) also sets the
$\omega(e_2)^{th}$ coordinate to $0$, giving it the wrong codimension.

Conversely, suppose $G \in
\sGoneL{\bullet}{\star}$. Specifically, $G$ is strongly connected. Let
$G'$ be a (not necessarily connected) subgraph of $G$. Then by Lemma
\ref{derivsindomain} $D_{G'}(G)$, is also in
$\sGoneL{\bullet}{\star}$. If $G'$ is not a forest, then $D_{G'}(G) = 0$. Therefore, we only consider the case when $G'$ is a
forest. Let $I = \{e \in E(G') | \sgn_e = +\}$ and $J = \{e \in E(G')
| \sgn_e = -\}$.  By Lemma \ref{equivderivs} $D_{G'}(G)$
amounts to intersecting $Z(G)$ with the face $F_{I,J}$. Since
$G'$ is a forest, $h_1(G') = 0$, and $h_1(G) = h_1(D_{G'}(G))$. Therefore $Z(G) \cap F_{I,J}$ has codimension $\star$ in
$F$, making it admissible.

Finally, if $G$ is not strongly connected, then there exists two
vertices $v_1$ and $v_2$ such that there is not an orientation
respecting path in $G$ from $v_1$ to $v_2$. Let $G_1$ be the largest subgraph
of $G$ defined by the vertices that can be reached by orientation
respecting paths from $v_1$. Let $G_2$ be the largest subgraph of $G$ defined
by the vertices that can reach $v_2$ by orientation preserving paths in
$G$. By construction, $G_1$ and $G_2$ are disjoint subgraphs.

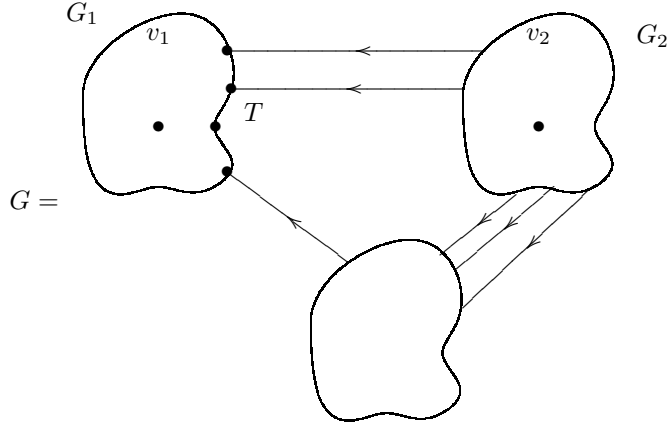
\bas
G = {\xy (0,14)*{}; (0,14)*{} **\crv{(0,20)&(16,30)&(22,15)&(15,10)
      &(22,5)&(15,0)&(10,3)&(5,0)&(0,3)&(0,13)};
(10,10) *{\bullet} +(0,12)*{v_1};
(0,25) *{G_1};
(50,14)*{}; (50,14)*{} **\crv{(50,20)&(66,30)&(72,15)&(65,10)
      &(72,5)&(65,0)&(60,3)&(55,0)&(50,3)&(50,13)};
(60,10) *{\bullet} +(0,12)*{v_2};
(75,22) *{G_2};
(30,-16)*{}; (30,-16)*{} **\crv{(30,-10)&(46,0)&(52,-15)&(45,-20)
      &(52,-25)&(45,-30)&(40,-27)&(35,-30)&(30,-27)&(30,-17)};
(19,20); (53, 20) **{\dir{-}}?/0pt/*{\dir{<}};
(19.5,15); (50, 15) **{\dir{-}}?/0pt/*{\dir{<}};
(19,4); (35, -8) **{\dir{-}}?/0pt/*{\dir{<}};
(49, -9); (62,2) **{\dir{-}}?/0pt/*{\dir{<}};
(50, -14); (67,2) **{\dir{-}}?/0pt/*{\dir{<}};
(47, -7); (57,1) **{\dir{-}}?/0pt/*{\dir{<}};
(19,20) *{\bullet};
(19.5,15) *{\bullet};
(19,4) *{\bullet};
(17.5,10) *{\bullet} +(5,2)*{T};
\endxy}
\eas

In particular, the subgraph $G_1$ has $i$ edges flowing into its
vertices from the rest of the graph, $G \setminus G_1$. Let $T$ be a
subtree of $G_1$ connecting all the sink vertices of these incoming
edges. The derivative $\D_T(G)$ has at least two connected
components.  Write \bas \D_T(G) = \pm G'
\D_T(G_1)\;, \eas with $G'$ the (possibly
disconnected) subgraph of $\D_T(G)$ that contains $G_2$ as a subgraph. The graph $G'$ has a sink vertex in
the connected component containing $(G_2)$. Therefore, the
cycle $Z(G')$ has at least two coordinates of the form
$f_i= (\onem{x}{ay})^{\sgn_i}$ and $f_j =
(\onem{z}{by})^{\sgn_j}$. Setting the coordinate $f_i = 0 (\infty)$
sets the coordinate $f_j = 0 (\infty)$ by the arguments above. Since
the derivative $\D_T(G)$ has the wrong codimension
intersecting the face $F_{i,\emptyset (\emptyset, i)}$, the cycle
$Z(G)$ is not admissible.
\end{proof}

It follows from Theorems \ref{Zhomo}, \ref{vertexscalesamecycle}, \ref{orihomo} and \ref{admissible}, that the homomorphism $Z$ is surjective.
\begin{cor}
The homomorphism \bas Z:\sGoneL{\bullet}{\star} \rightarrow \AoneL{\bullet}{\star}\eas is a surjection of DGAs.
\label{surjectioncor} \end{cor}

\begin{proof}
By Theorems \ref{Zhomo}, \ref{vertexscalesamecycle}, \ref{orihomo} and \ref{admissible}, we see that $Z$ is a homomorphism of DGA's with image contained in $\AoneL{\bullet}{\star}$. We check surjection of this map. By definition, if $Z \in  \AoneL{\bullet}{\star}$, there is a parametrization $\phi: \bP^{\bullet - \star} \rightarrow (\bP^1)^{2\bullet - \star}$, with $\phi_i = \onem{x_i}{a_iy_i}$. Assuming that $Z$ is reducible, Corrollary \ref{reduciblecor} states that this defines a connected graph $G$ with $2\bullet - \star$ edges and $\bullet - \star +1$ vertices. Since $Z$ is admissible, by Theorem \ref{admissible}, $G \in \sGoneL{}{}$.
\end{proof}

It remains to show that $Z$ is an isomorphism.

\begin{cor}
Any cycle in $\AoneL{}{}$ remains invariant under inverting all the
parameterizing variables, or scaling some of them by a constant
multiple.
\label{equivcycles}\end{cor}

\begin{proof}
This follows from Theorem \ref{systemeqthm}.

Let $\mathcal{Z} \in \AoneL{}{}$ be the cycle parametrized by the
variables $\{v_1, \ldots, v_n\}$, such that each coordinate is of the
form $\sgn_if_i^{\sgn_i}$ with \bas f_i = \onem{v_{i_s}}{a_i
  v_{i_t}}\;,\eas and $v_{i_s} ,\; v_{i_t} \in \{v_1, \ldots,
v_n\}$. Let $\mathcal{Z}' \in \AoneL{}{}$ be the cycle with
coordinates \bas f_i = \onem{b_{i_s} v_{i_s}}{a_i b_{i_t} v_{i_t}}\; ,
\eas for $b_{i_j} \in k^\times$, and $\hat {\mathcal{Z}} \in
\AoneL{}{}$ be the cycle with coordinates \bas f_i =
\onem{v_{i_t}}{a_i v_{i_s}}\;.\eas The claim of this Corrollary is
that \ba \mathcal{Z}' = \mathcal{Z} = \hat {\mathcal{Z}}\;
.\label{cycleequalities}\ea

Algebra and writing the cycles out in the form of \ref{systemeqeg}
shows that these equalities hold.
\end{proof}

In terms of graphs, the first equality in \eqref{cycleequalities}
corresponds to rescaling at vertices to pass from $G$ to $v_{1\;b_1}(
\ldots v_{n\;b_n}(G) \ldots)$ . The second equality corresponds to
changing the orientations of all the edges in the graph.

We are now ready to show that the two algebras
$\sGoneL{\bullet}{\star}$ and $\AoneL{\bullet}{\star}$ are isomorphic.

\begin{thm}
The map $Z :\sGoneL{\bullet}{\star} \rightarrow
\AoneL{\bullet}{\star}$ defined in\eqref{cyclemap} is an isomorphism
of DGAs. \label{isomorphism}
\end{thm}

\begin{proof}
Lemma \ref{graphcyclehomo} shows that $Z$ is a homomorphism of
DGAs. Corrollary \ref{surjectioncor} shows that this map is is surjective.

Rescaling a vertex on a graph $G$, that is passing from $G$ to
$v_\alpha(G)$, corresponds to rescaling the corresponding parameterizing
variable in $Z(G)$. Similarly, inverting the orientations of
all the edges, passing from $G$ to $\bar G$, corresponds to inverting
all the parameterizing variables in $Z(G)$. Since, by
corollary \ref{equivcycles}, neither of these reparameterizations
changes the underlying cycle, the map $Z$ is one to one.

Explicitly, define a map \bas G: \AoneL{\bullet}{\star} \rightarrow \sGoneL{\bullet}{\star} \eas that is a left inverse of $Z$. For any cycle parameterized in $\bP^1$-linear form, $G(\Alt [f_1^{\sgn_1}, \ldots, f_n^{\sgn_n}])$ is a graph constructed as follows. Write each $f_i$ as $1- \frac{x}{a_iy}$. If $f_i$ is a constant, write it $1- \frac{1}{a_i}$. Each independent variable in $\Alt [f_1, \ldots, f_n]$ corresponds to a vertex. For each $f_i$, draw an oriented edge of $G$, oriented from the numerator variable to the denominator variable, labeld by $a_i$. In this scheme, constant coordinates correspond to one edge loops. The term $\omega$ is defined by the ordering and signs of the $f_i$s.

\end{proof}

\section{Elements of $H^0(\BsG)$ \label{H0BG}}

In the previous section, we establish an isomorphism between the DGA
of $\bP^1$-linear cycles, $\AoneL{\bullet}{\star}$ and the DGA of
admissible graphs $\sGoneL{\bullet}{\star}$. We use this to establish
that everything that needs to be done for $\AoneL{\bullet}{\star}$
cycles can be done on the algebra of graphs
$\sGoneL{\bullet}{\star}$. For the rest of this paper, we restrict our
attention to the DGA of graphs.

In particular, to define the category of motives, we are interested in
studying the Hopf algebra, \bas H_0(\BsG) \simeq H_0(B(\Asm))
\;.\eas We maintain the definition of the bar construction $\BsG$ as in definition \ref{Bardfn}, with $A = \sGoneL{}{}$.  Following convention, we
indicate the tensor product in the bar construction by $|$.

As in definition \ref{Bardfn},
write the degree and tensor graded components of $\BsG$ as \ba \BsG^n_m =
\bigoplus_{\sum_1^n (w_i-1) = m } [\sGoneL{\bullet}{w_1} | \ldots | \sGoneL{\bullet}{w_n}]
\;. \label{barcomplexgrading} \ea Note that, as in definition {\ref{Bardfn}}, the degree of a graph in the bar construction is shifted from the degree of a graph in the algebra. I.e. if, $G \in \sGoneL{\bullet}{j}$, the $G \in \BsG_{j-1}^1$.

\begin{dfn}
Due to the multiple degrees assigned to graphs in an algebraic and bar construction context, we $\bdeg$ (as opposed so simply $\deg$) to be the shifted degree of a graph as it contributes to the total degree in the bar construction.
\end{dfn}

Explicitly $G \in \sGoneL{\bullet}{j}$, $\deg (G) = j$ and $\bdeg (G) = j-1$ for the graph above.

To set notation, we define differentials that make the bar complex $(\BsG, \D +\mu)$ as a bi-complex. Write $\D_\sG$ and $\mu_\sG$ for the derivatives and product on the graphs. Then $\D$ and $\mu$ are the degree one operators on $\BsG$ induced by $\D_\sG$ and $\mu_\sG$, calculated by the degree of graphs in the bar construction under the Leibnitz rule. Let $\D_j$ be the differential operator that acts by $(-1)^{\bdeg G_i} \id$ on the first $j-1$ tensor components, by $\D_{\sG}$ on the $j^{th}$ tensor component, and by $\id$ on the remaining tensor components. Then
for $[ G_1 | \ldots |  G_n] \in \BsG_m^n$, write \ba \D [ G_1 | \ldots |  G_n] : = \sum_{j=1} \D_j [ G_1 | \ldots |  G_n] = \sum_{j=1}^n (-1)^{\sum_{k=1}^{j-1} \bdeg(G_k)} [ G_1
  | \ldots | \D_{\sG} ( G_j)| \ldots |  G_n] \;, \label{bardiff}\ea is a degree one differential operator $\D: \BsG_m^n \rightarrow \BsG_{m+1}^n$.
Similarly, let $\mu_j$ to be the differential operator that acts by $(-1)^{\bdeg G_i}\id$ on the first $j-1$ tensor components, by $(-1)^{\bdeg G_j}\mu$ on the $j^{th}$ and $j+1^{th}$ components, and as $\id$ on the remaining components. Then \ba \mu [ G_1 | \ldots |  G_n]  :=  \sum_{j=1} \mu_j & [ G_1 | \ldots |  G_n] = \sum_{j=1}^{n-1}  (-1)^{\sum_{i=1}^j \deg G_i} [ G_1 | \ldots
  | G_j \cdot G_{j+1}| \ldots |
  G_n] \; .\label{barprod} \ea This is a degree $1$ differential operator, as for $G_i \in \sGoneL{r_i}{m_i}$, $[ G_1 | G_2] \in \BsG^2_{m_1+m_2-2}$, while $\mu [ G_1 | G_2] = [G_1G_2]\in \BsG^1_{m_1+m_2-1}$.

In order to study elements of $H^i(\BsG)$, identify elements in the
kernel of \bas D+ \mu: \bigoplus_{n \geq 1} \BsG^i_n \rightarrow
\bigoplus_{n \geq 1} \BsG^{i+1}_{n} \;. \eas

By definition \ref{decompdfn}, we see that elements of this kernel are exactly the elements with completely decomposable boundaries.

\begin{rem}Very few generators of $\sGoneL{\bullet}{\star}$ as an algebra have a
decomposable boundary. The completely decomposable objects in $\BsG$
correspond to linear combinations of tensor products of graphs. \label{decompsums}\end{rem}

In this paper, we wish to study $H^0(\BsG)$. Therefore, we study
completely decomposable elements of $\bigoplus_{i\geq 1} \BsG_i^0$
defined by completely decomposable elements of $\BsG_1^0$. From definition \ref{decompdfn}, a completely decomposable element of $\BsG_0^1$, $\varepsilon$, defines a trivial cycle in $H^0(\BsG)$ if it can be written as the coboundary of another sum of graphs $\sum_i  G_i \in \sG_2^1$, \bas \D
\sum_i G_i = \varepsilon \;,\eas or if it can be written as
the sum of a product of graphs, \bas \mu \sum_i [G_{1, i} | G_{2, i}] =
\varepsilon \;. \eas

In this section, we first give a result that greatly reduces the number of algebraic cycles in ${\Asm}_1^\bullet$ one needs to consider to construct $H^0(\BsG)$.

\begin{thm}
If $\varepsilon \in {\Asm}^\bullet_i$ is a completely decomposable algebraic cycle which can be written as $Z(\sum G_j)$, where each $G_j \in \sGoneL{\bullet}{i}$, and some $G_j$ have two valent vertices, then $\varepsilon$ defines a coboundary element of $B(\Asm)$. \label{twovalentthm}\end{thm}

In particular, taking $i = 0$, we see that sums of graphs involving two valent vertices have trivial motivic contributions. This is a major calculational aid in that it identifies a large class of cycles that we need not consider for motivic content. The proof of this theorem is the subject of Section \ref{twovalent}. See Theorem \ref{twovalentgraphthm} for the graphical version of this statement. In Section \ref{examples} we give examples of some completely decomposable graphs.

Since we are only interested in the $0^{th}$ cohomology henceforth, for the remainder of this paper, we only consider graphs in $\sGoneL{\bullet}{1}$, that is, cycles in ${\Asm}_1^\bullet$.

\subsection{Valence two vertices\label{twovalent}}

In this section we show that there is a large class of graphs in $\sGoneL{}{}$ that correspond to the trivial cycles in $H^i(\BsG)$. Namely, we show that completely decomposable sums of graphs with two valent vertices can be written as the coboundary of an element of $\sGoneL{\bullet}{i-1}$.
We start by studying the properties of decomposable graphs in $\sGoneL{}{}$
with two valent vertices.

\begin{dfn}
A handle of length $n>1$ is a linear subgraph $h \in G$ defined by $n$
edges and $n+1$ vertices $\{v_0, \ldots , v_n\}$ labelled as follows: the vertex $v_i$ is
two valent if $1\leq i <n$, and $v_0$ and $v_n$ have valence
$1$. Write $E(h) = \{e^1, \ldots e^n\}$, with $e^i$
the edge in the $h$ connecting vertex $v_{i-1}$ and
$v_{i}$. Write $H(G)$ to the be set of handles of a graph $G$.
\end{dfn}

Minimally decomposable sums of graphs can be classified
by the number of handles they have.

\begin{lem}
Consider $G \in \sGoneL{}{}$ a connected graph such that there is a handle, $h \in H(G)$, of length $n$. Then \bas \sum_{e\in E(h)} (-1)^{\omega(e)} \D_{e} G
= \begin{cases}0 & \textrm{if $n$ even} \\ (-1)^{\omega(e^1)} \D_{e^1}
  G & \textrm{if $n$ odd} \;.\end{cases} \eas \label{evenzero}
\end{lem}

\begin{proof}
The essence of this proof comes from showing the following relation:  \ba
(-1)^{\omega(e^i)} \D_{e^i} G = -(-1)^{\omega(e^{i+1})}
\D_{e^{i+1}}  G\;. \label{keyrel}\ea

To see this, choose a representation of $G$ such that the edges of
$h$ are labeled by ones.

Write $c(e^i, e^{i+1}) \in \fS_{|E(G)|}$ as the cyclic element of order
$|\omega(e^{i+1})-\omega(e^{i})|$.  Write this \bas c(e^i, e^{i+1}) :=
(\omega(e^{i}))(\omega(e^{i})+1) \ldots (\omega(e^{i+1})-1)
(\omega(e^{i+1})) \;. \eas The sign of this permutation is given by $\sgn(c(e^i, e^{i+1})) =
(-1)^{\omega(e^{i+1})-\omega(e^{i})  + 1}$.  In this notation,
the orderings of the contracted graphs can be related by \bas \hat \omega_{e^{i+1}} = \begin{cases} c(e^i, e^{i+1})
  \hat \omega_{e^{i}} & \textrm{ if }\omega(e^{i})<
  \omega(e^{i+1}) \\ c(e^i, e^{i+1})^{-1} \hat \omega_{e^{i}} &
  \textrm{ if } \omega(e^{i})> \omega(e^{i+1})\;.\end{cases}\eas Since
 the underlying contracted graphs, $G/ e^{i} = G/ e^{i+1}$, are the same, we have, by Lemma \ref{orderequiv} \bas
 (-1)^{\omega(e^{i+1})-\omega(e^{i}) + 1} \D_{e^{i+1}}( G) =   \D_{e^{i}}G \;, \eas which is equivalent to \eqref{keyrel}.

Summing over all edges in a fixed handle $h$ gives \bas \sum_{e
  \in E(h)} (-1)^{\omega(e)} \D_{e} ( G)
= \begin{cases}0 & \textrm{if $n$ even} \\ (-1)^{\omega(e^{1})} \D_{e^{1}}
  (G) & \textrm{if $n$ odd} \;. \end{cases} \eas
\end{proof}

Call edges of $G$ that are not handles, interior edges of $G$.

\begin{dfn}
By abuse of notation, write $\mathring{G}$ to indicate the interior
graph of $G$. This is the $G$ with all its handles removed (not
contracted). More precisely, \bas \mathring{G} = G \setminus \{e | e
\in E(h); h \in H(G)\}\;. \eas \end{dfn}

In this section, we write \ba \D|_H (\omega, G) = \sum_{e \in H(G)}
(-1)^{\omega(e)}\D_e(G) \label{Dhandle}\ea and \ba \D|_{\mathring{G}} ( G) =
\sum_{e \in \mathring{G}} (-1)^{\omega(e)}\D_e(G)\;, \label{Dint}\ea such
that $\D = \D|_H + \D|_{\mathring{G}}$. This allows for a neat reorganizing of the terms in the derivative $\D G$ by interior edges and edges with two valent endpoints.

\begin{cor}
The derivative \bas \D ( G) = \sum_{e \in E(\mathring{G})} (-1)^{\omega(e)} \D_e
( G) + \sum_{\begin{subarray}{c}e \in h \; ; h \in H(G) \\ h \textrm{ odd
      length}\end{subarray}} (-1)^{\omega(e_1(h))} \D_{e_1(h)}
( G) \;.\eas
\label{handlederivs}\end{cor}

As a direct corollary, we see that graphs with two valent vertices form a separate class of graphs in themselves. If $\varepsilon \in \BsG$ is a minimally decomposable sum of graphs, then either all the summands involve a two valent vertex, or none of them do. In fact, one can be more specific than this.

\begin{cor}
Consider a minimally decomposable sum of graphs $\varepsilon = \sum_j G_j \in
\sGoneL{\bullet}{i}$ of fixed degree. The summand graph
$G_j$ has a valence two vertex if and only if the graphs
in each of the summand have the same number of handles: \bas |H(G_j)|
= |H(G_{j'})| \; \forall j \neq j' \;.\eas \label{samehandle}
\end{cor}

\begin{proof}
If $\D_e( G_j)$ is not decomposable, then it must cancel with
a sum of another derivative $\D_{e'}( G_{j'})$. By Lemma \ref{evenzero} and Corollary \ref{handlederivs}, applying $\D$ does not change the number of
handles on a graph. Since $\varepsilon$ has a minimally completely decomposable boundary, there are no summands that do not contribute to the cancellation of the terms in $\D_e$.  Therefore, $G_j$ and $G_j'$ must have the same
number of handles.
\end{proof}

Finally, we show that sums of graphs with decomposable boundaries and
two valent vertices characterize trivial classes in $H^0(\BsG)$. In the proof of this theorem, we work up to products of graphs. For this, we establish some notation. For $G$ a connected graph in $\sGoneL{}{}$, if $\D_e G$ is decomposable, we write \bas \D_e G \doteq 0 \;. \eas In general, we write $\D G \doteq G'$, where $G'$ is a linear sum of connected graphs, that is, $G'$ is the sum of graphs corresponding to edge differentials that do not split the graph into two connected components.

\begin{thm}
If $\varepsilon = \sum_j (G_j) \in \sGoneL{\bullet}{i}$ has a
minimally completely decomposable boundary, there exists a sum of
graphs $\eta \in \sGoneL{\bullet}{i-1}$ such that $[\D \eta] =
[\varepsilon]$. In otherwords $[\varepsilon]$ is exact. \label{twovalentgraphthm}
\end{thm}

Translated into the language of algebraic cycles, instead of graphs, this theorem gives Theorem \ref{twovalentthm}.

\begin{proof}
Write $\varepsilon$ as a sum of terms defined by the number of odd
handles the summands have: \bas \varepsilon = \sum_{i=0}^n
\varepsilon_i \eas where $\varepsilon_i$ is a sum of ordered signed
graphs with $i$ handles of odd length. (By Corollary \ref{samehandle}, each summand has the same number of total handles. Call this number $m \geq n$.) It suffices to work with sums of connected graphs.

Recalling the decomposition of the differential operator $\D = \D|_H +
\D|_{\mathring{G}}$ from equations \eqref{Dhandle} and \eqref{Dint}, the sum $\D(\varepsilon)$ decomposes into $n+1$ sums that
evaluate to $0$, up to decomposable elements. By collecting terms
according to the number of odd handles are present in the graph: \ba \D|_{\mathring{G}} \varepsilon_n \doteq 0 \nonumber \\ \D|_H
  \varepsilon_n + \D|_{\mathring{G}} \varepsilon_{n-1} \doteq 0 \nonumber \\ \vdots
\nonumber\\ \D|_H \varepsilon_1 + \D|_{\mathring{G}} \varepsilon_0 \doteq
0 \label{lineqstwovalence} \ea

First we show that for all $0 \leq i \leq n$,  $\varepsilon_i \neq 0$. Otherwise, $\varepsilon$ would not be a minimally decomposable sum of graphs. Namely, if there existed some $i$ such that $\varepsilon_i = 0$, i.e. that no summand of $\varepsilon$ had $i$ odd handles, then $\D|_{\mathring{G}} \varepsilon_i \doteq 0$ implying
that $\D|_H \varepsilon_{i+1} \doteq  0$. However, $\D|_H
\varepsilon_{i+1}$ is completely decomposable if and only if
$\varepsilon_{i+1}\doteq 0$. But this is not possible since $\varepsilon$ is a minimally decomposable sum.

By abuse of notation, for $\varepsilon_n = \sum_i G_i$, write $\mathring{\varepsilon}_n = \sum_i \mathring{G}_i$. The set of equations \eqref{lineqstwovalence} imply that $\D \mathring{\varepsilon}_n \doteq 0$, and $\D \mathring{\varepsilon}_i \doteq - \mathring{\varepsilon}_{i+1}$. This gives rise to a statement about coboundary relations between the interiors of the graphs involved: \bas \sum_{j=0}^n \D \mathring{\varepsilon}_j \doteq - \sum_{i=1}^n \mathring{\varepsilon}_i \;.\eas

To give a full statement about the graphs, we first construct an element $\eta_n$ such that $\mathring{\eta}_n = \mathring{\varepsilon}_{n-1}$ with $n$ handles of odd length such that $-\D|_H \eta_n = \varepsilon_{n-1} + R_{n-1}$, where $R_{n-1}$ is a sum of graphs with $n-1$ handles of odd length that we refer to as the remainder for now. As with the summands of $\varepsilon_n$, we construct $\eta_n$ such that each summand has $m \geq n$ handles. We impose a handle structure on the summands of $\eta_n$ to satisfy \bas - \D|_{\mathring{G}} \eta_n =  \varepsilon_n \;.\eas

Since $\mathring{\varepsilon}_{n-1} = \mathring{R}_{n-1}$, and $\D \mathring{\varepsilon}_{n-2} = - \mathring{\varepsilon}_{n-1}$, we can build an $\eta_{n-1}$ such that $\mathring{\eta}_{n-1} = \mathring{\varepsilon}_{n-2}$, with handle structure such that $\D|_{\mathring{G}} \eta_{n-1} = R_{n-1}$ and $-\D|_H \eta_{n-1} \doteq \varepsilon_{n-2} + R_{n-2}$. We can construct such $\eta_i$ for all $0\leq i \leq n$. As with the summands of $\varepsilon_i$, each summand of $\eta_i$ has $m \geq n$ handles.

These $\eta_i$ are constructed using the fact that since $\D \mathring{\varepsilon}_i \doteq - \mathring{\varepsilon}_{i+1}$, graphs that cancel in the equation $\D|_{\mathring{G}}\varepsilon_i +\D|_H \varepsilon_{i+1} = 0$ have the same handle structure. Build $\eta_{i+1}$ out of $\varepsilon_i$ by extending a handle of even length by one edge, making it a handle of odd length. Which handle of even length is extended is chosen according to the external handle structure of $R_{n-1}$, and $\varepsilon_{i+1}$. If this new edge is the first edge of $\eta_{i+1}$, while the relative ordering of the remaining edges of $\eta_{i+1}$ and $\varepsilon{i}$ are the same, this gives the required alternating signs on the remainder terms $R_i$.

Under this construction, \bas - \D \sum_{i=1}^n \eta_i \doteq \sum_{j=0}^n \varepsilon_j \;.\eas Since we work up to decomposable terms, we may assume that all graphs in the sum $\sum_{i=0}^{n-1} \eta_i$ are connected. Therefore, we may write \bas -(\D+\mu) \sum_{i=0}^{n-1} \D \eta_i = \varepsilon\;,\eas i.e. $\varepsilon$ is exact, as desired.
\end{proof}

So far we have shown a class of minimally decomposable sums of graphs (algebraic cycles) that give rise to trivial motives. We have said nothing about how to find such minimally decomposable sums. In the next section we give some examples of minimally decomposable sums in degree $4$, only one of which has been previously studied \cite{GGL05}. There is little technology developed to identify minimally decomposable cycles, which define classes in $H^0(\BsG)$, let alone in understanding relations between such. What little progress there has been \cite{GGL05, GGL07, Souderes12} has been on a case by case basis. We hope to revisit this question in the future in a more systematic manner.
%We hope that in future work, and a reliance on new graph theoretic tools will develop a more systematic approach to this problem.

\subsection{Examples of elements in $H^0(\BsG)$}

In this section, we give several examples of classes of $H^0(\BsG)$. Generally speaking, it is nontrivial to find linear combinations of graphs which define classes in $H^0(\BsG)$.
%there are  distinct algebraic cycles (graphs), but only a small number of cohomological generators (motives).
Individual graphs do not have decomposable boundaries. It is only when summed with appropriate graphs with whom the boundaries cancel does one find classes in $H^0(\BsG)$.

In the following subsection, we give examples of several sums in weight four.

\begin{rem} In all of these examples in this section, we write only a sum of graphs in $\sGoneL{4}{1}$, and not the full representative in $\BsG$. We can do this since the indecomposable graphs in a completely decomposable sum of graphs determines its class in $\BsG$ (see remark \ref{mindecomprmk}). \end{rem}
%, up to products of graphs.
%Disjoint graphs, or non-primitive elements of
%$\sGoneL{\bullet}{\star}$, are in the image of product map $\mu$ and
%therefore define a trivial class in $H^0(\BsG)$. Therefore, in the
%following, we work up to products of graphs.

After giving examples in weight 4, we turn our attention to an particularly nice infinite family of graphs for which we compute the Hodge realization functor in Section \ref{Hodgerealization}.

\subsubsection{Some minimally decomposable examples in degree 4 \label{examples}}

In this section we give several examples of minimally decomposable sums of graphs in weight four. One of these, example \ref{Herbert4}, corresponds exactly to the decomposable cycles identified in \cite{GGL05} that correspond with $\Li_{1,1,1,1}(\frac{b}{a}, \frac{c}{b}, \frac{d}{c}, \frac{1}{d})$. We also find a different minimally decomposable sum of graphs that involves the same unoriented graphs, but with different coefficients and orientations on the edges.  In example \ref{slashedbox} we give two minimally decomposable sums that involve a different underlying graph, though closely related to the underlying graph of the previous example. Example \ref{necklace} gives the degree four example of the family of graphs studied in detail in Section \ref{necklacegraph}. (In section \ref{necklacehodge} we calculate the Hodge realization of these graphs.) Finally, %examples \ref{sauron2} and \ref{sauron} are
example \ref{sauron} gives a more complicated minimally decomposable sum in degree four involving several distinct underlying graphs.

The reader is encouraged to play with these examples and construct others.  There seems to be
a lot of variety as to the type and number of underlying graphs in a sum that is decomposable.  It would be very interesting to understand this structure better.
%The existence of two such elements of $H^0(\BsG)$, in combination with Theorem \ref{systemeqthm}, offers hope that this graphical depiction of motives, which allows one to study a large number of algebraic cycles, will lead to relations between mixed Tate Motives their realizations.

\begin{eg} \label{Herbert4}
In \cite{GGL05}, the authors define a family of five binary graphs that correspond to $Li_{1,1,1,1}(\frac{b}{a}, \frac{c}{b}, \frac{d}{c}, \frac{1}{d})$. In the notation developed in this paper, we depict this same minimally decomposable sum of trees as

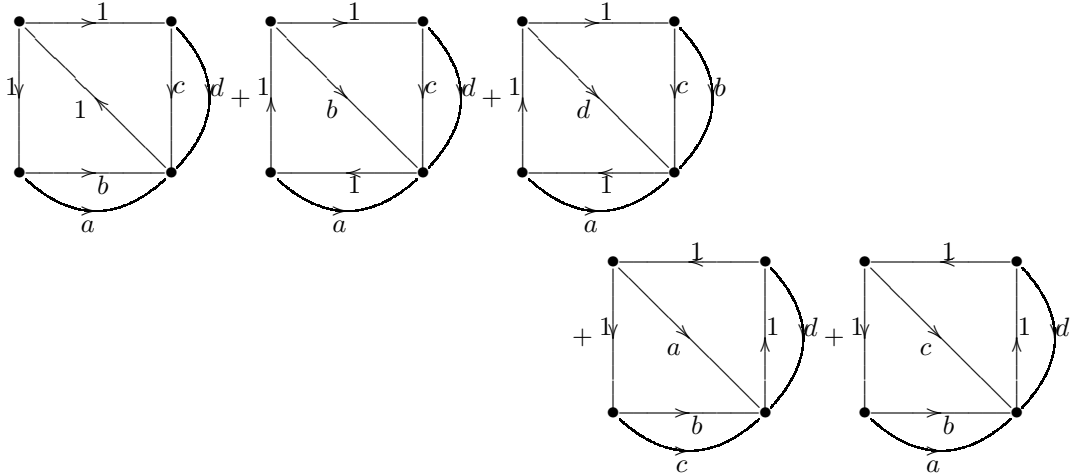
\bmls {\begin{xy}
(-10,10) *{\bullet}="x",
(10,10) *{\bullet}= "y",
(-10, -10) *{\bullet}= "z",
(10, -10) *{\bullet}= "w",
"x"; "y" **{\dir{-}}?/0pt/*{\dir{>}}+(1,1.5)*{1},
"x"; "z" **{\dir{-}}?/0pt/*{\dir{>}}+(-1,1.5)*{1},
"x"; "w" **{\dir{-}}?/0pt/*{\dir{<}}+(-2,-1.5)*{1},
"y"; "w" **{\dir{-}}?/0pt/*{\dir{>}}+(1,1.5)*{c},
"y"; "w" **\crv{+(10,10)}?/0pt/*{\dir{>}}+(1,1.5)*{d},
"z"; "w" **{\dir{-}}?/0pt/*{\dir{>}}+(1,-1.5)*{b},
"z"; "w" **\crv{+ (-10, -10)}?/0pt/*{\dir{>}}+(-1,-2)*{a},
\end{xy}} +{\begin{xy}
(-10,10) *{\bullet}="x",
(10,10) *{\bullet}= "y",
(-10, -10) *{\bullet}= "z",
(10, -10) *{\bullet}= "w",
"x"; "y" **{\dir{-}}?/0pt/*{\dir{>}}+(1,1.5)*{1},
"x"; "z" **{\dir{-}}?/0pt/*{\dir{<}}+(-1,1.5)*{1},
"x"; "w" **{\dir{-}}?/0pt/*{\dir{>}}+(-2,-1.5)*{b},
"y"; "w" **{\dir{-}}?/0pt/*{\dir{>}}+(1,1.5)*{c},
"y"; "w" **\crv{+(10,10)}?/0pt/*{\dir{>}}+(1,1.5)*{d},
"z"; "w" **{\dir{-}}?/0pt/*{\dir{<}}+(1,-1.5)*{1},
"z"; "w" **\crv{+ (-10, -10)}?/0pt/*{\dir{>}}+(-1,-2)*{a},
\end{xy}} +{\begin{xy}
(-10,10) *{\bullet}="x",
(10,10) *{\bullet}= "y",
(-10, -10) *{\bullet}= "z",
(10, -10) *{\bullet}= "w",
"x"; "y" **{\dir{-}}?/0pt/*{\dir{>}}+(1,1.5)*{1},
"x"; "z" **{\dir{-}}?/0pt/*{\dir{<}}+(-1,1.5)*{1},
"x"; "w" **{\dir{-}}?/0pt/*{\dir{>}}+(-2,-1.5)*{d},
"y"; "w" **{\dir{-}}?/0pt/*{\dir{>}}+(1,1.5)*{c},
"y"; "w" **\crv{+(10,10)}?/0pt/*{\dir{>}}+(1,1.5)*{b},
"z"; "w" **{\dir{-}}?/0pt/*{\dir{<}}+(1,-1.5)*{1},
"z"; "w" **\crv{+ (-10, -10)}?/0pt/*{\dir{>}}+(-1,-2)*{a},
\end{xy}} \\ +{\begin{xy}
(-10,10) *{\bullet}="x",
(10,10) *{\bullet}= "y",
(-10, -10) *{\bullet}= "z",
(10, -10) *{\bullet}= "w",
"x"; "y" **{\dir{-}}?/0pt/*{\dir{<}}+(1,1.5)*{1},
"x"; "z" **{\dir{-}}?/0pt/*{\dir{>}}+(-1,1.5)*{1},
"x"; "w" **{\dir{-}}?/0pt/*{\dir{>}}+(-2,-1.5)*{a},
"y"; "w" **{\dir{-}}?/0pt/*{\dir{<}}+(1,1.5)*{1},
"y"; "w" **\crv{+(10,10)}?/0pt/*{\dir{>}}+(1,1.5)*{d},
"z"; "w" **{\dir{-}}?/0pt/*{\dir{>}}+(1,-1.5)*{b},
"z"; "w" **\crv{+ (-10, -10)}?/0pt/*{\dir{>}}+(-1,-2)*{c},
\end{xy}} +{\begin{xy}
(-10,10) *{\bullet}="x",
(10,10) *{\bullet}= "y",
(-10, -10) *{\bullet}= "z",
(10, -10) *{\bullet}= "w",
"x"; "y" **{\dir{-}}?/0pt/*{\dir{<}}+(1,1.5)*{1},
"x"; "z" **{\dir{-}}?/0pt/*{\dir{>}}+(-1,1.5)*{1},
"x"; "w" **{\dir{-}}?/0pt/*{\dir{>}}+(-2,-1.5)*{c},
"y"; "w" **{\dir{-}}?/0pt/*{\dir{<}}+(1,1.5)*{1},
"y"; "w" **\crv{+(10,10)}?/0pt/*{\dir{>}}+(1,1.5)*{d},
"z"; "w" **{\dir{-}}?/0pt/*{\dir{>}}+(1,-1.5)*{b},
"z"; "w" **\crv{+ (-10, -10)}?/0pt/*{\dir{>}}+(-1,-2)*{a},
\end{xy}} \emls
\end{eg}

There is another decomposable sum of graphs involving the same underlying unoriented graphs:
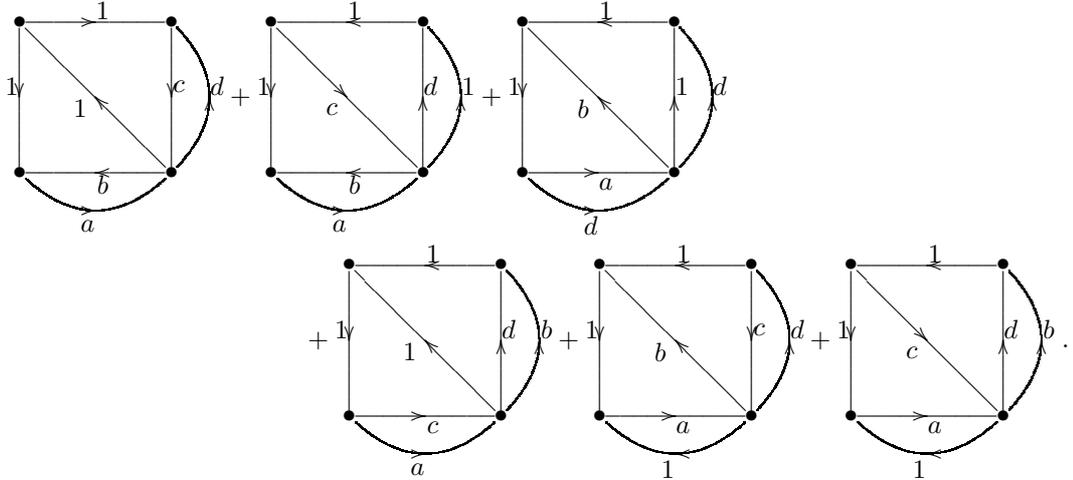
\begin{eg}
\bmls {\begin{xy}
(-10,10) *{\bullet}="x",
(10,10) *{\bullet}= "y",
(-10, -10) *{\bullet}= "z",
(10, -10) *{\bullet}= "w",
"x"; "y" **{\dir{-}}?/0pt/*{\dir{>}}+(1,1.5)*{1},
"x"; "z" **{\dir{-}}?/0pt/*{\dir{>}}+(-1,1.5)*{1},
"x"; "w" **{\dir{-}}?/0pt/*{\dir{<}}+(-2,-1.5)*{1},
"y"; "w" **{\dir{-}}?/0pt/*{\dir{>}}+(1,1.5)*{c},
"y"; "w" **\crv{+(10,10)}?/0pt/*{\dir{<}}+(1,1.5)*{d},
"z"; "w" **{\dir{-}}?/0pt/*{\dir{<}}+(1,-1.5)*{b},
"z"; "w" **\crv{+ (-10, -10)}?/0pt/*{\dir{>}}+(-1,-2)*{a},
\end{xy}} + {\begin{xy}
(-10,10) *{\bullet}="x",
(10,10) *{\bullet}= "y",
(-10, -10) *{\bullet}= "z",
(10, -10) *{\bullet}= "w",
"x"; "y" **{\dir{-}}?/0pt/*{\dir{<}}+(1,1.5)*{1},
"x"; "z" **{\dir{-}}?/0pt/*{\dir{>}}+(-1,1.5)*{1},
"x"; "w" **{\dir{-}}?/0pt/*{\dir{>}}+(-2,-1.5)*{c},
"y"; "w" **{\dir{-}}?/0pt/*{\dir{<}}+(1,1.5)*{d},
"y"; "w" **\crv{+(10,10)}?/0pt/*{\dir{<}}+(1,1.5)*{1},
"z"; "w" **{\dir{-}}?/0pt/*{\dir{<}}+(1,-1.5)*{b},
"z"; "w" **\crv{+ (-10, -10)}?/0pt/*{\dir{>}}+(-1,-2)*{a},
\end{xy}} + {\begin{xy}
(-10,10) *{\bullet}="x",
(10,10) *{\bullet}= "y",
(-10, -10) *{\bullet}= "z",
(10, -10) *{\bullet}= "w",
"x"; "y" **{\dir{-}}?/0pt/*{\dir{<}}+(1,1.5)*{1},
"x"; "z" **{\dir{-}}?/0pt/*{\dir{>}}+(-1,1.5)*{1},
"x"; "w" **{\dir{-}}?/0pt/*{\dir{<}}+(-2,-1.5)*{b},
"y"; "w" **{\dir{-}}?/0pt/*{\dir{<}}+(1,1.5)*{1},
"y"; "w" **\crv{+(10,10)}?/0pt/*{\dir{<}}+(1,1.5)*{d},
"z"; "w" **{\dir{-}}?/0pt/*{\dir{>}}+(1,-1.5)*{a},
"z"; "w" **\crv{+ (-10, -10)}?/0pt/*{\dir{>}}+(-1,-2)*{d},
\end{xy}} \\  + {\begin{xy}
(-10,10) *{\bullet}="x",
(10,10) *{\bullet}= "y",
(-10, -10) *{\bullet}= "z",
(10, -10) *{\bullet}= "w",
"x"; "y" **{\dir{-}}?/0pt/*{\dir{<}}+(1,1.5)*{1},
"x"; "z" **{\dir{-}}?/0pt/*{\dir{>}}+(-1,1.5)*{1},
"x"; "w" **{\dir{-}}?/0pt/*{\dir{<}}+(-2,-1.5)*{1},
"y"; "w" **{\dir{-}}?/0pt/*{\dir{<}}+(1,1.5)*{d},
"y"; "w" **\crv{+(10,10)}?/0pt/*{\dir{<}}+(1,1.5)*{b},
"z"; "w" **{\dir{-}}?/0pt/*{\dir{>}}+(1,-1.5)*{c},
"z"; "w" **\crv{+ (-10, -10)}?/0pt/*{\dir{>}}+(-1,-2)*{a},
\end{xy}}
+ {\begin{xy}
(-10,10) *{\bullet}="x",
(10,10) *{\bullet}= "y",
(-10, -10) *{\bullet}= "z",
(10, -10) *{\bullet}= "w",
"x"; "y" **{\dir{-}}?/0pt/*{\dir{<}}+(1,1.5)*{1},
"x"; "z" **{\dir{-}}?/0pt/*{\dir{>}}+(-1,1.5)*{1},
"x"; "w" **{\dir{-}}?/0pt/*{\dir{<}}+(-2,-1.5)*{b},
"y"; "w" **{\dir{-}}?/0pt/*{\dir{>}}+(1,1.5)*{c},
"y"; "w" **\crv{+(10,10)}?/0pt/*{\dir{<}}+(1,1.5)*{d},
"z"; "w" **{\dir{-}}?/0pt/*{\dir{>}}+(1,-1.5)*{a},
"z"; "w" **\crv{+ (-10, -10)}?/0pt/*{\dir{<}}+(-1,-2)*{1},
\end{xy}} + {\begin{xy}
(-10,10) *{\bullet}="x",
(10,10) *{\bullet}= "y",
(-10, -10) *{\bullet}= "z",
(10, -10) *{\bullet}= "w",
"x"; "y" **{\dir{-}}?/0pt/*{\dir{<}}+(1,1.5)*{1},
"x"; "z" **{\dir{-}}?/0pt/*{\dir{>}}+(-1,1.5)*{1},
"x"; "w" **{\dir{-}}?/0pt/*{\dir{>}}+(-2,-1.5)*{c},
"y"; "w" **{\dir{-}}?/0pt/*{\dir{<}}+(1,1.5)*{d},
"y"; "w" **\crv{+(10,10)}?/0pt/*{\dir{<}}+(1,1.5)*{b},
"z"; "w" **{\dir{-}}?/0pt/*{\dir{>}}+(1,-1.5)*{a},
"z"; "w" **\crv{+ (-10, -10)}?/0pt/*{\dir{<}}+(-1,-2)*{1},
\end{xy}} \;.
\emls
\end{eg}

For $G \in \sGoneL{\bullet}{\star}$, a connected graph, and $\beta= \{L_1 \ldots L_\bullet\}$, a loop basis of $H_1(G)$, let $\beta$ index the system of polynomial equations $f_{L_i}$ that define the admissible cycle $Z(G)$ in Theorem \ref{systemeqthm}. Namely, $f_L$ is the equation \bas 1 = \prod_{e \in E(L)} a_e (1 - \phi_{w(e)})^{\epsilon(e, L)}
\; .\eas Then reversing the orientation of an edge $e$ in graph $G$ without changing its label replaces every factor of $a_e (1 - \phi_{w(e)})$ with $(a_e (1 - \phi_{w(e)}))^{-1}$. In other words, such graphs represent closely related algebraic cycles. For instance, in the above example, the first graph in the five term sum and the first graph in the second term sum differ by changing the orientations of the edge labeled $b$ and the edge labeled $d$. This is also true of the last graph in the first sum and the second graph in the second sum. The second graph in the first sum and the fifth graph in the second sum differ by the orientation of the edges labeled $b$ and $d$, along with the orientation of two of the edges labeled $1$. Presumably these two sums of graphs give rise to closely related sums of algebraic cycles.

While the motive associated to the first sum has been studied (see e.g. \cite{GGL05}) the other appears to be new.   We suspect that they define dependent classes.
It would be interesting to use the Hodge realization techniques developed in this paper (Section \ref{Hodgerealization}) and/or other graphical tools to analyze the motives they represent.

There is a related family of graphs, defined by changing the labelings and orientations of \bas 
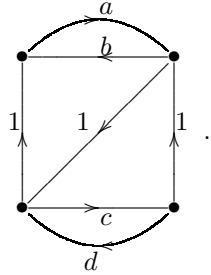
{\begin{xy}
(-10,10) *{\bullet}="x",
(10,10) *{\bullet}= "y",
(-10, -10) *{\bullet}= "z",
(10, -10) *{\bullet}= "w",
"z"; "x" **{\dir{-}}?/0pt/*{\dir{>}}+(-1,1.5)*{1},
"w"; "y" **{\dir{-}}?/0pt/*{\dir{>}}+(1,1.5)*{1},
"x"; "y" **\crv{+(-10,10)}?/0pt/*{\dir{>}}+(1,1.5)*{a},
"y"; "x" **{\dir{-}}?/0pt/*{\dir{>}}+(1,1.5)*{b},
"y"; "z" **{\dir{-}}?/0pt/*{\dir{>}}+(-2,1.5)*{1},
"z"; "w" **{\dir{-}}?/0pt/*{\dir{>}}+(1,-1.5)*{c},
"w"; "z" **\crv{+ (10, -10)}?/0pt/*{\dir{>}}+(-1,-2)*{d},
\end{xy}} \; \;. \eas
%These graphs have not been as carefully studied as the sum in the first example \ref{Herbert4}.
%However, due to the graphical symmetries we are about to write a similar pair of minimally decomposable sums based on this set of graphs (algebraic cycles).

\begin{eg} \label{slashedbox} The following sum of six diagrams is minimally decomposable,
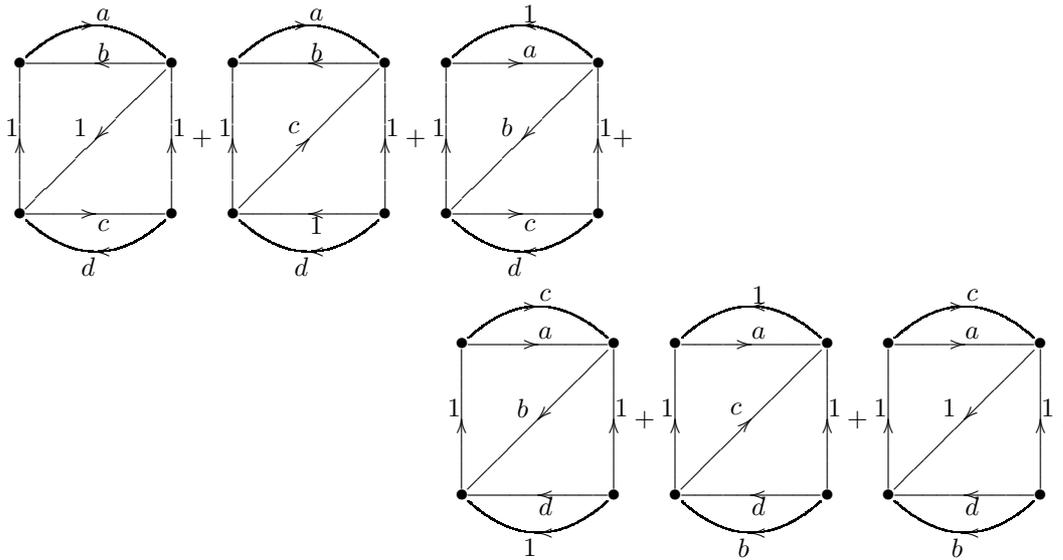
\bmls
{\begin{xy}
(-10,10) *{\bullet}="x",
(10,10) *{\bullet}= "y",
(-10, -10) *{\bullet}= "z",
(10, -10) *{\bullet}= "w",
"z"; "x" **{\dir{-}}?/0pt/*{\dir{>}}+(-1,1.5)*{1},
"w"; "y" **{\dir{-}}?/0pt/*{\dir{>}}+(1,1.5)*{1},
"x"; "y" **\crv{+(-10,10)}?/0pt/*{\dir{>}}+(1,1.5)*{a},
"y"; "x" **{\dir{-}}?/0pt/*{\dir{>}}+(1,1.5)*{b},
"y"; "z" **{\dir{-}}?/0pt/*{\dir{>}}+(-2,1.5)*{1},
"z"; "w" **{\dir{-}}?/0pt/*{\dir{>}}+(1,-1.5)*{c},
"w"; "z" **\crv{+ (10, -10)}?/0pt/*{\dir{>}}+(-1,-2)*{d},
\end{xy}} +
{\begin{xy}
(-10,10) *{\bullet}="x",
(10,10) *{\bullet}= "y",
(-10, -10) *{\bullet}= "z",
(10, -10) *{\bullet}= "w",
"z"; "x" **{\dir{-}}?/0pt/*{\dir{>}}+(-1,1.5)*{1},
"w"; "y" **{\dir{-}}?/0pt/*{\dir{>}}+(1,1.5)*{1},
"x"; "y" **\crv{+(-10,10)}?/0pt/*{\dir{>}}+(1,1.5)*{a},
"y"; "x" **{\dir{-}}?/0pt/*{\dir{>}}+(1,1.5)*{b},
"z"; "y" **{\dir{-}}?/0pt/*{\dir{>}}+(-2,1.5)*{c},
"w"; "z" **{\dir{-}}?/0pt/*{\dir{>}}+(1,-1.5)*{1},
"w"; "z" **\crv{+ (10, -10)}?/0pt/*{\dir{>}}+(-1,-2)*{d},
\end{xy}} +
{\begin{xy}
(-10,10) *{\bullet}="x",
(10,10) *{\bullet}= "y",
(-10, -10) *{\bullet}= "z",
(10, -10) *{\bullet}= "w",
"z"; "x" **{\dir{-}}?/0pt/*{\dir{>}}+(-1,1.5)*{1},
"w"; "y" **{\dir{-}}?/0pt/*{\dir{>}}+(1,1.5)*{1},
"x"; "y" **{\dir{-}}?/0pt/*{\dir{>}}+(1,1.5)*{a},
"y"; "x" **\crv{+(10,10)}?/0pt/*{\dir{>}}+(1,1.5)*{1},
"y"; "z" **{\dir{-}}?/0pt/*{\dir{>}}+(-2,1.5)*{b},
"z"; "w" **{\dir{-}}?/0pt/*{\dir{>}}+(1,-1.5)*{c},
"w"; "z" **\crv{+ (10, -10)}?/0pt/*{\dir{>}}+(-1,-2)*{d},
\end{xy}} + \\
{\begin{xy}
(-10,10) *{\bullet}="x",
(10,10) *{\bullet}= "y",
(-10, -10) *{\bullet}= "z",
(10, -10) *{\bullet}= "w",
"z"; "x" **{\dir{-}}?/0pt/*{\dir{>}}+(-1,1.5)*{1},
"w"; "y" **{\dir{-}}?/0pt/*{\dir{>}}+(1,1.5)*{1},
"x"; "y" **\crv{+(-10,10)}?/0pt/*{\dir{>}}+(1,1.5)*{c},
"x"; "y" **{\dir{-}}?/0pt/*{\dir{>}}+(1,1.5)*{a},
"y"; "z" **{\dir{-}}?/0pt/*{\dir{>}}+(-2,1.5)*{b},
"w"; "z" **{\dir{-}}?/0pt/*{\dir{>}}+(1,-1.5)*{d},
"w"; "z" **\crv{+ (10, -10)}?/0pt/*{\dir{>}}+(-1,-2)*{1},
\end{xy}} +
{\begin{xy}
(-10,10) *{\bullet}="x",
(10,10) *{\bullet}= "y",
(-10, -10) *{\bullet}= "z",
(10, -10) *{\bullet}= "w",
"z"; "x" **{\dir{-}}?/0pt/*{\dir{>}}+(-1,1.5)*{1},
"w"; "y" **{\dir{-}}?/0pt/*{\dir{>}}+(1,1.5)*{1},
"x"; "y" **{\dir{-}}?/0pt/*{\dir{>}}+(1,1.5)*{a},
"y"; "x" **\crv{+(10,10)}?/0pt/*{\dir{>}}+(1,1.5)*{1},
"z"; "y" **{\dir{-}}?/0pt/*{\dir{>}}+(-2,1.5)*{c},
"w"; "z" **{\dir{-}}?/0pt/*{\dir{>}}+(1,-1.5)*{d},
"w"; "z" **\crv{+ (10, -10)}?/0pt/*{\dir{>}}+(-1,-2)*{b},
\end{xy}} +
{\begin{xy}
(-10,10) *{\bullet}="x",
(10,10) *{\bullet}= "y",
(-10, -10) *{\bullet}= "z",
(10, -10) *{\bullet}= "w",
"z"; "x" **{\dir{-}}?/0pt/*{\dir{>}}+(-1,1.5)*{1},
"w"; "y" **{\dir{-}}?/0pt/*{\dir{>}}+(1,1.5)*{1},
"x"; "y" **{\dir{-}}?/0pt/*{\dir{>}}+(1,1.5)*{a},
"x"; "y" **\crv{+(-10,10)}?/0pt/*{\dir{>}}+(1,1.5)*{c},
"y"; "z" **{\dir{-}}?/0pt/*{\dir{>}}+(-2,1.5)*{1},
"w"; "z" **{\dir{-}}?/0pt/*{\dir{>}}+(1,-1.5)*{d},
"w"; "z" **\crv{+ (10, -10)}?/0pt/*{\dir{>}}+(-1,-2)*{b},
\end{xy}} \; ,
\emls

as is this sum of five related diagrams

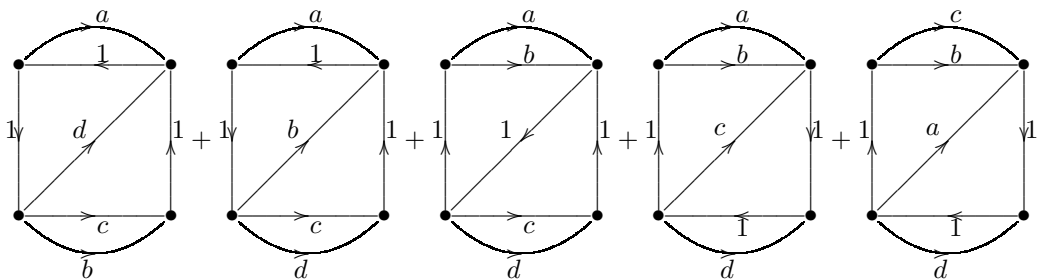
\bas
{\begin{xy}
(-10,10) *{\bullet}="x",
(10,10) *{\bullet}= "y",
(-10, -10) *{\bullet}= "z",
(10, -10) *{\bullet}= "w",
"x"; "y" **\crv{+(-10,10)}?/0pt/*{\dir{>}}+(1,1.5)*{a},
"x"; "y" **{\dir{-}}?/0pt/*{\dir{<}}+(1,1.5)*{1},
"x"; "z" **{\dir{-}}?/0pt/*{\dir{>}}+(-1,1.5)*{1},
"z"; "y" **{\dir{-}}?/0pt/*{\dir{>}}+(-2,1.5)*{d},
"y"; "w" **{\dir{-}}?/0pt/*{\dir{<}}+(1,1.5)*{1},
"z"; "w" **{\dir{-}}?/0pt/*{\dir{>}}+(1,-1.5)*{c},
"z"; "w" **\crv{+ (-10, -10)}?/0pt/*{\dir{>}}+(-1,-2)*{b},
\end{xy}} + {\begin{xy}
(-10,10) *{\bullet}="x",
(10,10) *{\bullet}= "y",
(-10, -10) *{\bullet}= "z",
(10, -10) *{\bullet}= "w",
"x"; "y" **\crv{+(-10,10)}?/0pt/*{\dir{>}}+(1,1.5)*{a},
"x"; "y" **{\dir{-}}?/0pt/*{\dir{<}}+(1,1.5)*{1},
"x"; "z" **{\dir{-}}?/0pt/*{\dir{>}}+(-1,1.5)*{1},
"z"; "y" **{\dir{-}}?/0pt/*{\dir{>}}+(-2,1.5)*{b},
"y"; "w" **{\dir{-}}?/0pt/*{\dir{<}}+(1,1.5)*{1},
"z"; "w" **{\dir{-}}?/0pt/*{\dir{>}}+(1,-1.5)*{c},
"z"; "w" **\crv{+ (-10, -10)}?/0pt/*{\dir{>}}+(-1,-2)*{d},
\end{xy}} + {\begin{xy}
(-10,10) *{\bullet}="x",
(10,10) *{\bullet}= "y",
(-10, -10) *{\bullet}= "z",
(10, -10) *{\bullet}= "w",
"x"; "y" **\crv{+(-10,10)}?/0pt/*{\dir{>}}+(1,1.5)*{a},
"x"; "y" **{\dir{-}}?/0pt/*{\dir{>}}+(1,1.5)*{b},
"x"; "z" **{\dir{-}}?/0pt/*{\dir{<}}+(-1,1.5)*{1},
"z"; "y" **{\dir{-}}?/0pt/*{\dir{<}}+(-2,1.5)*{1},
"y"; "w" **{\dir{-}}?/0pt/*{\dir{<}}+(1,1.5)*{1},
"z"; "w" **{\dir{-}}?/0pt/*{\dir{>}}+(1,-1.5)*{c},
"z"; "w" **\crv{+ (-10, -10)}?/0pt/*{\dir{>}}+(-1,-2)*{d},
\end{xy}} + {\begin{xy}
(-10,10) *{\bullet}="x",
(10,10) *{\bullet}= "y",
(-10, -10) *{\bullet}= "z",
(10, -10) *{\bullet}= "w",
"x"; "y" **\crv{+(-10,10)}?/0pt/*{\dir{>}}+(1,1.5)*{a},
"x"; "y" **{\dir{-}}?/0pt/*{\dir{>}}+(1,1.5)*{b},
"x"; "z" **{\dir{-}}?/0pt/*{\dir{<}}+(-1,1.5)*{1},
"z"; "y" **{\dir{-}}?/0pt/*{\dir{>}}+(-2,1.5)*{c},
"y"; "w" **{\dir{-}}?/0pt/*{\dir{>}}+(1,1.5)*{1},
"z"; "w" **{\dir{-}}?/0pt/*{\dir{<}}+(1,-1.5)*{1},
"z"; "w" **\crv{+ (-10, -10)}?/0pt/*{\dir{>}}+(-1,-2)*{d},
\end{xy}} + {\begin{xy}
(-10,10) *{\bullet}="x",
(10,10) *{\bullet}= "y",
(-10, -10) *{\bullet}= "z",
(10, -10) *{\bullet}= "w",
"x"; "y" **\crv{+(-10,10)}?/0pt/*{\dir{>}}+(1,1.5)*{c},
"x"; "y" **{\dir{-}}?/0pt/*{\dir{>}}+(1,1.5)*{b},
"x"; "z" **{\dir{-}}?/0pt/*{\dir{<}}+(-1,1.5)*{1},
"z"; "y" **{\dir{-}}?/0pt/*{\dir{>}}+(-2,1.5)*{a},
"y"; "w" **{\dir{-}}?/0pt/*{\dir{>}}+(1,1.5)*{1},
"z"; "w" **{\dir{-}}?/0pt/*{\dir{<}}+(1,-1.5)*{1},
"z"; "w" **\crv{+ (-10, -10)}?/0pt/*{\dir{>}}+(-1,-2)*{d},
\end{xy}}
\eas
\end{eg}

Next we present the weight four example of the necklace graphs that are the subject of Section \ref{necklacegraph}.

\begin{eg}
The following sum of graphs is minimally decomposable.
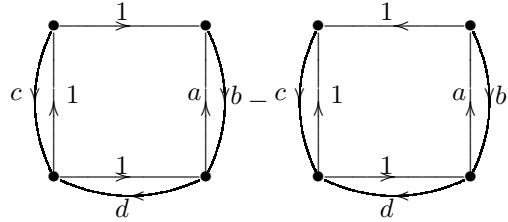
\bas
 {\begin{xy}
(-10,10) *{\bullet}="x",
(10,10) *{\bullet}= "y",
(-10, -10) *{\bullet}= "z",
(10, -10) *{\bullet}= "w",
"x"; "y" **{\dir{-}}?/0pt/*{\dir{>}}+(-1,2)*{1},
"y"; "w" **{\dir{-}}?/0pt/*{\dir{<}}+(-1.5,1)*{a} ,
"y"; "w"**\crv{+(5,10)}?/0pt/*{\dir{>}}+(1.5,1)*{b},
"x"; "z" **{\dir{-}}?/0pt/*{\dir{<}}+(2.5,1)*{1},
"x"; "z" **\crv{+(-5,10)}?/0pt/*{\dir{>}}+(-2.5,1)*{c},
"z"; "w" **{\dir{-}}?/0pt/*{\dir{>}}+(-1,1.5)*{1},
"z"; "w" **\crv{+(-10,-5)}?/0pt/*{\dir{<}}+(-1,-1.5)*{d},
\end{xy}}  -
{\begin{xy}
(-10,10) *{\bullet}="x",
(10,10) *{\bullet}= "y",
(-10, -10) *{\bullet}= "z",
(10, -10) *{\bullet}= "w",
"x"; "y" **{\dir{-}}?/0pt/*{\dir{<}}+(-1,2)*{1},
"y"; "w" **{\dir{-}}?/0pt/*{\dir{<}}+(-1.5,1)*{a} ,
"y"; "w"**\crv{+(5,10)}?/0pt/*{\dir{>}}+(1.5,1)*{b},
"x"; "z" **{\dir{-}}?/0pt/*{\dir{<}}+(2.5,1)*{1},
"x"; "z" **\crv{+(-5,10)}?/0pt/*{\dir{>}}+(-2.5,1)*{c},
"z"; "w" **{\dir{-}}?/0pt/*{\dir{>}}+(-1,1.5)*{1},
"z"; "w" **\crv{+(-10,-5)}?/0pt/*{\dir{<}}+(-1,-1.5)*{d},
\end{xy}}
 \eas
\label{necklace}
\end{eg}

We end this section with two complicated minimally decomposable sum that, unlike the previous examples, involves several different types of unoriented graphs.

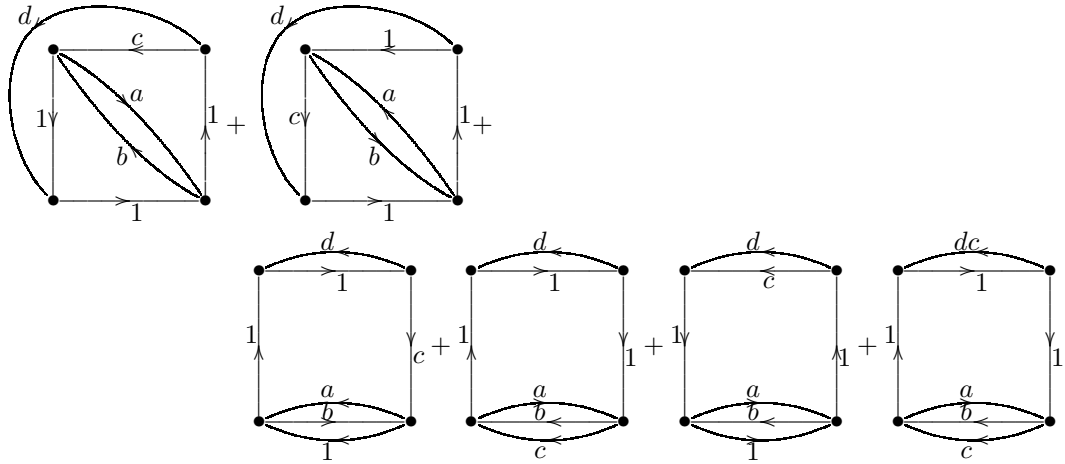
\begin{eg} \label{sauron}
\bmls
{\begin{xy}
(-10,10) *{\bullet}="x",
(10,10) *{\bullet}= "y",
(-10, -10) *{\bullet}= "z",
(10, -10) *{\bullet}= "w",
"y"; "x" **{\dir{-}}?/0pt/*{\dir{>}}+(1,1.5)*{c},
"w"; "y" **{\dir{-}}?/0pt/*{\dir{>}}+(1,1.5)*{1},
"x"; "z" **{\dir{-}}?/0pt/*{\dir{>}}+(-1.5,1)*{1},
"z"; "w" **{\dir{-}}?/0pt/*{\dir{>}}+(1,-1.5)*{1},
"x"; "w" **\crv{+(-10,15)}?/0pt/*{\dir{>}}+(1,1.5)*{a},
"x"; "w" **\crv{+(-10,5)}?/0pt/*{\dir{<}}+(-1,-1.5)*{b},
"z"; "y" **\crv{(-16, -5)&(-20,20) & (5, 16)}?/0pt/*{\dir{<}}+(-1,2)*{d},
\end{xy}}
+ {\begin{xy}
(-10,10) *{\bullet}="x",
(10,10) *{\bullet}= "y",
(-10, -10) *{\bullet}= "z",
(10, -10) *{\bullet}= "w",
"y"; "x" **{\dir{-}}?/0pt/*{\dir{>}}+(1,1.5)*{1},
"w"; "y" **{\dir{-}}?/0pt/*{\dir{>}}+(1,1.5)*{1},
"x"; "z" **{\dir{-}}?/0pt/*{\dir{>}}+(-1.5,1)*{c},
"z"; "w" **{\dir{-}}?/0pt/*{\dir{>}}+(1,-1.5)*{1},
"x"; "w" **\crv{+(-10,15)}?/0pt/*{\dir{<}}+(1,1.5)*{a},
"x"; "w" **\crv{+(-10,5)}?/0pt/*{\dir{>}}+(-1,-1.5)*{b},
"z"; "y" **\crv{(-16, -5)&(-20,20) & (5, 16)}?/0pt/*{\dir{<}}+(-1,2)*{d},
\end{xy}}
 +  \\ {\begin{xy}
(-10,10) *{\bullet}="x",
(10,10) *{\bullet}= "y",
(-10, -10) *{\bullet}= "z",
(10, -10) *{\bullet}= "w",
"x"; "y" **{\dir{-}}?/0pt/*{\dir{>}}+(1,-1.5)*{1},
"y";"x"  **\crv{+(10,5)}?/0pt/*{\dir{>}}+(-1,1.5)*{d},
"y"; "w" **{\dir{-}}?/0pt/*{\dir{>}}+(1,-1.5)*{c} ,
"x"; "z" **{\dir{-}}?/0pt/*{\dir{<}}+(-1,1.5)*{1},
"z"; "w" **{\dir{-}}?/0pt/*{\dir{>}}+(-1,1.5)*{b},
"z"; "w" **\crv{+(-10,-5)}?/0pt/*{\dir{<}}+(-1,-1.5)*{1},
"z"; "w" **\crv{+(-10,5)}?/0pt/*{\dir{<}}+(-1,1.5)*{a}
\end{xy}} + {\begin{xy}
(-10,10) *{\bullet}="x",
(10,10) *{\bullet}= "y",
(-10, -10) *{\bullet}= "z",
(10, -10) *{\bullet}= "w",
"x"; "y" **{\dir{-}}?/0pt/*{\dir{>}}+(1,-1.5)*{1},
"y";"x"  **\crv{+(10,5)}?/0pt/*{\dir{>}}+(-1,1.5)*{d},
"y"; "w" **{\dir{-}}?/0pt/*{\dir{>}}+(1,-1.5)*{1} ,
"x"; "z" **{\dir{-}}?/0pt/*{\dir{<}}+(-1,1.5)*{1},
"z"; "w" **{\dir{-}}?/0pt/*{\dir{<}}+(-1,1.5)*{b},
"z"; "w" **\crv{+(-10,-5)}?/0pt/*{\dir{<}}+(-1,-1.5)*{c},
"z"; "w" **\crv{+(-10,5)}?/0pt/*{\dir{>}}+(-1,1.5)*{a}
\end{xy}} + {\begin{xy}
(-10,10) *{\bullet}="x",
(10,10) *{\bullet}= "y",
(-10, -10) *{\bullet}= "z",
(10, -10) *{\bullet}= "w",
"x"; "y" **{\dir{-}}?/0pt/*{\dir{<}}+(1,-1.5)*{c},
"y";"x"  **\crv{+(10,5)}?/0pt/*{\dir{>}}+(-1,1.5)*{d},
"y"; "w" **{\dir{-}}?/0pt/*{\dir{<}}+(1,-1.5)*{1} ,
"x"; "z" **{\dir{-}}?/0pt/*{\dir{>}}+(-1,1.5)*{1},
"z"; "w" **{\dir{-}}?/0pt/*{\dir{<}}+(-1,1.5)*{b},
"z"; "w" **\crv{+(-10,-5)}?/0pt/*{\dir{>}}+(-1,-1.5)*{1},
"z"; "w" **\crv{+(-10,5)}?/0pt/*{\dir{>}}+(-1,1.5)*{a}
\end{xy}}+  {\begin{xy}
(-10,10) *{\bullet}="x",
(10,10) *{\bullet}= "y",
(-10, -10) *{\bullet}= "z",
(10, -10) *{\bullet}= "w",
"x"; "y" **{\dir{-}}?/0pt/*{\dir{>}}+(1,-1.5)*{1},
"y";"x"  **\crv{+(10,5)}?/0pt/*{\dir{>}}+(-1,1.5)*{dc},
"y"; "w" **{\dir{-}}?/0pt/*{\dir{>}}+(1,-1.5)*{1} ,
"x"; "z" **{\dir{-}}?/0pt/*{\dir{<}}+(-1,1.5)*{1},
"z"; "w" **{\dir{-}}?/0pt/*{\dir{<}}+(-1,1.5)*{b},
"z"; "w" **\crv{+(-10,-5)}?/0pt/*{\dir{<}}+(-1,-1.5)*{c},
"z"; "w" **\crv{+(-10,5)}?/0pt/*{\dir{>}}+(-1,1.5)*{a}
\end{xy}} \emls \begin{comment} + {\begin{xy}
(-10,10) *{\bullet}="x",
(10,10) *{\bullet}= "y",
(-10, -10) *{\bullet}= "z",
(10, -10) *{\bullet}= "w",
"x"; "y" **{\dir{-}}?/0pt/*{\dir{>}}+(-1,2)*{1},
"y"; "w" **{\dir{-}}?/0pt/*{\dir{<}}+(-1.5,1)*{a} ,
"y"; "w"**\crv{+(5,10)}?/0pt/*{\dir{>}}+(1.5,1)*{b},
"x"; "z" **{\dir{-}}?/0pt/*{\dir{<}}+(2.5,1)*{1/d},
"x"; "z" **\crv{+(-5,10)}?/0pt/*{\dir{>}}+(-2.5,1)*{c},
"z"; "w" **{\dir{-}}?/0pt/*{\dir{>}}+(-1,1.5)*{1},
"z"; "w" **\crv{+(-10,-5)}?/0pt/*{\dir{<}}+(-1,-1.5)*{d},
\end{xy}}  + {\begin{xy}
(-10,10) *{\bullet}="x",
(10,10) *{\bullet}= "y",
(-10, -10) *{\bullet}= "z",
(10, -10) *{\bullet}= "w",
"x"; "y" **{\dir{-}}?/0pt/*{\dir{>}}+(-1,2)*{1},
"y"; "w" **{\dir{-}}?/0pt/*{\dir{<}}+(-1.5,1)*{b} ,
"y"; "w"**\crv{+(5,10)}?/0pt/*{\dir{>}}+(1.5,1)*{a},
"x"; "z" **{\dir{-}}?/0pt/*{\dir{>}}+(2.5,1)*{1/d},
"x"; "z" **\crv{+(-5,10)}?/0pt/*{\dir{<}}+(-2.5,1)*{c},
"z"; "w" **{\dir{-}}?/0pt/*{\dir{<}}+(-1,1.5)*{1},
"z"; "w" **\crv{+(-10,-5)}?/0pt/*{\dir{>}}+(-1,-1.5)*{d},
\end{xy}} \end{comment}
\end{eg}

It is highly likely that the classes defined by all of the above examples are related.  It would be very interesting to work out the precise dependencies.
%While there is no algorithm as yet for determining which collections of graphs lead to minimally decomposable sums, there are several notable characteristics of the examples given in this section. Following Goncharov, one expects generators of $H^0(\BsG)$ to correspond to multiple $\zeta$ values of the correct weight \cite{Gonch98, Gonch01}. This immediately points to the likelihood that many of the examples given above are either trivial or lead to the same cohomology class. We hope that graphically related minimally decomposable sums, such as those depicted in examples \ref{Herbert4} and \ref{slashedbox} will lead to related cohomology classes. We hope to explore this in future work.

%\begin{comment}\hlfix{It is also well known that there are missing relations between the multiple polylogarithm, starting at weight four \cite{Deligne2010}. It is possible that if some of the minimally decomposable sums are coboundaries, that these graphs will lead to some of these missing relations}{I feel like nixing the talk about missing relations as it is so completely conjectural. What do you think?} \end{comment}

These examples illustrate
%graphical techniques show
that even in the vastly simplified case of $\Asm$, there is a richness  and complexity amongst the minimally decomposable classes of $\BsG$.
%The graphical representation of cycles sees far more cycles than other combinatoral approaches \cite{GGL05, Souderes12}.
%In particular, it sees algebraic cycles which do not correspond to multiple zeta values.
By further studying these minimally decomposable sums of graphs, we hope to gain a better understanding of the structure of (our subcategory of) mixed Tate motives.

\subsubsection{The $n$-beaded necklace graph \label{necklaceeg}}
In this section, we introduce an infinite family of terms in $H^0(\BsG)$, which we refer to as necklace diagrams. In Section \ref{Hodgerealization}, we show that these correspond to trivial classes.

\begin{dfn} A necklace graph with $n$ beads is the graph of the form

\ba G^*(a_0, \ldots, a_n) = {\begin{xy}
(-10,5) *{\bullet}="x",
(10,5) *{\bullet}= "y",
(-10, -5) *{\bullet}= "z",
(10, -5) *{\bullet}= "w",
(0, 5) *{*}+(0,2)*{a_0},
"y"; "x" **{\dir{-}},
"x"; "z" **\crv{+(-5,5)}?/0pt/*{\dir{>}}+(-3.5,1)*{a_1},
"x"; "z" **\crv{+(5,5)}?/0pt/*{\dir{<}}+(3, -1)*{1},
"y"; "w" **\crv{+(-5,5)}?/0pt/*{\dir{>}}+(-2,1)*{1},
"y"; "w" **\crv{+(5,5)}?/0pt/*{\dir{<}}+(4, -1)*{a_n},
"z"; "w" **\crv{~**\dir{..}+(-10,-10)},
\end{xy}}
\label{necklacegraph} \ea with $* \in \{L(eft), R(ight)\}$ to indicate the orientation of the marked edge.
The ordering is given as follows: each edge labeled $a_i$ is in the $(2i+1)^{th}$ position, for $i >0$, the parallel edge labeled $1$ (that shares vertices with that labeled $a_i$) is in the $(2i)^{th}$ position. The signs associated to the edges are all positive.
\end{dfn}

When $n=0$, we write \ba G_0(a) = G^R_0(a)= G^L_0(a) = \Go{a} \;. \label{G0dfn}\ea

We consider the following linear combination of $n$ beaded necklace
graphs \ba \varepsilon^n (a_0, \ldots, a_n) = G^L(a_0,a_1 \ldots
a_n) - G^R(\frac{1}{a_0},a_1 \ldots a_n) = \\
{\begin{xy}
(-10,5) *{\bullet}="x",
(10,5) *{\bullet}= "y",
(-10, -5) *{\bullet}= "z",
(10, -5) *{\bullet}= "w",
"y"; "x" **{\dir{-}}?/0pt/*{\dir{>}}+(0,2)*{a_0},
"x"; "z" **\crv{+(-5,5)}?/0pt/*{\dir{>}}+(-3.5,1)*{a_1},
"x"; "z" **\crv{+(5,5)}?/0pt/*{\dir{<}}+(3, -1)*{1},
"y"; "w" **\crv{+(-5,5)}?/0pt/*{\dir{>}}+(-2,1)*{1},
"y"; "w" **\crv{+(5,5)}?/0pt/*{\dir{<}}+(4, -1)*{a_n},
"z"; "w" **\crv{~**\dir{..}+(-10,-10)},
\end{xy}} - {\begin{xy}
(-10,5) *{\bullet}="x",
(10,5) *{\bullet}= "y",
(-10, -5) *{\bullet}= "z",
(10, -5) *{\bullet}= "w",
"y"; "x" **{\dir{-}}?/0pt/*{\dir{<}}+(0,3)*{\frac{1}{a_0}},
"x"; "z" **\crv{+(-5,5)}?/0pt/*{\dir{>}}+(-3.5,1)*{a_1},
"x"; "z" **\crv{+(5,5)}?/0pt/*{\dir{<}}+(3, -1)*{1},
"y"; "w" **\crv{+(-5,5)}?/0pt/*{\dir{>}}+(-2,1)*{1},
"y"; "w" **\crv{+(5,5)}?/0pt/*{\dir{<}}+(4, -1)*{a_n},
"z"; "w" **\crv{~**\dir{..}+(-10,-10)},
\end{xy}} \;.\label{necklacedecompsum}\ea

To avoid extreme notational complexity in keeping track of labels of graphs, we introduce some notation.

\begin{dfn}
Define a set  $\bm{n} = \{1 \ldots n\}$. We define $\bm{a_n}$ to be the $n$-tuple $(a_1, \ldots, a_n)$, and for any $S \subset \bm{n}$, $\bm{a}_{\bm{n} \setminus S} = (a_1, \ldots, \hat{a}_S \ldots a_n)$ is the $n-|S|$-tuple with the elements labeled by $s \in S$ removed.
\label{setdfns}\end{dfn}

\begin{lem}
The sum of graphs $\varepsilon^n(a_0, \bm{a_n})$ is completely decomposable.
\label{epsiloncompdecomp}\end{lem}

\begin{proof}
By direct calculation, \bas \D \varepsilon^n(a_0, \bm{a_n}) =
\sum_{i=1}^n \left( \varepsilon^{n-1}(a_0, \bm{a}_{\bm{n}\setminus i}) -
\varepsilon^{n-1}(a_0 a_i, \bm{a}_{\bm{n}\setminus i}) \right) \cdot
G_0(a_i).\eas The proof follows from induction.
\end{proof}

Therefore, we can construct the element of minimally decomposable element of $\BsG$ it defines.

Recall that \bas [a_1 | \ldots | a_n] \sha [b_1 |
  \ldots | b_m] \eas is the shuffle product of the ordered sets $(a_1,
\ldots , a_n)$ and $(b_1, \ldots, b_m)$.

In particular, for $a, b \in \sGoneL{}{|1}$, \bas a \sha b = a|b + b|a
\;.\eas The shuffle product $a \sha b \in \ker \mu$. That is \ba \mu(a
\sha b) = 0\;.  \label{shainker}\ea

\begin{lem}The element \bas \bm{\varepsilon^n}(a_0, \bm{a_n}) = \sum_{S
  \subset \bm{n} }(-1)^{|S|}\sum_{J \subseteq S} (-1)^{|J|}[\varepsilon^{n-|S|}(a_0\prod_{j \in J}a_j,  \bm{a}_{\bm{n}\setminus S})| \sha_{s \in S} \varepsilon^0(a_s) ] \eas
is in $H^0(\BsG)$. \label{HBGelement}\end{lem}

\begin{proof}
Recall that since $\varepsilon^n \in \sGoneL{n+1}{1}$, it defines an element of degree $0$ in $\BsG$.

Consider the contribution to $\bm{\varepsilon^n}(a_0, \bm{a_n}) \in \BsG_{k+1}^0$. We compute $\D +\mu$ on this term. By Lemma \ref{epsiloncompdecomp}, \bmls \D \sum_{\begin{subarray}{c} |S|=k\\ J \subseteq S\end{subarray}} (-1)^{|J|}[\varepsilon^{n-k}(a_0\prod_{j \in J}a_j, \bm{a}_{\bm{n}\setminus S})| \sha_{s \in S} \varepsilon^0(a_s)] = \\  \sum_{\begin{subarray}{c} |S|=k \\ i \in n\setminus S \\ J \subseteq S \cup i \end{subarray}} (-1)^{|J|}[ \varepsilon^{n-k-1}(a_0\prod_{j \in J}a_j, \bm{a}_{\bm{n}\setminus S})
\cdot G_0(a_i) | \sha_{s \in S} \varepsilon^0(a_s)]\;.\emls

However, by equation \eqref{shainker}, this is \bas  \mu \sum_{\begin{subarray}{c} |S|=k+1 \\  J \subseteq S  \end{subarray}} (-1)^{|J|}[ \varepsilon^{n-k-1}(a_0\prod_{j \in J}a_j, \bm{a}_{\bm{n}\setminus S}) |\varepsilon^0(a_i)| \sha_{s \in S} \varepsilon^0(a_s)]
\;.\eas

Therefore \bas (\D + \mu)
(\bm{\varepsilon^n}(a_0, \bm{a_n})) = 0\;.\eas
\end{proof}

Therefore, $\bm{\varepsilon^n}(a_0, \bm{a_n})$ defines a class in
$H^0(\BsG)$, as stated in remark \ref{mindecomprmk}.

\begin{dfn}
We write $\epclass \in H^0(\BsG)$ to be the class defined by
$\bm{\varepsilon^n}(a_0, \bm{a_n})$. \label{epclass}
\end{dfn}

This choice of notation emphasises that this is the class in
$H^0(\BsG)$ associated to an element in $\sGoneL{\bullet}{1}$ with
completely decomposable boundary.

\section{Hodge realization \label{Hodgerealization}}

In this section we describe the Hodge realization for a number field $k$ for our category and compute some examples. We follow the approach to constructing a Hodge realization as described in \cite{BlochKriz} sections 7 and 8, and in \cite{Kimura}.  Namely, we first note that the Hodge realization as constructed in section 7 can be defined independently of choice.  However, as noted at the beginning of section 8, this construction is not very amenable to computation, and a second description of the Hodge realization functor is given.  In this paper we will restrict to this second description of the Hodge realization. Namely, we explicitly
construct a co-module $J$ of $\sH_T = H^0(\BsG)$ and construct a
natural mixed Tate Hodge
structure on $J$. This, as in  \cite{GGL05}, provides the Hodge realization
for our graphical structure as $J$ associates a natural mixed Tate Hodge structure
on any graded co-module $\sM$ of $\sH_T$.

In the context of the graphs, the $\Q$ mixed Tate Hodge structure is given by the rational lattice \bas H_{\Q} = H^0(\BsTsG) \;. \eas where $\sToneL{}{}$ is a right $\sGoneL{}{}$ module, and $\BsTsG$ is the corresponding cyclic bar construction. Both filtrations are induced from the weights of graphs (or the codimension of the corresponding cycles), as defined in Section \ref{algofgraphs}. These are introduced in detail in Section \ref{topaugadmisgraphs}.

\subsection{Topologically augmented admissible graphs\label{topaugadmisgraphs}}

 As in \cite{BlochKriz}, in order to create the construction outlined above, one must define a set of topologically supported cycles in $\square^n$.

\begin{dfn} Let $\sZ_{top}^\bullet(\Delta_\bullet, \square^{2\bullet - \star})$ be the free abelian group (vector space) generated by admissible algebraic cycles supported on the image of a smooth map $\sigma: \Delta_\bullet \rightarrow \bP(\C)^{2\bullet - \star}$ of codimension $\bullet$ and algebraic degree $\star$. Then define a vector space $\sZ_{top} = \bigoplus_{\bullet, \star} \Alt \sZ_{top}^\bullet(\Delta_\bullet, \square^{2\bullet - \star})$.  \label{Ztopdfn}\end{dfn}

 For ease of notation, we assume these are under the alternating projection. These topological cycles define a means to pass from the algebraic cycles they support to integrals. In particular, given a completely decomposable element, $\varepsilon \in B(\Asm)$,  $[\varepsilon] \in H^0(B(\Asm))$,  one considers the element  $1 \otimes \varepsilon \in B(\sZ_{top}, \Asm)$. This does not define a cohomology class. Namely, it is not completely decomposable. The task then is to find an element $\xi \in B(\sZ_{top}, \Asm)$ such that $1 \otimes \varepsilon + \xi$ is completely decomposable. That is \bas [1 \otimes \varepsilon + \xi] \in H^0(B(\sZ_{top}, \Asm)) \;.\eas It is worth noting that while the cohomology class thus defined is unique, the element $\xi$ need not be. In particular, in the example worked out in Section \ref{necklacehodge}, the given $\xi$ is by no means the only possible construction.

In the context of graphs, we parallel this construction by defining topologically augmented admissible graphs, which, under a natural extension of the DGA homomorphism $Z$ defined in Section \ref{homomorphism}, correspond to elements of $\sZ_{top}^\bullet(\Delta_\bullet, \square^{2\bullet - *})$. These topologically augmented graphs generate a $\sGoneL{}{}$ module, which we develop in this section. First we establish some notation.

Let $\Delta_n \subset \R^n$, be the standard real $n$ simplex.  Let $C^\infty(n,m)$ be the set of smooth maps from $\Delta_n$ to $(\bP(\C)^1)^N$ of dimension $m$. Here $N$ is an arbitrary integer $N \geq n$.

\begin{dfn}
We say that $m$ is the simplicial dimension of maps in $C^\infty(n, m)$.
\end{dfn}

Note that $\sigma$ need not be injective, that is $m$ may be less than $n$. In particular, $\C^\infty(n, 0)$ consists of all constant maps from $\Delta_n$. We view $C^\infty(n, m)$ as a chain complex, $C(n)_{m}$.

We parameterize $\Delta_n$ by an ordered  set  as usual $ 0  \leq t_1  \leq \ldots t_n
\leq 1 $, sometimes writing $0 = t_0$ and $1 = t_{n+1}$. Then any $\sigma \in C(n)_m$ is a continuous function of $\{t_1, \ldots t_n\}$.

\begin{dfn} Given the standard face and degeneracy maps on $\delta_n$, $s_i$ and $d_i$ respectively, for any subset $I \in \{ t_0, \ldots t_n\}$ of size $|I| = p$, we write $d_I$ to be the standard codimension $p$ face map. \end{dfn}

Let $\bm{n} = \{1, \ldots, n\}$ as before. Any continuous map $\sigma \in C(n)_m$ can be written in terms of codimension $n-m$ face maps. That is, there is a set $I \in \bm{n}$ and $\sigma' \in C(m)_m$ such that \bas \sigma = d_{I*} \sigma'\;.\eas The degeneracy maps define a differential on the chain complex $C(n)_m$. In particular, we write \bas \delta_i : C(m)_m \rightarrow C(m)_{m-1}   \\ \sigma \rightarrow  d_{i*}s_{i*} \sigma \;, \eas with $\delta = \sum_{i=0}^m (-1)^i \delta_i$. More generally, for $\sigma \in C(n)_m$, where $\sigma = d_I \sigma'$, write \ba \delta_i :  C(n)_m \rightarrow C(n)_{m-1}  \nonumber \\ \sigma \rightarrow  d_{I*}\delta_i \sigma'  \label{deltaidef} \;.\ea The degeneracy maps on $\Delta_n$ induce a differential on the chain complex $C(n)_m$, namely $\delta =  \sum_{i=0}^m (-1)^i \delta_i$. That is, $\delta \circ \delta = 0$ and \bas \delta: C(n)_m \rightarrow C(n)_{m-1}\;.\eas Therefore, we have shown

\begin{lem} For a fixed $n$, $(C(n)_*, \delta)$ is a chain complex. \end{lem}

\begin{rem}Henceforth, we write elements of $C(m)_m$ as $\sigma_m$. For the rest of this paper, as in the prequel, the symbol $\bullet$ will always correspond to the codimension of a cycle (loop number of a graph). The symbol $\star$ will always correspond to the algebraic degree, and the symbol $*$ always the simplicial dimension of the graph.
\end{rem}

Given this notation, we define the right module of topologically augmented admissible graphs. Generators of this algebra are given by the pair $\sigma \in C(n)_*$ and an admissible graph $G \in \sGoneL{n}{i}$. In particular, the topologically augmented graph $(G, \sigma)$ has edges labeled not by elements of $k^\times$ as usual, but by the image of $\sigma$. For $t \in \Delta_n$, write $\sigma(t)$ as the $(2n-i)$-tuple $\sigma(t) = (\sigma_1(t), \ldots , \sigma_{2n-i}(t))$. The coordinate $\sigma_i(t)$ labeles the edge $e \in E(G)$ that is in the $i^{th}$ position, that is, such that $\omega(e) = i$. There is a natural extension of the vector space homomorpism $Z$ defined in Section \ref{homomorphism} to the topologically augmented admissible graphs, such that each graph maps to a topologically supported cycle in $\sZ_{top}$.

For each $\sigma \in C(\bullet)_*$, and $t \in \Delta_\bullet$ such that $\sigma_{\omega(e)}(t) \neq 0, \;\infty$ for any $e \in E(G)$, the pair $(G, \sigma(t))$ defines a graph in $\Gprg/(\simord,\simori,\simv)$. If $\sigma_{\omega(e)}(t) \neq 0, \;\infty$, we say that $(G, \sigma(t))$ is the trivial graph. As we show in Lemma \ref{Zextend}, graphs with such labels correspond to algebraic cycles with $1$ in the appropriate coordinate. In particular, for a general $\sigma$, the labels $\sigma(t)$ need not correspond to an admissible labeling of the underlying graph $G$. We wish to consider pairs $(G, \sigma(t))$ which evaluate to admissible graphs almost everywhere on $\Delta_n$. Such $\sigma \in C(\bullet)_*$ are called admissible simplices for $G$.

\begin{dfn}
A map $\sigma \in C(\bullet)_*$ is admissible for a graph $G$ if the following hold: \begin{enumerate} \item Let $\delta_J (\sigma)$ indicate the degeneracy map onto the face opposite that defined by $J$ in $\delta_\bullet$. For all $J$, each loop of the augmented graph $(G, \delta_J \sigma)$ does not have loop coefficient $1$ almost everywhere on $\Delta_\bullet$.
\item For all $e \in E(G)$, if there exists a $t \in \Delta_{|E(G)|}$
  such that $\sigma_{\omega(e)}(t)= 0$, there exists an $e' \in
  E(G)$ such that $\sigma_{\omega(e')}(t)= \infty$. Therefore, the
  cycle $Z(G, \sigma(t))$ is trivial. \item Writing $\delta \sigma = \sum_{i= 0}^\bullet (-1)^i \delta_i \sigma$, there is some $i$ for which no coordinate of $\delta_i \sigma$ is $\infty$.
\end{enumerate}
\label{topoadmisconds}
\end{dfn}

We are now ready to define the vector space of admissible topologically augmented graphs.

\begin{dfn}
Let $\sToneL{\bullet}{2\bullet- *}$ be the vector space of admissible codimension $\bullet$ graphs with edges labeled by an admissible labeling, $\sigma \in C(\bullet)_*$.
\end{dfn}

\begin{eg}
Consider the necklace graph $G^L(a_0, \ldots a_n) \in \sGoneL{n+1}{1}$, \bas G^L(a_0, \ldots a_n) = \begin{xy}
(-10,5) *{\bullet}="x",
(10,5) *{\bullet}= "y",
(-10, -5) *{\bullet}= "z",
(10, -5) *{\bullet}= "w",
"y"; "x" **{\dir{-}}?/0pt/*{\dir{>}}+(0,2)*{a_0},
"x"; "z" **\crv{+(-5,5)}?/0pt/*{\dir{>}}+(-3.5,1)*{a_1},
"x"; "z" **\crv{+(5,5)}?/0pt/*{\dir{<}}+(3, -1)*{1},
"y"; "w" **\crv{+(-5,5)}?/0pt/*{\dir{>}}+(-2,1)*{1},
"y"; "w" **\crv{+(5,5)}?/0pt/*{\dir{<}}+(4, -1)*{a_n},
"z"; "w" **\crv{~**\dir{..}+(-10,-10)},
\end{xy} \;. \eas

There is a constant map $\sigma \in C(n+1)_0$ of the form $\sigma_0(\Delta_{n+1}) = (a_0, 1, a_1, \ldots,1, a_{n})$. As this has $0$ dimensional topological support, this is the constant map. The pair $(G^L, \sigma_0) \in \sToneL{n+1}{2n+2}$ is a trivially topologically augmented graph. That is $(G^L, \sigma_0) = G^L \in \sGoneL{n+1}{1}$.

Consider a different map, $\sigma' \in C(n+1)_2$, of the form $\sigma'(\Delta_{n+1}) = (\frac{a_0}{t_n+1}, 1, a_1, \ldots,1, \frac{t_{n+1}a_{n}}{t_n})$. Then the pair \bas (G^L, \sigma') = {\begin{xy}
(-10,5) *{\bullet}="x",
(10,5) *{\bullet}= "y",
(-10, -5) *{\bullet}= "z",
(10, -5) *{\bullet}= "w",
"y"; "x" **{\dir{-}}?/0pt/*{\dir{>}}+(0,3)*{\frac{a_0}{t_n+1}},
"x"; "z" **\crv{+(-5,5)}?/0pt/*{\dir{>}}+(-3.5,1)*{a_1},
"x"; "z" **\crv{+(5,5)}?/0pt/*{\dir{<}}+(3, -1)*{1},
"y"; "w" **\crv{+(-5,5)}?/0pt/*{\dir{>}}+(-2,1)*{1},
"y"; "w" **\crv{+(5,5)}?/0pt/*{\dir{<}}+(5, -1)*{\frac{a_{n}t_{n+1}}{t_n}},
"z"; "w" **\crv{~**\dir{..}+(-10,-10)},
\end{xy}} \; \eas is an element of $\sToneL{n+1}{2n}$.

%Note that we may write $\sigma(n+1)_2$ as $d_{I*}\sigma'$, with $\sigma' \in C(2)_2$ the map $(\frac{a_0}{t_2}, 1, a_1, \ldots,1, \frac{t_{2}a_{n}}{t_1}) $, and $I= \{1, \ldots , n-1\}$.
\label{necklaceinT1L}\end{eg}

Note that $\sToneL{\bullet}{2\bullet - *}$ is not an algebra. In particular, there is no natural product structure on $C(n)_*$. For general $(G, \sigma),  \in \sToneL{n}{2n - m}$, and $(G', \sigma') \in \sToneL{n'}{2n'- m'}$, the product is given by the graph $(G G', \sigma\times \sigma')$. If $\sigma = d_{I*}\sigma_m$ and $\sigma'= d_{I'*}\sigma_{m'}$, then $\sigma\times \sigma' =d_{I*}\sigma_{m} \times d_{I'*}\sigma_{m'}$ which does not correspond to a face of $\Delta_{n +n'} $ unless either $m = 0$ or $m'=0$. Therefore, we consider $\sToneL{\bullet}{2\bullet-*}$ as a $\sGoneL{\bullet}{\star}$ module.

There is an inclusion of the algebra of admissible non-augmented graphs into $\sToneL{}{}$.

\begin{eg}
There is an inclusion $\sGoneL{\bullet}{\star} \hookrightarrow \sToneL{\bullet}{2 \bullet}$. Any graph $G \in \sGoneL{\bullet}{\star}$ can be written as $(G, d_{I*}\sigma_0)$ via the constant map \bas \sigma_0(\Delta_\bullet) = (a_1 \ldots
a_{|E(G)|}) \;,\eas where $a_{\omega(e)}$ is the label of edge $e \in
E(G)$.
\label{inclusioneg}
\end{eg}

\begin{prop}
The vector space $\sToneL{}{}$ is a
$\sGoneL{}{}$ module. \label{graphmodule}\end{prop}

\begin{proof}
As done in example \ref{inclusioneg}, write $G \in
\sGoneL{\bullet}{\star}$, as $(G, \sigma) \in
\sToneL{\bullet}{2\bullet}$. For $n= 2\bullet - \star$, write $\sigma = d_{I*}\sigma_0 \in C(n)_0$. Consider and augmented graph $(G', \sigma') \in \sToneL{\bullet'}{2\bullet'-m}$ with $\sigma' = d_{I'*}\sigma_m \in C(2\bullet' - \star')_m$.

In general, we cannot write $(G, \sigma)(G', \sigma')= (G G', \sigma \times \sigma')$ as an element in $\sToneL{}{}$. However, since $\sigma \in C(n)_0$, we can rewrite this as $(G G', d_{(I'')*}\sigma_m)$, where $|I''|= |n+n'-m|$. This is an augmented graph.

Therefore, the product of a non-topologically augmented graph, $G \in \sGoneL{\bullet}{\star}$ with an augmented one $(G', \sigma') \in \sToneL{\bullet'}{2\bullet'-m}$, \bas (G,
\sigma_0) \cdot (G', \sigma') = (G \cdot G', (\sigma_0,\sigma'))
\in \sToneL{\bullet+\bullet'}{2(\bullet + \bullet') -m}\;. \eas

This gives the module structure.
\end{proof}

The vector space $\sToneL{\bullet}{2\bullet-*}$ is a bigraded vector space. We may write \bas \sToneL{}{} = \bigoplus_{\begin{subarray}{c}0 \leq \bullet \\ 0\leq  * \leq n\end{subarray}} \sToneL{\bullet}{2\bullet-*}\;. \eas

Finally, we consider $\sToneL{\bullet}{2\bullet - *}$ as a complex. The module
 has two natural differential structures on it, induced by the topological differential $\delta$ on $\Delta_n$, and the algebraic differential $\D$ on $\sGoneL{}{}$. Before defining these explicitly and the associated bicomplex structure on augmented graphs, it is necessary to introduce a shifted vector space, $\sTt$.

\begin{dfn} For $(G, \sigma) \in \sToneL{}{}$ we define a twisted module $\sTt$, where the grading of each element is shifted from that of $\sToneL{}{}$ by the dimension of the range of $\sigma$, i.e the number of edges of the graph $G$. That is, for $G \in \sGoneL{\bullet}{\star}$ and $(G, \sigma) \in \sToneL{\bullet}{2\bullet -*}$, the same element is in $\sToneLt{\bullet}{\star_t} : = \sToneL{\bullet}{2\bullet-*- n}$, for $n = 2\bullet -\star$. Henceforth define a topologically twisted degree $\star_t:= \star - *$ to be the difference between the algebraic degree and topological dimension. Write \bas \sTt= \bigoplus_{\bullet, \star_t} \sToneLt{\bullet}{\star_t}\;.\eas\end{dfn}

For $\sigma_m \in C(n)_m$, write $\sigma_m = d_{I*} \sigma'$ for some $\sigma' \in C(m)_m$. The topological
differential, $\delta$ is induced by the differential on the chain $C(n)_m$, defined in equation
\eqref{deltaidef}: \ba \delta: \sToneLt{\bullet}{\star_t} & \rightarrow
\sToneLt{\bullet}{\star_t +1} \nonumber \\ (G, \sigma_m) & \rightarrow \sum_{i=0}^m
(-1)^{i} (G, \delta_i\sigma_m) \;. \label{deltagraphs}\ea This is a degree one differential operator on $\sToneL{}{}$.

The algebraic differential $\D$ is induced from the differential $\D$ on $\sGoneL{}{}$. On $\sTt$, vertex rescaling is a direct generalization of rescaling on $\sGoneL{}{}$, allowing one to rescale by functions $\sigma \in C^\infty(|E(G)|, |E(G)|)$. For $s_e$ and $t_e$ the source and terminal vertices of $e \in G$, write \bas (\D_e \sigma)_{\omega(e')} = \begin{cases} 1 & \textrm{if } e = e'\\ \sigma_{\omega(e')} & \textrm{if } s_e \textrm{ not a vertex of } e'\\ \sigma_{\omega(e')}\sigma_{\omega(e)} & \textrm{if } s_e = t_{e'} \\\sigma_{\omega(e')}/\sigma_{\omega(e)}  & \textrm{if } s_e = s_{e'}\end{cases} \eas as one expects from vertex rescaling and definition \ref{Desourcedef}. Then \bas \D : \sToneLt{\bullet}{\star_t} & \rightarrow
\sToneLt{\bullet}{\star_t+1} \nonumber \\ (G, \sigma_m) & \rightarrow \sum_{e\in E(G)}
(-1)^{\omega(e)} (\D_eG, \D_e\sigma) \; \eas which is a degree one differential operator on $\sTt$.

The topologically augmented graphs correspond to the vector space of topologically supported admissible algebraic cycles $\mathcal{Z}_{top}^{\bullet}(\Delta_*, \square^{2\bullet - \star})$.

\begin{lem}The map $Z$ defined in Section \ref{homomorphism} extends to a module homomorphism \bas Z: \sToneLt{}{} \rightarrow \sZ_{top}\;, \eas as defined in Definition \ref{Ztopdfn}. \label{Zextend}\end{lem}

\begin{proof}
Each edge of the augmented graph $(G, \sigma_m)$ defines a coordinate $\phi_{\omega(e)} = \onem{s_e}{\sigma_{m, \omega(e)} t_e}$, where $s_e$ and $t_e$ are the source and the target vertices of the edge $e$ as usual. Then $\phi = \Alt (\phi_1, \ldots , \phi_n)$ parametrizes an algebraic cycles supported on an $m$ simplex in $\square^n$.

It remains to check that $Z(G, \sigma_m)$ is an admissible topologically supported cycle. By definition \ref{topoadmisconds}, the loop number of any loop in $(G, \sigma_m)$ is not one almost everywhere in $\Delta_m$ or on any of its faces. If $\sigma_{\omega(e)}(t) = 0$ for some $t \in \sigma_m$, then the cycle $Z(G, \sigma_m(t))$ is trivial, as the corresponding coordinate is $1$. Therefore, by condition 2 of definition \ref{topoadmisconds}, if there is some $t \in \sigma_m$ and an edge $e \in E(G)$, $Z(G, \sigma_m(t))$ is trivial. Therefore, by Theorem \ref{admissible}, $Z(G, \sigma_m)$ is admissible almost everywhere on $\Delta_m$.
\end{proof}

The third condition in definition \ref{topoadmisconds} gives rise to the following statement.

\begin{lem}
The image of $\sTt$ under $Z$ is an acyclic chain complex under $\delta$.
\end{lem}

\begin{proof}
Equation \eqref{deltagraphs}, shows that $\sTt$ is a chain complex on under $\delta$. In particular \bas \delta: \sToneLt{\bullet}{\star_t} \rightarrow \sToneLt{\bullet}{\star_t+1} \;.\eas The third condition of definition \ref{topoadmisconds} imposes acyclicity. By Lemma \ref{Zextend}, if $\delta_i \sigma$ has a coordinate set at $\infty$, then $Z(G, \delta_i \sigma)$ is a trivial cycle. Requiring that there is some face of $\Delta_\bullet$ such that $(\delta_i \sigma)_{\omega(e)} \neq \infty$ for all $e \in E(G)$ implies that $\delta Z(G, \sigma) \neq 0$. That is, that the image of $Z(\sTt)$ is an acyclic chain complex under $\delta$.
\end{proof}

\begin{eg} \label{necklacediffeg}
In this example, we augment the sum of graphs $\varepsilon^n(a_0, \ldots a_n)$ defined in equation \ref{necklacedecompsum} by a $2$ dimensional support, as in in example \ref{necklaceinT1L}. First, recall notation from Definition \ref{setdfns}. Writing $\bm{n} = \{1\ldots n\}$, define an $n$-tuple $\bm{a_n} = (a_1, \ldots ,a_n)$. Similarly, for any $S \subset \bm{n}$, write $\bm{a}_{\bm{n} \setminus S} = (a_1, \ldots, \hat{a}_S, \ldots ,a_n)$ to be the same $n$-tuple with the elements $\{a_s| s \in S\}$ removed. Then write the augmented sum of graphs \bas (\varepsilon^n, \sigma(a_0, \bm{a_n})_2) = {\begin{xy}
(-10,5) *{\bullet}="x",
(10,5) *{\bullet}= "y",
(-10, -5) *{\bullet}= "z",
(10, -5) *{\bullet}= "w",
"y"; "x" **{\dir{-}}?/0pt/*{\dir{>}}+(0,3)*{\frac{a_0}{t_n+1}},
"x"; "z" **\crv{+(-5,5)}?/0pt/*{\dir{>}}+(-3,1)*{a_1},
"x"; "z" **\crv{+(5,5)}?/0pt/*{\dir{<}}+(3, -1)*{1},
"y"; "w" **\crv{+(-5,5)}?/0pt/*{\dir{>}}+(-2,1)*{1},
"y"; "w" **\crv{+(5,5)}?/0pt/*{\dir{<}}+(5, -1)*{\frac{a_{n}t_{n+1}}{t_n}},
"z"; "w" **\crv{~**\dir{..}+(-10,-10)},
\end{xy}} - {\begin{xy}
(-10,5) *{\bullet}="x",
(10,5) *{\bullet}= "y",
(-10, -5) *{\bullet}= "z",
(10, -5) *{\bullet}= "w",
"y"; "x" **{\dir{-}}?/0pt/*{\dir{<}}+(0,3)*{\frac{1}{a_0 t_n+1}},
"x"; "z" **\crv{+(-5,5)}?/0pt/*{\dir{>}}+(-3,1)*{a_1},
"x"; "z" **\crv{+(5,5)}?/0pt/*{\dir{<}}+(3, -1)*{1},
"y"; "w" **\crv{+(-5,5)}?/0pt/*{\dir{>}}+(-2,1)*{1},
"y"; "w" **\crv{+(5,5)}?/0pt/*{\dir{<}}+(5, -1)*{\frac{a_{n}t_{n+1}}{t_n}},
"z"; "w" **\crv{~**\dir{..}+(-10,-10)},
\end{xy}} \;.\eas Here $\sigma(a_0, \bm{a}_{\bm{n\setminus S}})_2 \in C(n -|S|+1 )_2$ is a labeling on the decomposable sum of $n -|S|$ beaded necklace.

Then the topological differential is \bas \delta(\varepsilon^n, \sigma(a_0, \bm{a_n})_2) =  (-1)^0({\begin{xy}
(-10,5) *{\bullet}="x",
(10,5) *{\bullet}= "y",
(-10, -5) *{\bullet}= "z",
(10, -5) *{\bullet}= "w",
"y"; "x" **{\dir{-}}?/0pt/*{\dir{>}}+(0,3)*{\frac{a_0}{t_n+1}},
"x"; "z" **\crv{+(-5,5)}?/0pt/*{\dir{>}}+(-3,1)*{a_1},
"x"; "z" **\crv{+(5,5)}?/0pt/*{\dir{<}}+(3, -1)*{1},
"y"; "w" **\crv{+(-5,5)}?/0pt/*{\dir{>}}+(-2,1)*{1},
"y"; "w" **\crv{+(5,5)}?/0pt/*{\dir{<}}+(5, -1)*{\frac{a_{n}t_{n+1}}{0}},
"z"; "w" **\crv{~**\dir{..}+(-10,-10)},
\end{xy}} - {\begin{xy}
(-10,5) *{\bullet}="x",
(10,5) *{\bullet}= "y",
(-10, -5) *{\bullet}= "z",
(10, -5) *{\bullet}= "w",
"y"; "x" **{\dir{-}}?/0pt/*{\dir{<}}+(0,3)*{\frac{1}{a_0t_n+1}},
"x"; "z" **\crv{+(-5,5)}?/0pt/*{\dir{>}}+(-3,1)*{a_1},
"x"; "z" **\crv{+(5,5)}?/0pt/*{\dir{<}}+(3, -1)*{1},
"y"; "w" **\crv{+(-5,5)}?/0pt/*{\dir{>}}+(-2,1)*{1},
"y"; "w" **\crv{+(5,5)}?/0pt/*{\dir{<}}+(5, -1)*{\frac{a_{n}t_{n+1}}{0}},
"z"; "w" **\crv{~**\dir{..}+(-10,-10)},
\end{xy}}) \\ (-1)^1( {\begin{xy}
(-10,5) *{\bullet}="x",
(10,5) *{\bullet}= "y",
(-10, -5) *{\bullet}= "z",
(10, -5) *{\bullet}= "w",
"y"; "x" **{\dir{-}}?/0pt/*{\dir{>}}+(0,3)*{\frac{a_0}{t_n+1}},
"x"; "z" **\crv{+(-5,5)}?/0pt/*{\dir{>}}+(-3,1)*{a_1},
"x"; "z" **\crv{+(5,5)}?/0pt/*{\dir{<}}+(3, -1)*{1},
"y"; "w" **\crv{+(-5,5)}?/0pt/*{\dir{>}}+(-2,1)*{1},
"y"; "w" **\crv{+(5,5)}?/0pt/*{\dir{<}}+(5, -1)*{a_{n}},
"z"; "w" **\crv{~**\dir{..}+(-10,-10)},
\end{xy}} - {\begin{xy}
(-10,5) *{\bullet}="x",
(10,5) *{\bullet}= "y",
(-10, -5) *{\bullet}= "z",
(10, -5) *{\bullet}= "w",
"y"; "x" **{\dir{-}}?/0pt/*{\dir{<}}+(0,3)*{\frac{1}{a_0t_n+1}},
"x"; "z" **\crv{+(-5,5)}?/0pt/*{\dir{>}}+(-3,1)*{a_1},
"x"; "z" **\crv{+(5,5)}?/0pt/*{\dir{<}}+(3, -1)*{1},
"y"; "w" **\crv{+(-5,5)}?/0pt/*{\dir{>}}+(-2,1)*{1},
"y"; "w" **\crv{+(5,5)}?/0pt/*{\dir{<}}+(5, -1)*{a_{n}},
"z"; "w" **\crv{~**\dir{..}+(-10,-10)},
\end{xy}})  \\ (-1)^2({\begin{xy}
(-10,5) *{\bullet}="x",
(10,5) *{\bullet}= "y",
(-10, -5) *{\bullet}= "z",
(10, -5) *{\bullet}= "w",
"y"; "x" **{\dir{-}}?/0pt/*{\dir{>}}+(0,3)*{a_0},
"x"; "z" **\crv{+(-5,5)}?/0pt/*{\dir{>}}+(-3,1)*{a_1},
"x"; "z" **\crv{+(5,5)}?/0pt/*{\dir{<}}+(3, -1)*{1},
"y"; "w" **\crv{+(-5,5)}?/0pt/*{\dir{>}}+(-2,1)*{1},
"y"; "w" **\crv{+(5,5)}?/0pt/*{\dir{<}}+(5, -1)*{\frac{a_{n}}{t_n}},
"z"; "w" **\crv{~**\dir{..}+(-10,-10)},
\end{xy}} - {\begin{xy}
(-10,5) *{\bullet}="x",
(10,5) *{\bullet}= "y",
(-10, -5) *{\bullet}= "z",
(10, -5) *{\bullet}= "w",
"y"; "x" **{\dir{-}}?/0pt/*{\dir{>}}+(0,3)*{\frac{1}{a_0}},
"x"; "z" **\crv{+(-5,5)}?/0pt/*{\dir{>}}+(-3,1)*{a_1},
"x"; "z" **\crv{+(5,5)}?/0pt/*{\dir{<}}+(3, -1)*{1},
"y"; "w" **\crv{+(-5,5)}?/0pt/*{\dir{>}}+(-2,1)*{1},
"y"; "w" **\crv{+(5,5)}?/0pt/*{\dir{<}}+(5, -1)*{\frac{a_{n}}{t_n}},
"z"; "w" **\crv{~**\dir{..}+(-10,-10)},
\end{xy}}) \;.\eas  The first two terms in this sum correspond to trivial graphs.

Recall from Section \eqref{G0dfn} that $G_0 (a)= \Go{a}$ is the graph with a single edge and a single loop labelled by $a$.

The algebraic differential on this graph is \bmls \D (\varepsilon^n, \sigma(a_0, \bm{a_n})_2) = \sum_{i=1}^{n-1} (\varepsilon^{n-1}, \sigma(a_0, \bm{a}_{\bm{n} \setminus i})_2) - (\varepsilon^{n-1}, \sigma(a_0a_i, \bm{a}_{\bm{n} \setminus i})_2)G_0(a_i) \\ + \left( {\begin{xy}
(-8,5) *{\bullet}="x",
(8,5) *{\bullet}= "y",
(-8, -5) *{\bullet}= "z",
(8, -5) *{\bullet}= "w",
"y"; "x" **{\dir{-}}?/0pt/*{\dir{>}}+(0,3)*{\frac{a_0}{t_n}},
"x"; "z" **\crv{+(-5,5)}?/0pt/*{\dir{>}}+(-3,1)*{a_1},
"x"; "z" **\crv{+(5,5)}?/0pt/*{\dir{<}}+(3, -1)*{1},
"y"; "w" **\crv{+(-5,5)}?/0pt/*{\dir{>}}+(-2,1)*{1},
"y"; "w" **\crv{+(5,5)}?/0pt/*{\dir{<}}+(5, -1)*{a_{n-1}},
"z"; "w" **\crv{~**\dir{..}+(-10,-10)},
\end{xy}} - {\begin{xy}
(-8,5) *{\bullet}="x",
(8,5) *{\bullet}= "y",
(-8, -5) *{\bullet}= "z",
(8, -5) *{\bullet}= "w",
"y"; "x" **{\dir{-}}?/0pt/*{\dir{<}}+(0,3)*{\frac{1}{a_0t_n}},
"x"; "z" **\crv{+(-5,5)}?/0pt/*{\dir{>}}+(-3,1)*{a_1},
"x"; "z" **\crv{+(5,5)}?/0pt/*{\dir{<}}+(3, -1)*{1},
"y"; "w" **\crv{+(-5,5)}?/0pt/*{\dir{>}}+(-2,1)*{1},
"y"; "w" **\crv{+(5,5)}?/0pt/*{\dir{<}}+(5, -1)*{a_{n-1}},
"z"; "w" **\crv{~**\dir{..}+(-10,-10)},
\end{xy}} \right. \\ \left. - {\begin{xy}
(-8,5) *{\bullet}="x",
(8,5) *{\bullet}= "y",
(-8, -5) *{\bullet}= "z",
(8, -5) *{\bullet}= "w",
"y"; "x" **{\dir{-}}?/0pt/*{\dir{>}}+(0,3)*{\frac{a_0a_n}{t_n}},
"x"; "z" **\crv{+(-5,5)}?/0pt/*{\dir{>}}+(-3,1)*{a_1},
"x"; "z" **\crv{+(5,5)}?/0pt/*{\dir{<}}+(3, -1)*{1},
"y"; "w" **\crv{+(-5,5)}?/0pt/*{\dir{>}}+(-2,1)*{1},
"y"; "w" **\crv{+(5,5)}?/0pt/*{\dir{<}}+(5, -1)*{a_{n-1}},
"z"; "w" **\crv{~**\dir{..}+(-10,-10)},
\end{xy}} + {\begin{xy}
(-8,5) *{\bullet}="x",
(8,5) *{\bullet}= "y",
(-8, -5) *{\bullet}= "z",
(8, -5) *{\bullet}= "w",
"y"; "x" **{\dir{-}}?/0pt/*{\dir{<}}+(0,3)*{\frac{1}{a_0a_nt_n}},
"x"; "z" **\crv{+(-5,5)}?/0pt/*{\dir{>}}+(-3,1)*{a_1},
"x"; "z" **\crv{+(5,5)}?/0pt/*{\dir{<}}+(3, -1)*{1},
"y"; "w" **\crv{+(-5,5)}?/0pt/*{\dir{>}}+(-2,1)*{1},
"y"; "w" **\crv{+(5,5)}?/0pt/*{\dir{<}}+(5, -1)*{a_{n-1}},
"z"; "w" **\crv{~**\dir{..}+(-10,-10)},
\end{xy}} \right)G_0(a_n) \;.\emls Due to the form of the augmentation, $\sigma(n+1)_2$, chosen in this example, we may write the second line above as \bas (\varepsilon^{n-1}, \delta_{n-1}\sigma(a_0, \bm{a_{n-1}})_2) - (\varepsilon^{n-1}, \delta_{n-1}(a_0a_n, \bm{a_{n-1}})_2)\;.\eas
\end{eg}

\subsection{The comodule $J$}

We are now ready to define the Hodge comodule $J$.

First we build the circular bar construction $B(\sTt,\sGoneL{}{}, \Q)$. In the sequel, we take the last entry as given, and simply write $\BsTsG$. As in \cite{BlochKriz, KM} and references
therein, we define the $\BsTsG$ on the tensor algebra $\sTt \otimes T(\sGoneL{}{})/ D(\sGoneL{}{})$ as in Definition \ref{Bardfn}.

Consider $(G_0, \sigma) \otimes G_1 \otimes \ldots \otimes G_k \in \BsTsG^k_w$ with $G_i \in \sGoneL{r_i}{w_i}$ for $0 \leq i \leq k$, and $\sigma \in C(r_0)_m$. The total degree of this bar element is $w = w_0 + \sum_{i=1}^k w_i -k-m$.

We define the bicomplex structure on it by  extending the
differentials \eqref{bardiff} and \eqref{barprod} for the bar
construction $(\BsG, \mu, \D)$.

As before, for $j >0$, write $\D_j$ to indicate the operator on $\BsTsG$ that acts as $\D$ on the $j^{th}$ tensor component of $T(\sGoneL{}{})$ and by $(-1)^{\bdeg G_i} (\id)$ on $\sTt$ and the first $j-1$ tensor components of $T(\sGoneL{}{})$ and by $\id$ on the rest. Here $\bdeg G_i$ refers to the \emph{graphical} bar degree of the component, excluding any topological considerations. Hence, for $(G_0, \sigma) \in \sToneLt{\bullet}{\star_t}$, with $\sigma \in C(\bullet)_*$, $\bdeg (G_0, \sigma) = \star_t + * -1 = \star -1$. Define $\D_0$ as $\D+ \delta$ on $\sTt$ and the identity on the other tensor components of the bar element. In this shifted notation, $\D_0$ is a degree one operator on $\sTt$.

For the product, with $(G_0, \sigma)$ as above, define $\mu_j$ as the degree one operator on $\BsTsG$ that acts by $(-1)^{\bdeg G_0 -m}\id$ on the $0^{th}$ tensor component and by $(-1)^{\bdeg G_i}\id$ on the next $(j-1)^{th}$ tensor components of $T(\sGoneL{}{})$, by $(-1)^{\bdeg G_j}\mu$ on the $j^{th}$ and $j+1^{th}$ components, and as the identity on the remaining elements. %In particular for $G_0 \in \sGoneL{\bullet}{\star}$ and $\sigma \in C(\bullet)_{m}$,  \bas \mu_0((G_0, \sigma)\otimes G_1 \otimes \ldots  \otimes G_n) = (-1)^{\bdeg G_0 + \bdeg G_1 - m}((G_0, \sigma_m)G_1 \otimes G_2 \otimes \ldots  \otimes G_n) \eas is the signed right module product structure on $\sTt \otimes \sGoneL{}{}$ and the identity on all other tensor components.

Then, in parallel to \eqref{barprod}, for $\sigma \in C(\bullet)_m$, write \bml \mu [(G_0, \sigma)| G_1 | \ldots |  G_n]  := \sum_{j=0} \mu_j  [ (G_0, \sigma)| G_1 | \ldots |  G_n] \\ = \sum_{j=0}^{n-1} (-1)^{(\sum_{i=0}^j \bdeg G_i)-m} [(G_0, \sigma)| G_1 | \ldots
  | G_j \cdot G_{j+1}| \ldots | G_n] \;. \label{barprodmod} \eml Similarly, in parallel to \eqref{bardiff}, write \bml \D [(G_0, \sigma)| G_1 | \ldots |  G_n]  := \sum_{j=0} \D_j  [ (G_0, \sigma)| G_1 | \ldots |  G_n] \\ = \sum_{j=0}^{n-1} (-1)^{\sum_{i=0}^{j-1} \bdeg G_i}  [(G_0, \sigma)| G_1 | \ldots
  | \D G_j | \ldots |  G_n] \;. \label{bardiffmod} \eml

In parallel to display {\eqref{barbit}}, we explicitly draw a few terms of  $(\BsTsG, \mu, \D)$. Here we use

\ba \xymatrix{ & \vdots & \vdots & \vdots & \vdots \\ \cdots \ar[r]^\mu &
  \BsTsG^3_0 \ar[r]^\mu \ar[u]^\D & \BsTsG^2_1 \ar[r]^\mu
  \ar[u]^\D & \BsTsG^1_2 \ar[r]^\mu \ar[u]^\D & \oplus_n\sToneLt{n}{3}\ar[u]^\D
  \\ \cdots \ar[r]^\mu & \BsTsG^3_{-1} \ar[r]^\mu \ar[u]^\D &
  \BsTsG^2_0 \ar[r]^\mu \ar[u]^\D & \BsTsG^1_1 \ar[r]^\mu
  \ar[u]^\D & \oplus_n\sToneLt{n}{2}\ar[u]^\D \\ \cdots \ar[r]^\mu & \BsTsG^3_{-2} \ar[r]^\mu
  \ar[u]^\D & \BsTsG^2_{-1} \ar[r]^\mu \ar[u]^\D & \BsTsG^1_0
  \ar[r]^\mu \ar[u]^\D & \oplus_n\sToneLt{n}{1} \ar[u]^\D \\ \cdots \ar[r]^\mu &
  \BsTsG^3_{-3} \ar[r]^\mu \ar[u]^\D & \BsTsG^2_{-2} \ar[r]^\mu
  \ar[u]^\D & \BsTsG^1_{-1} \ar[r]^\mu \ar[u]^\D & \oplus_n\sToneLt{n}{0} \ar[u]^\D\\ &
  \vdots \ar[u]^\D & \vdots \ar[u]^\D & \vdots
  \ar[u]^\D & \vdots
  \ar[u]^\D}  \label{cyclicbarbit}\ea

Recall that $\oplus_n\sToneLt{n}{i} = \BsTsG_{i}^0$.

\begin{dfn}
We may now define the comodule $J = H^0(\BsTsG)$.
\end{dfn}

The weight filtration, $W_{2r}= W_{2r-1}$, is induced by the algebraic weight (codimension) filtration on $\BsTsG$. Write \bas \BsTsG(r)= \bigoplus_{k \geq 0}\bigoplus_{\sum_{i=1}^k r_i = r}\sToneLt{s}{\star_t} \otimes \sGoneL{r_1}{\star} \otimes \ldots \otimes \sGoneL{r_k}{\star}\eas to be the tensor product of unaugemted graphs with total codimension $r$. That is, we may write \bas W_r (\BsTsG) =  \bigoplus_{q \leq r} \BsTsG(q) \;.\eas This induces the weight filtration on $J$ in the usual way, $gr_{2r}^W J = gr_{2r-1}^W J = H^0(\BsTsG(r))$. The Hodge filtration is given by $F^k J = \bigoplus_{r \geq k} H^0(\BsTsG(r))$.

Before continuing, we recall the evaluation map on graphs that is at the heart of the Hodge realization functor.

\begin{dfn}Define $\omega_n=\frac{1}{(2\pi
  i)^n} \frac{dz_1}{z_1} \wedge \ldots \wedge \frac{dz_n}{z_n}$  be the logarithmic $n$-form on $\square^N$. \label{formdef} \end{dfn}

\begin{dfn} For $(G,\sigma)\in
\sToneLt{\bullet}{\star_t}$, and $\sigma \in C(\bullet)_*$, we define an evaluation map $\sI$  \bas
\sI:\sToneL{}{} &\rightarrow \C \nonumber \\ (G,\sigma) &\rightarrow
 \int_{(G,\sigma)}\omega_{2\bullet - \star}  \;.\eas \label{evalr} \end{dfn}

This integral is only well defined if, for $\sigma \in C(\bullet)_m$, $m = 2 \bullet - \star$. However, since $ m \leq \bullet \leq 2 \bullet - \star$, this implies that $m = \bullet = \star$.

Explicitly, \bas \int_{(G,\sigma)}\omega_n = \int_{\Delta_n} \sigma_*(\omega_m) = \frac{1}{(2\pi i)^{m}}\int_{\Delta_m}\frac{d(1-\frac{1}{\sigma_1})}{1-\frac{1}{\sigma_1}} \wedge \ldots \wedge  \frac{d(1-\frac{1}{\sigma_m})}{1-\frac{1}{\sigma_m}}
\;, \eas where the ordering of the coordinates of $\sigma$ are given by the ordering of the edges of $G$.

We call $\sI (G,\sigma_m)$ the period associated to $(G,\sigma_m)$. The evaluation map induces a quasiisomorphism between $\BsTsG$ and $\BsG$: \ba \sI \otimes \id : \BsTsG & \rightarrow \C \otimes \BsG \nonumber \\ [(G_0, \sigma) | G_1 | \ldots | G_n ] & \rightarrow \sI (G_0, \sigma) [ G_1 | \ldots | G_n ] \;. \label{Iextend}\ea

Under this quasiisomorphism, the weight graded quotient $gr_{2r}^WJ = gr_{2r-1}^WJ$ is canonically isomorphic to $H^0(B(\sGoneL{}{}))(r)$.

\begin{rem}
The realization functor $\sI$ appears to depend on choices of simplices. However, in $H^(\BsTsG)$,it is well defined and independent of choice, as our complex is isomorphic to a subcomplex (via the equivalence with algebraic cycles) of the full realization map on the category of mixed Tate motives as defined in section 7 of \cite{BlochKriz}.
\end{rem}

\subsection{Hodge realization for necklace diagrams \label{necklacehodge}}

For the remainder of this paper, we study the Hodge realization of the
specific class $ \epclass \in H^0(\BsG)$. This is defined in Definition
\ref{epclass} by the sum of graphs
\bas \varepsilon^n (a_0, \bm{a_n}) =
{\begin{xy} (-10,5)
    *{\bullet}="x", (10,5) *{\bullet}= "y", (-10, -5) *{\bullet}= "z",
    (10, -5) *{\bullet}= "w", "y"; "x"
    **{\dir{-}}?/0pt/*{\dir{>}}+(0,2)*{a_0}, "x"; "z"
    **\crv{+(-5,5)}?/0pt/*{\dir{>}}+(-3.5,1)*{a_1}, "x"; "z"
    **\crv{+(5,5)}?/0pt/*{\dir{<}}+(3, -1)*{1}, "y"; "w"
    **\crv{+(-5,5)}?/0pt/*{\dir{>}}+(-2,1)*{1}, "y"; "w"
    **\crv{+(5,5)}?/0pt/*{\dir{<}}+(4, -1)*{a_n}, "z"; "w"
    **\crv{~**\dir{..}+(-10,-10)},
\end{xy}} - {\begin{xy}
(-10,5) *{\bullet}="x",
(10,5) *{\bullet}= "y",
(-10, -5) *{\bullet}= "z",
(10, -5) *{\bullet}= "w",
"y"; "x" **{\dir{-}}?/0pt/*{\dir{<}}+(0,3)*{\frac{1}{a_0}},
"x"; "z" **\crv{+(-5,5)}?/0pt/*{\dir{>}}+(-3.5,1)*{a_1},
"x"; "z" **\crv{+(5,5)}?/0pt/*{\dir{<}}+(3, -1)*{1},
"y"; "w" **\crv{+(-5,5)}?/0pt/*{\dir{>}}+(-2,1)*{1},
"y"; "w" **\crv{+(5,5)}?/0pt/*{\dir{<}}+(4, -1)*{a_n},
"z"; "w" **\crv{~**\dir{..}+(-10,-10)},
\end{xy}}\; \;.\eas  As always, $\bm{a_n}$ is the $n$-tuple $(a_1, \ldots , a_n)$ that labels the beads of the completely decomposable sum of necklace graphs. Section \ref{necklacehodge} calculates the period of $ [\bm{\varepsilon}^n(a_0, \bm{a_n})] \in H^0(\BsG)$ defined by this graph. In Section \ref{topobarelement}, we construct an element $[\bm{\xi}^n(a_0, \bm{a_n}) + 1 \otimes\bm{\varepsilon}^n(a_0, \bm{a_n})] \in H^0(\BsTsG)$ that defines the period. For ease
of notation, we drop the arguments $(a_0, \bm{a_n})$ whenever possible.

The current state of art for Hodge realization functor calculates the periods associated to elements of $H^0(B(A_{1L}))$ that can be represented by binary trees. See \cite{BlochKriz, Kimura} for cycles that map to classical polylogarithms, and \cite{GGL05, GGL07} for cycles that map to multiple polylogarithms. In this section, we compute the period associated to an algebraic cycle that is not in this small family of $\bP^1$ linear cycles.

\subsubsection{Corresponding element of $\BsTsG$\label{topobarelement}}

By Lemma \ref{epsiloncompdecomp}, the sum of graphs, $\varepsilon^n$, is completely decomposable. Therefore, by Lemma \ref{HBGelement}, the sum
  \bas \bm{\varepsilon}^n = \sum_{S
  \subset \bm{n} }(-1)^{|S|}\sum_{J \subseteq S} (-1)^{|J|}[\varepsilon^{n-|S|}(a_0\prod_{j \in J}a_j,  \bm{a}_{n\setminus S})| \sha_{s \in S} \varepsilon^0(a_s) ] \eas is a representative element
defining the class $[\epclass]\in H^0(\BsG)$.

In this section, we define an element $\bm{\xi^n} \in \bigoplus_{i=1}^{n+1}\BsTsG^i_{0}$ such that $\bm{\xi^n} + 1 \otimes \bm{\varepsilon^n}$ defines a class in $H^0(\BsTsG)$. Since $(\mu + \D) \bm{\varepsilon^n} = 0 $ in $\BsG$, we see that $(\mu + \D) 1 \otimes \bm{\varepsilon^n} = \bm{\varepsilon^n}$ seen as an element in $\bigoplus_{i=1}^{n}\BsTsG^i_{1}$. Here, as in example \ref{inclusioneg}, we write \bas \varepsilon^n = (\varepsilon^n, \sigma(a_0, \bm{a_n})_0) \in \sTt(n)^1 \;. \eas Therefore, it is sufficient to identify an element $ \bm{\xi^n}\in \bigoplus_{i=1}^{n+1}\BsTsG^i_{0}$ such that \ba (\D + \mu) \bm{\xi^n} = - \bm{\varepsilon^n} \;. \label{circcohomcond}\ea

The remainder of this section is devoted to identifying $\bm{\xi^n}$, which is a complicated sum of elements in the circular bar construction. We introduce it in stages, starting with the easiest to state, then breaking each sum into component pieces in order to demonstrate the appropriate properties. We state what criteria these summands need to satisfy, and provide proofs along the way.

Write $ \bm{\xi^n} = \sum_{k=0}^n (-1)^{k} \xi^{n-k}$ with $\xi^{n-k} \in \BsTsG^k_{0}$ defined as \bas \xi^{n-k} = \sum_{\begin{subarray}{c} S \subset \bm{n} \\ |S| = k\end{subarray}} \xi_{top}^{n-k}(a_0, \bm{a}_{\bm{n}\setminus S}) \otimes \sha_{i \in S} G_0(a_i) \;.\eas Here $\xi_{top}^{n-k}(a_0, \bm{a}_{\bm{n}\setminus S}) $ is a topologically augmented graph in $\sToneLt{n-k+1}{0}$ such that \bml (\D + \delta) \D\xi_{top}^{n-k}(a_0, \ldots, \hat{a}_S, \ldots  a_n) + \mu (\sum_{i \in \bm{n} \setminus S} \xi_{top}^{n-k-1}(a_0, \bm{a}_{\bm{n}\setminus S \cup i}) \otimes G_0(a_i)) \\ = - \varepsilon^{n-k} (a_0, \bm{a}_{\bm{n}\setminus S}) \;.\label{kthcohomcond}\eml This is the key condition that we prove explicitly in Theorem \ref{correctelement}.

In order to define $\xi_{top}^{n}$, we begin with a family of disconnected sums of unaugmented graphs \bas \xi^n_m (a_0, \bm{a_n})=  \varepsilon^{n-m}(a_0, \bm{a_{n-m}}) G_0(a_{n-m+1}) \ldots G_0(a_{n})\;.\eas  Each graph $\xi^n_m \in \sGoneL{n+1}{m+1}$ consists of $m+1$ connected components, with graphical degree $m+1$. We impose upon this family of graphs two topological augmentations $\sigma(a_0, \bm{a_n})_{m+1}$ and  $\rho(a_0, \bm{a_n})_{m+1} \in C(n+1)_{m+1}$ of the form  \bmls (\xi^n_m, \sigma(a_0, \bm{a_n})_{m+1}) = \left( {\begin{xy}
(-10,5) *{\bullet}="x",
(10,5) *{\bullet}= "y",
(-10, -5) *{\bullet}= "z",
(10, -5) *{\bullet}= "w",
"y"; "x" **{\dir{-}}?/0pt/*{\dir{>}}+(0,3)*{\frac{a_0}{t_{n-m+1}}},
"x"; "z" **\crv{+(-5,5)}?/0pt/*{\dir{>}}+(-3.5,0)*{a_1},
"x"; "z" **\crv{+(5,5)}?/0pt/*{\dir{<}}+(3, -1)*{1},
"y"; "w" **\crv{+(-5,5)}?/0pt/*{\dir{>}}+(-2,0)*{1},
"y"; "w" **\crv{+(5,5)}?/0pt/*{\dir{<}}+(10, -1)*{\frac{t_{n-m+1}a_{n-m}}{t_{n-m}}},
"z"; "w" **\crv{~**\dir{..}+(-10,-10)},
\end{xy}} - {\begin{xy}
(-10,5) *{\bullet}="x",
(10,5) *{\bullet}= "y",
(-10, -5) *{\bullet}= "z",
(10, -5) *{\bullet}= "w",
"y"; "x" **{\dir{-}}?/0pt/*{\dir{<}}+(0,4)*{\frac{t_{n-m+1}}{a_0 }},
"x"; "z" **\crv{+(-5,5)}?/0pt/*{\dir{>}}+(-3.5,0)*{a_1},
"x"; "z" **\crv{+(5,5)}?/0pt/*{\dir{<}}+(3, -1)*{1},
"y"; "w" **\crv{+(-5,5)}?/0pt/*{\dir{>}}+(-2,0)*{1},
"y"; "w" **\crv{+(5,5)}?/0pt/*{\dir{<}}+(10, -1)*{\frac{t_{n-m+1}a_{n-m}}{t_{n-m}}},
"z"; "w" **\crv{~**\dir{..}+(-10,-10)},
\end{xy}}\right)  \\ \Go{\frac{a_{n-m+1}t_{n-m+2}}{t_{n-m+1}}} \ldots \Go{\frac{a_n}{t_n} }\;,\emls and \bmls \left(\xi^n_m, \rho(a_0, \bm{a_n})_{m+1}) = ( {\begin{xy}
(-10,5) *{\bullet}="x",
(10,5) *{\bullet}= "y",
(-10, -5) *{\bullet}= "z",
(10, -5) *{\bullet}= "w",
"y"; "x" **{\dir{-}}?/0pt/*{\dir{>}}+(0,3)*{\frac{a_0}{t_{n-m+1}}},
"x"; "z" **\crv{+(-5,5)}?/0pt/*{\dir{>}}+(-3.5,0)*{a_1},
"x"; "z" **\crv{+(5,5)}?/0pt/*{\dir{<}}+(3, -1)*{1},
"y"; "w" **\crv{+(-5,5)}?/0pt/*{\dir{>}}+(-2,0)*{1},
"y"; "w" **\crv{+(5,5)}?/0pt/*{\dir{<}}+(6, -1)*{a_{n-m}},
"z"; "w" **\crv{~**\dir{..}+(-10,-10)},
\end{xy}} - {\begin{xy}
(-10,5) *{\bullet}="x",
(10,5) *{\bullet}= "y",
(-10, -5) *{\bullet}= "z",
(10, -5) *{\bullet}= "w",
"y"; "x" **{\dir{-}}?/0pt/*{\dir{<}}+(0,4)*{\frac{t_{n-m+1}}{a_0 }},
"x"; "z" **\crv{+(-5,5)}?/0pt/*{\dir{>}}+(-3.5,0)*{a_1},
"x"; "z" **\crv{+(5,5)}?/0pt/*{\dir{<}}+(3, -1)*{1},
"y"; "w" **\crv{+(-5,5)}?/0pt/*{\dir{>}}+(-2,0)*{1},
"y"; "w" **\crv{+(5,5)}?/0pt/*{\dir{<}}+(6, -1)*{a_{n-m}},
"z"; "w" **\crv{~**\dir{..}+(-10,-10)},
\end{xy}}\right)  \\ \Go{\frac{a_{n-m+1}t_{n-i+2}}{t_{n-m}}} \Go{\frac{a_{n-m+2}t_{n-m+3}}{t_{n-m+2}}} \ldots \Go{\frac{a_n}{t_n} }\;,\emls

Note that the only difference between the labeling $\sigma(a_0, \bm{a_n})_m$ and $\rho(a_0, \bm{a_n})_m$ is the label on the last bead of the first connected component, $\varepsilon^{n-m}(a_0, \bm{a_{n-m}})$, and that of the second connected component. This distinction is necessary for the appropriate cancellations between algebraic and topological differentials needed to satisfy condition \eqref{kthcohomcond}. Before writing down the expression for $\xi^{n-k}_{top}$, we introduce some further notation to simplify the expression.

\begin{comment}
By abuse of notation, we write these augmented graphs \bmls (\xi^n_m, \sigma(n+1)_{m+1}) (a_0, \ldots , a_n) = \varepsilon^{n-m} (\frac{a_0}{t_{n-m+1}}, a_1, \ldots , a_{n-m-1}, a_{n-m}\frac{t_{n-m+1}}{t_{n-m}}) \\ G_0(a_{n-m+1}\frac{t_{n-m+2}}{t_{n-m+1}}) \ldots G_0(\frac{a_n}{t_n})\;,\emls and \bmls (\xi^n_m, \rho(n+1)_{m+1}) (a_0, \ldots , a_n) = \varepsilon^{n-m} (\frac{a_0}{t_{n-m+1}}, a_1, \ldots , a_{n-m}) \\ G_0(a_{n-m+1}\frac{t_{n-m+2}}{t_{n-m}}) G_0(a_{n-m+2}\frac{t_{n-m+3}}{t_{n-m+2}}) \ldots G_0(\frac{a_n}{t_n})\;.\emls Define \bmls \xi^n^{top}(a_0, \ldots , a_n) = \sum_{m=0}^{n}\sum_{ J \subset \{n-m+1, \ldots n\}}(-1)^{|J|} (\xi^n_m, \sigma(n+1)_{m+1})(a_0\prod_{j \in J}a_j, \ldots , a_n) - \\ \sum_{m=1}^{n}\sum_{ I \subset \{n-m+2, \ldots n\}}(-1)^{|I|} (\xi^n_m, \rho(n+1)_{m+1})(a_0 \prod_{i \in I}a_i , \ldots , a_n)\;. \emls
\end{comment}

We define two new terms as sums of $\xi^n_m$ with variants of $\sigma$ and $\rho$: \ba \lambda^n_m(a_0, \bm{a_n}) = \sum_{ J \subset \{n-m+1, \ldots n\}}(-1)^{|J|} (\xi^n_m, \sigma(a_0\prod_{j \in J}a_j, \bm{a_n})_{m+1}) \label{lambdadef}\ea and \ba \chi^n_m(a_0,\bm{a_n}) = \sum_{ I \subset \{n-m+2, \ldots n\}}(-1)^{|I|} (\xi^n_m, \rho(a_0 \prod_{i \in I}a_i , \bm{a_n})_{m+1}) \label{chidef}\;.\ea Under this notation, we write \bas \xi^n_{top} = \sum_{m=0}^{n} \lambda^n_m - \sum_{m=1}^n \chi^n_m \;.\eas

Note that sum for $\chi^n_m$ starts at $m=1$ whilst the sum for $\lambda^n_m$ starts at $m=0$. Furthemore, the sets $I$ and $J$ differ. Namely, the first argument for $\rho$, augmenting $\chi^n_m$, never contains $a_{n-m+1}$, which this label appears in the first argument of $\sigma$ summands of $\lambda^n_m$. The terms $\lambda^n_m$, $\chi^n_m$ and $\chi^n_{top}$ are constructed such that the summands of the differentials of $\lambda^n_m$ cancel with terms in the differentials of $\chi^n_{m+1}$, and terms of the form $\sum_{i \in \bm{n}}\xi^{n-1}_{top}(a_0, \bm{a}_{\bm{n}\setminus i})$ leaving the term $\varepsilon^n$. This is how $\chi_{top}^{n-k}$ satisfies equation \eqref{kthcohomcond}. We show this cancellation explicitly in Theorem \ref{correctelement}.

The unaugmented graphs $\xi^n_m$ are in $\sGoneL{n+1}{m+1}$. Therefore the augmented graphs $\lambda^n_m$ and $\chi_n^m$ are in $\sTt(n+1)^{0}$. In  particular, $\xi^n_{top}$ is a sum of admissible augmented graphs. If $t_{n-m+k} = 0$, then $t_{n-m+i}=0$ for all $i <k$. Therefore, the edges labeled $\frac{a_0}{t_{n-m+1}}$ and $\frac{1}{a_0t_{n-m+1}}$ are labeled by $\infty$, making the graph $(\xi^n_m, \sigma(n+1)_{m+1})(a_0, \ldots a_n)$ trivial at this point.

It remains to check that $\bm{\xi^n}$ defined above satisfies the necessary conditions.

\begin{thm}
The element $\bm{\xi^n} + 1 \otimes \bm{\varepsilon^n} \in \bigoplus_{i=1}^{n+1}\BsTsG^i_{0}$ defines a class in $H^0(\BsTsG)$.
\label{correctelement}\end{thm}

\begin{proof}
By the arguments presented in this section, it is sufficient to check that $\xi_{top}^n$ satisfies equation \eqref{kthcohomcond}. It is enough to show this for $k=0$.

We proceed by computing the four terms of $(\delta + \D) (\lambda^n_m - \chi^n_m)$ to show that \bas (\delta + \D)\xi_{top}^n = -\varepsilon^n -  \mu \left(\sum_{i \in \bm{n}} \xi^{n-1}_{top}(a_0, \bm{a}_{\bm{n}\setminus i}) \otimes G_0(a_i) \right)\;, \eas as required.

When $m = 0$, the graph $(\xi^n_0, \sigma(a_0, \bm{a_n})_1) = \lambda^n_0$ is augmented by a $1$ simplex with topological boundary \bas \delta \lambda^n_0 = -\delta^1\lambda^n_0 =  -\varepsilon^n \;. \eas

For more general $m$, the algebraic boundary of the augmented sum of graphs $\lambda^n_m$ is \bml \D \lambda^n_m = -\mu \left( \sum_{i = 1}^{n-m-1} \lambda^{n-1}_{m}(a_0, \bm{a}_{\bm{n}\setminus i})   - \lambda^{n-1}_{m}(a_0a_i, \bm{a}_{\bm{n}\setminus i}) \otimes G_0(a_i)\right) \\ + \delta^2 \chi^n_{m+1} (a_0, \bm{a_n}) - \delta^1 \chi^n_{m+1} (a_0a_{n-m}, \bm{a_n})  \label{algderiv1} \;.\eml The algebraic boundary of the augmented sum of graphs $\chi^n_m$ is \ba - \D \chi^n_m = \mu \left( \sum_{i = 1}^{n-m} \chi^{n-1}_{m}(a_0,\bm{a}_{\bm{n}\setminus i})  - \chi^{n-1}_{m}(a_0a_i, \bm{a}_{\bm{n}\setminus i}) \otimes G_0(a_i) \right) \label{algderiv2} \;.\ea

For $m \geq 1$, the topological boundary of the augmented sum of graphs $\lambda^n_m$ is \bml\delta \lambda^n_m =  - \delta^1 \chi^n_{m} (a_0, \bm{a_n}) + \delta^1 \chi_n^{m} (a_0a_{n-m+1}, \bm{a_n})  \\ - \mu \left( \sum_{i = n-m+1}^{n} \lambda_{n-1}^{m-1}(a_0, \bm{a}_{\bm{n}\setminus i}) \otimes G_0(a_i) \right) \label{topoderiv1} \;.\eml The topological boundary of the augmented sum of graphs $\chi_n^m$ is \bml  - \delta \chi_n^m = \delta^1 \chi_n^{m} (a_0, \bm{a_n}) - \delta^2 \chi_n^{m+1} (a_0, \bm{a_n}) \\ +\mu \left( \sum_{i = n-m+1}^{n} \chi_{n-1}^{m-1}(a_0, \bm{a}_{\bm{n}\setminus i})  \otimes G_0(a_i) \right) \label{topoderiv2} \;. \eml

Adding up equations \eqref{algderiv1}, \eqref{algderiv2}, \eqref{topoderiv1} and \eqref{topoderiv2}, we see that \bas (\delta + \D) \xi^n_{top} = -\varepsilon_n -  \mu \left(\sum_{i \in \bm{n}} \xi^{n-1}_{top}(a_0, \bm{a}_{\bm{n}\setminus i}) \otimes G_0(a_i) \right)\;, \eas which matches equation \eqref{kthcohomcond}.
\end{proof}

\subsubsection{Integrals associated to necklace diagrams}

This section is devoted to calculating the period associate to $\bm{\varepsilon}_n$. We show that for this is $0$ for $n \geq 1$.

By abuse of notation, in this section we write the augmented graphs $\lambda^n_m$ and $\chi^n_m$ as \bas \lambda^n_m= \varepsilon^{n-m} (\frac{a_0}{t_{n-m+1}}, \bm{a_{n-m-1}}, a_{n-m}\frac{t_{n-m+1}}{t_{n-m}}) G_0(a_{n-m+1}\frac{t_{n-m+2}}{t_{n-m+1}}) \ldots G_0(\frac{a_n}{t_n})\;,\eas and \bas \chi^n_m = \varepsilon^{n-m} (\frac{a_0}{t_{n-m+1}}, \bm{a_{n-m}}) G_0(a_{n-m+1}\frac{t_{n-m+2}}{t_{n-m}}) G_0(a_{n-m+2}\frac{t_{n-m+3}}{t_{n-m+2}}) \ldots G_0(\frac{a_n}{t_n})\;.\eas

\begin{comment}
We apply the map $\sI \otimes \id$ from equation \ref{Iextend} to the element $\bm{\xi}_n + 1 \otimes\bm{\varepsilon}_n$. This integral is only well defined when $m$, the simplicial dimension of the augmented graph is equal to $n$, the loop number of the graph. Therefore, \bas \sI \bm{\xi}_n = \sum_{k=0}^n (-1)^k\sum_{\begin{subarray}{c} S \subset \bm{n} \\ |S|=k\end{subarray}} \sI \left(\lambda_{n-k}^{n-k}(a_0 \ldots \hat{a}_S \ldots a_n) - \chi_{n-k}^{n-k}(a_0 \ldots \hat{a}_S \ldots a_n)\right)[\sha_{s \in S} G_0(a_s)] \;.\eas Since $\sI(1) = 1$, the evaluation map $(\sI \otimes \id )(1 \otimes\bm{\varepsilon}_n) = \bm{\varepsilon}_n$. \end{comment}

\begin{thm} The period associated to $\bm{\xi^n} + 1 \otimes \bm{\varepsilon^n}$ is $0$ for all $n$. Therefore, $\epclass \in H^0(\BsG)$ defines a trivial cohomology class. \end{thm}

\begin{proof}
We apply the map $\sI \otimes \id$ from equation \ref{Iextend} to the element $\bm{\xi}^n + 1 \otimes\bm{\varepsilon}Tn$. This integral is only well defined when $m$, the simplicial dimension of the aumented graph is equal to $n$, the loop number of the graph. Therefore, \bas \sI \bm{\xi}^n = \sum_{k=0}^n (-1)^k\sum_{\begin{subarray}{c} S \subset \bm{n} \\ |S|=k\end{subarray}} \sI\otimes \id  \left(\lambda_{n-k}^{n-k}(a_0, \bm{a}_{\bm{n} \setminus S}) - \chi_{n-k}^{n-k}(a_0, \bm{a}_{\bm{n} \setminus S}) \right)[\sha_{s \in S} G_0(a_s)] \;.\eas Since $\sI(1) = 0$, the evaluation map $(\sI \otimes \id )(1 \otimes\bm{\varepsilon}^n) = 0$.

Recall that $\varepsilon_0(a_0) = G_0(a_0) - G_0(1/a_0)$. Therefore, from equations \eqref{lambdadef} and \eqref{chidef}, we have \bas \lambda_{n}^{n}(a_0, \bm{a_n}) = \sum_{J \subset \{1 \ldots n\}} (-1)^{|J|}\epsilon_0(a_0 \prod_{j \in J} a_j \frac{1}{t_0})G_0(a_1\frac{t_2}{t_1}) \ldots G_0(a_n\frac{1}{t_n})  \eas and \bas \chi_{n}^{n}(a_0, \bm{a_n}) = \sum_{I \subset \{2 \ldots n\}} (-1)^{|I|}\epsilon_0(a_0 \prod_{i \in I} a_i \frac{1}{t_1})G_0(a_1\frac{t_2}{t_0}) \ldots G_0(a_n\frac{1}{t_n})  \;.\eas Collecting like terms, we write \bml \lambda_{n}^{n} - \chi^n_n = \left(\sum_{J \subset \{1\ldots n\}} (-1)^{|J|}  \epsilon_0(a_0 \prod_{j \in J} a_j \frac{1}{t_0})G_0(a_1\frac{t_2}{t_1}) -  \sum_{I \subset \{2\ldots n\}} (-1)^{|I|}\epsilon_0(a_0 \prod_{i \in I} a_i \frac{1}{t_1})G_0(a_1\frac{t_2}{t_0}) \right) \\
G_0(a_2\frac{t_3}{t_2}) \ldots G_0(a_n\frac{1}{t_n}) \;. \label{Ixitopexpression}\eml

To evaluate this integral, we recall a few facts about the iterated integrals associated to multiple polylogarithms. First of all, for a constant cycle supported on a $1$ simplex, \bas \sI(G_0(a/t)) = \int_0^1 \frac{d (1 - \frac{t}{a})}{1 - \frac{t}{a}} = - \int_0^1 \frac{dt}{a-t} = - \Li(\frac{1}{a}) \;.\eas Inverting the label of the edge gives \bas \sI(G_0(t/a)) = \int_0^1 \frac{d (1 - \frac{a}{t})}{1 - \frac{a}{t}} = \int_0^1 \frac{adt}{t^2-at} =- \int_0^1 \frac{dt}{t} - \int_0^1 \frac{dt}{a-t} \;.\eas Subtracting the second expression from the first gives $\sI(\epsilon_0(a/t)) = \int_0^1 \frac{dt}{t}$.  We may write this as $\Li_1(1) = 0$, by standard renormalization of polylogarithms, \cite{Gonch01}.

Similarly, for the cycle supported on a two simplex, \bas \sI(\varepsilon^0(a/t_0) G_0(b/t_1)) =  - \int_0^1 \frac{1}{b-t_1} \left(\int_0^{t_1} \frac{dt_0 }{t_0}\right) dt_1  =   \Li_2(\frac{1}{b}) \;.\eas The last equality in this equation comes from the shuffle product on iterated integrals: \bas
\left(\int_0^z\frac{dt}{b-t}\right)\left(\int_0^z\frac{ds}{s} \right)= \int_0^z\frac{1}{b-t}\left(\int_0^t \frac{ds}{s}\right) dt + \int_0^z\frac{1}{s}\left(
\int_0^s\frac{dt}{b-t} \right)ds
\;.\eas By the standard regularization arguments above, the left hand side is $0$. Therefore,
\ba \int_0^z\frac{1}{t-b}\left(\int_0^t\frac{ds}{s}\right)dt =
\Li_2(\frac{1}{b}) \;. \label{renorm}\ea This does not depend on the first argument, $a$. Therefore, the alternating signs in the sums for $\lambda_1^1$ and $\chi_1^1$ force both $\sI(\lambda_1^1(a, b)) = \sI(\chi_1^1(a, b)) = 0$.

For cycles supported on a three simplex, there are two terms to check: \bas \sI(\lambda^2_2(a,b,c)) =  (-1)^2\int_0^1\frac{1}{c-t_2}  \left( \int_0^{t_2} \frac{1}{bt_2-t_1} \left( \int_0^{t_1}  \frac{dt_0}{t_0} \right) dt_1 \right) dt_2  =  \Li_{1}(\frac{1}{c})  \Li_2(\frac{1}{b}) \;, \eas and \bas \sI(\chi^2_2(a,b,c)) =  (-1)^2 \int_0^1\frac{1}{c-t_2}  \left( \int_0^{t_2}\frac{1}{t_1}  \left(\int_0^{t_1} \frac{dt_0}{bt_2-t_0}\right)dt_1 \right) dt_2 = \Li_{1}(\frac{1}{c}) \Li_2(\frac{1}{b}) \;. \eas As before, since neither integral depends on $a$, the alternating signs in the sums for $\lambda_2^2$ and $\chi_2^2$ force both $\sI(\lambda_2^2(a, b,c)) = \sI(\chi_2^2(a, b, c)) = 0$.

For a general $n+1$ simplex, we have \bas \sI(\xi^n_n, \sigma(a_0, \bm{a_n})_{n+1}) = (-1)^{n}\prod_{i=2}^n\Li_{1}(\frac{1}{a_i}) \Li_2(\frac{1}{a_1})  \;.\eas Similarly, \bas \sI(\xi^n_n, \rho(a_0, \bm{a_n})_{n+1}) = (-1)^{n}\prod_{i=2}^{n}\Li_{1}(\frac{1}{a_i}) \Li_2(\frac{1}{a_1})  \;.\eas Since neither of these expressions depend on $a_0$ we have that both $\sI(\lambda_n^n) (a_0, \bm{a_n})=0=\sI(\chi_n^n) (a_0, \bm{a_n})=0$. Therefore, $\sI(\xi^{top}_n) = 0$ for all $n$.  This is the period associated to $\bm{\varepsilon}^n$.

This implies that \bas \sI\otimes  \id (\bm{\xi}^n + 1 \otimes\bm{\varepsilon}^n) = 0 \;,\eas for all $n$. Therefore this defines a trivial class in $H^0(\BsG)$.  \end{proof}

\section{Outlook and future work}
This paper is a first step in a program to understand the cohomology of (part of) the Bloch Kriz cycle complex, and by extension to understand the motives associated to these cycles. Historically, there has been a long list and variety of tools used to tackle the problem of understanding mixed Tate motives, the algebra of multiple Zeta values, and the relations between such values. In this paper, by introducing a graphical representation of certain cycles, we pave the way for graph theoretic tools, such as matroids to be added to this list.

A few benchmarks of this program include:

\begin{enumerate}
\item As we show in Section \ref{examples}, there is a surprising fact, that various closely related cycles, or indeed the same algebraic cycle, can contribute as a summand of multiple different minimally decomposable sums. This immediately raises the question of how the different cohomology classes defined by these sums may be related. If both sums of graphs define the same cohomology class, then this gives insights about trivial classes in $H^0(\BsG)$. If not, it gives insight into how varying the cycle in certain well defined ways corresponds to changing classes in $H^0(\BsG)$. In the best case scenario, studying the relationships between these types of related sums of algebraic cycles may lead to more relations among multiple Zeta values.
\item One of the major shortcomings of this paper is our inability to include graphs with edges labeled in by $0$, i.e. precisely the graphs needed to correspond to the classical polylogarithms. There is an unwritten conjecture of Brown and Gangl that only the multiple logarithms are necessary to generate the entire space of multiple polylogarithms (including the standard polylogarithms). If one assumes this conjecture, then our inability to label our edges with $0$ is not a significant setback. However, we hope than in the near future, we will devise a way of encoding edges labeled by $0$s, possibly by including colored, unoriented edges, such that all the results of this paper hold in the new general setting.
\item The graphs we study lend themselves easily to study via the language of matroids. Roughly speaking, a matroid is a combinatorial way of encoding the independence data of a matrix or graph (in this case, the subtrees of a graph). While simple to define, this is a powerful tool when it comes to studying boundaries of geometric objects. We hope that this will lead to some insight for an algorithm for finding, or a classification of sums of algebraic cycles that lead to elements with completely decomoposable boundary.
\item The Hodge realization functor, while known to be well defined, is admitted to be difficult to compute. The algebra of graphs that we study is rich enough that there are many completely decomposable elements in it, while being simple enough to allow for a clear visual and combinatoric representation. We hope this this combination of simplicity and breadth of examples provides enough flexibility for one to be able to construct an algorithm for computing the Hodge realization functor, at least when restricted to $\Asm$, if not in full generality.
\end{enumerate}

{\bf{Acknowledgements}}

This research was partially supported by various ICMAT grants, the Heilbronn Institute for Mathematical Research, GRK 1670 and EPSRC grant EP/J019518/1.

The authors thank Tom McCourt, and Karen Yeats for many helpful comments and discussions.  The authors give an especially warm thank you to Ismael Souderes for many helpful discussions and comments along the way.  Without his help this paper wouldn't exist.
The authors also thank ICMAT and especially Kurusch Ebrahimi-Fard and Jose Burgos Gil for their hospitality in providing a place for us to meet, work on the project, and present our results to interested audiences. The second author also thanks Michael Collins for a helpful discussion and suggestion.

\bibliographystyle{plain}
%\bibliography{c:Documents/PowerbookG4/Research/Research\ Bristol\ 09/Tate\ Motives/Rational\ MTMs/MTM\ graphs/2016/Bibliography}
\bibliography{c:/users/S/Dropbox/bibliography/Bibliography}
%\bibliography{Bibliography}

\begin{comment}
%\renewcommand\refname{References}

\end{comment}

\hfill

Owen Patashnick

Bristol University and Kings College, London.

o.patashnick@bristol.ac.uk

Susama Agarwala

University of Nottingham

susama.agarwala@nottingham.ac.uk
\end{document}